\documentclass[a4paper,11pt,reqno]{amsart} 

\usepackage{comment}
\usepackage[utf8]{inputenc}
\usepackage{pifont}
\usepackage[english]{babel}
\usepackage{amsmath}
\usepackage{amsthm}
\usepackage{amsfonts,mathrsfs,relsize,bigints,stmaryrd}
\usepackage{amssymb}
\usepackage{anysize}
\usepackage{hyperref}
\usepackage{esint}
\usepackage{appendix}
\usepackage{mathtools}
\usepackage{graphicx}
\usepackage{tabularx}
\usepackage{xcolor}
\usepackage[normalem]{ulem}
\usepackage{enumitem}
\usepackage{xurl}

\makeatletter  \@addtoreset{equation}{section} \makeatother
\newtheorem{lem}{Lemma}[section]
\newtheorem{theo}{Theorem}[section]
\newtheorem{cor}{Corollary}[section]
\newtheorem{pro}{Proposition}[section]

\newtheorem{rem}{Remark}[section]

\DeclareMathOperator{\R}{\mathbb{R}}

\DeclareMathOperator{\C}{\mathbb{C}}
\DeclareMathOperator{\T}{\mathbb{T}}
\DeclareMathOperator{\D}{\mathbb{D}}

\allowdisplaybreaks 

\usepackage{subcaption}

\begin{document}

\title{Uniformly Rotating Euler Configurations \\ with Multiple Vorticity Holes}

\author[V. Baroncini]{Vittorio Baroncini}
\address{Departamento de An\'alisis Matem\'atico \& IMUS, Universidad de Sevilla, 41012 Seville, Spain}
\email{vbaroncini@us.es}

\author[J. C. Cantero]{Juan Carlos Cantero}
\address{Departament de Matem\`atiques, Facultat de Ci\`encies, Universitat Aut\`onoma de Barcelona, 08193 Cerdanyola del Vall\`es, Spain}
\email{juancarlos.cantero@uab.cat}

\author[C. Garc\'ia]{Claudia Garc\'ia}
\address{ Departamento de Matem\'atica Aplicada \& Research Unit ``Modeling Nature'' (MNat), Facultad de Ciencias, Universidad de Granada, 18071 Granada, Spain}
\email{claudiagarcia@ugr.es}

\author[Z. Hassainia]{Zineb Hassainia}
\address{Departamento de Matem\'atica Aplicada \& Research Unit ``Modeling Nature'' (MNat), Facultad de Ciencias, Universidad de Granada, 18071 Granada, Spain}
\email{zinebhassainia@ugr.es}

\author[J. Mateu]{Joan Mateu}
\address{Departament de Matem\`atiques, Facultat de Ci\`encies, Universitat Aut\`onoma de Barcelona, 08193 Cerdanyola del Vall\`es, Spain}
\email{joan.mateu@uab.cat}

\begin{abstract}
We construct new families of uniformly rotating vortex-patch solutions of the two-dimensional incompressible Euler equations consisting of a simply connected outer patch and multiple interior interfaces, which can be interpreted geometrically as holes.
More precisely, each solution consists of a single outer vortex patch
enclosing $\mathbf m\geq2$ identical, highly concentrated inner components
arranged at the vertices of a regular $\mathbf m$-gon; the entire
configuration rotates rigidly in the clockwise direction.
 As the concentration parameter tends to zero, 
the inner components shrink and collapse simultaneously toward the origin,
while the outer boundary converges to the unit circle. The corresponding
vorticities converge, in the sense of measures, to a Rankine vortex
supplemented by a point vortex of circulation $-\mathbf m$ at its center.

The proof is based on a contour-dynamics formulation, a symmetry reduction
to two nonlinear boundary equations, and a suitable singular rescaling. We
then apply an implicit function theorem with a  continuous parameter
in symmetry-adapted Hölder spaces. To the best of our knowledge, this is the
first analytical construction of a desingularization regime in which several
concentrated inner components are contained in a common outer patch and
collapse simultaneously toward its center.
\end{abstract}

\maketitle

\section{Introduction}

The evolution of a two-dimensional inviscid and incompressible fluid is described by the Euler equations
\begin{equation} \label{ES}
    \begin{cases}
        \partial_t u + \left(u\cdot \nabla \right)u= -\nabla p, & \mbox{in } [0, \infty) \times \R^2, \\
        \operatorname{div}u = 0, & \mbox{in } [0, \infty) \times \R^2, \\
        u(0, \cdot) = u_0, & \mbox{in } \R^2,
    \end{cases}
\end{equation}
where the unknowns are the velocity vector field of the fluid $u = u(t, x) = \left(u_1(t, x), u_2(t, x)\right)$ and its scalar pressure $p = p(t, x)$, while $u_0 : \R^2 \to \R^2$ is the prescribed initial velocity field. Here, we adopt standard conventions
$$u \cdot \nabla := u_1 \partial_{x_1} + u_2 \partial_{x_2}, \qquad \operatorname{div}u := \partial_{x_1} u_1 + \partial_{x_2} u_2.$$
The vorticity is defined by
$$\omega := \nabla^{\perp} \cdot u = \partial_{x_1} u_2 - \partial_{x_2} u_1, \qquad \nabla^{\perp} := \left(-\partial_{x_2}, \partial_{x_1}\right),$$
and measures the local tendency of the fluid to rotate. Taking the curl of the momentum equation in \eqref{ES} eliminates the
pressure and yields the transport equation
$$
    \partial_t\omega+u\cdot\nabla\omega=0.
$$
Since $u$ is divergence free, it can be represented in terms of a stream
function $\psi$ through
$$
    u=\nabla^\perp\psi.    
$$
Thus, the connection between the stream function $\psi$ and the vorticity is given by the Poisson equation
$$
    \Delta\psi=\omega,
$$
and therefore
\begin{equation}
\label{steam_vorticity}
    \psi(t,x)
    =
    \frac{1}{2\pi}
    \int_{\mathbb R^2}
    \log|x-y|\,\omega(t,y)\,dy.
\end{equation}
Consequently, the velocity is recovered from the vorticity through the
Biot--Savart law
\begin{equation}
\label{Biot-Savart_Law}
    u(t,x)
    =
    \bigl(K*\omega\bigr)(t,x),
    \qquad
    K(x)
    :=
    \frac{1}{2\pi}\frac{x^\perp}{|x|^2},
    \qquad
    x^\perp:=(-x_2,x_1).
\end{equation}
Thus, the Euler equations can be written in vorticity form as
\begin{equation}
\label{ESVF}
\begin{cases}
    \partial_t\omega+u\cdot\nabla\omega=0,
    & (t,x)\in[0,\infty)\times\mathbb R^2,
    \\[1mm]
    u=K*\omega,
    & (t,x)\in[0,\infty)\times\mathbb R^2,
    \\[1mm]
    \omega(0,x)=\omega_0(x),
    & x\in\mathbb R^2.
\end{cases}
\end{equation}

\subsection{Vortex patches and rotating states}
Yudovich's classical theory \cite{Y63} establishes the global existence and
uniqueness of weak solutions of \eqref{ESVF} for initial data
$$
    \omega_0\in L^1(\mathbb R^2)\cap L^\infty(\mathbb R^2).
$$
Moreover, the vorticity is transported along the associated Lagrangian flow
$\Phi$, defined by
$$
    \Phi(t,x)
    =
    x+
    \int_0^t
    u\bigl(s,\Phi(s,x)\bigr)\,ds.
$$
More precisely,
$$
    \omega\bigl(t,\Phi(t,x)\bigr)=\omega_0(x).
$$

A fundamental class of solutions within this framework is provided by
vortex patches, for which the initial vorticity is the characteristic
function of a bounded domain $D_0 \subset \R^2$:
$$
    \omega_0=\mathbf 1_{D_0}.
$$
The transport property implies that
$$
    \omega(t,\cdot)=\mathbf 1_{D_t},
    \qquad
    D_t=\Phi(t,D_0).
$$
Hence, the patch dynamics may be described in terms of the evolution of the
boundary $\partial D_t$. If
$$
    z(t,\cdot):\mathbb T\longrightarrow\partial D_t
$$
is a regular parametrization, then its normal velocity must coincide with
the normal component of the fluid velocity:
\begin{equation*}
\label{CDE}
    \left[
        \partial_tz(t,\vartheta)
        -
        u\bigl(t,z(t,\vartheta)\bigr)
    \right]
    \cdot
    \vec{n}_{\partial D_t}\bigl(z(t,\vartheta)\bigr)
    =
    0,
    \qquad
    (t,\vartheta)\in[0,\infty)\times\mathbb T.
\end{equation*}
This is the contour-dynamics formulation of the vortex-patch problem.

Radially symmetric patches are stationary as sets, although the individual
fluid particles may rotate. The most familiar examples are the Rankine
vortex
$$
    D_0=B_R(0), \qquad {R > 0}
$$
and the annular patch
$$
    D_0
    =
    \left\{
        x\in\mathbb R^2:
        r<|x|<R
    \right\},
    \qquad
    0<r<R.
$$
A classical nonradial example is Kirchhoff's ellipse \cite{K74}, which
preserves its shape while rotating uniformly.

A bounded domain $D_0\subset\mathbb R^2$ is called a rotating vortex patch,
or a \emph{V-state}, if there exists an angular velocity $\Omega\in\mathbb R$ such that
$$
    D_t=e^{i\Omega t}D_0,
    \qquad t\geq0.
$$
Equivalently,
$$
    \omega(t,x)
    =
    \mathbf 1_{D_0}
    \left(
        e^{-i\Omega t}x
    \right).
$$
More generally, we shall use the term rotating vortex-patch configuration
for compactly supported, piecewise constant vorticities whose interfaces
rotate rigidly about the origin.

Motivated by the numerical observations of Deem and Zabusky \cite{DZ78},
Burbea \cite{B82} used local bifurcation theory to construct noncircular,
simply connected V-states. For each integer $\mathbf m\geq2$, he obtained a
local branch of $\mathbf m$-fold symmetric patches bifurcating from the unit
disk at
$$
    \Omega_{\mathbf m}
    =
    \frac{\mathbf m-1}{2\mathbf m}.
$$
Hmidi, Mateu and Verdera \cite{HMV13} subsequently refined this construction
and proved, in particular, the smoothness of the boundaries along the
bifurcating branches. Castro, C\'ordoba and G\'omez-Serrano
\cite{CCGS16} later established analyticity of these local V-state
boundaries and constructed analytic rotating patches bifurcating from
Kirchhoff ellipses. Global continua of simply connected V-states
bifurcating from the disk were constructed by Hassainia, Masmoudi and
Wheeler \cite{HMW20}, who also proved that the boundary of any sufficiently
regular rotating vortex patch is analytic.

The first systematic constructions in the multiply connected setting
concern doubly connected V-states. In particular, de la Hoz, Hmidi, Mateu
and Verdera \cite{HHMV16} constructed branches of nonannular,
$\mathbf m$-fold symmetric doubly connected V-states bifurcating from
annuli under an appropriate nondegeneracy condition. The corresponding
degenerate bifurcation regimes were investigated in
\cite{HM16,WXZ22}.

The theory of rotating vortex structures has since developed in several
directions. We mention, among others, the construction of rotating patches
bifurcating from Kirchhoff ellipses \cite{CCGS16,HmidiMateu2016},
rotating patches in the unit disk \cite{deLaHozHassainiaHmidiMateu2016},
nonradial compactly supported smooth rotating solutions of the Euler
equations \cite{CastroCordobaGomezSerrano2019}, and rotating vortices with
nonuniform vorticity distributions
\cite{GarciaHmidiSoler2020}. Rigidity and radial-symmetry properties of
stationary and uniformly rotating solutions have also been studied in
\cite{Hmidi2015,GomezSerranoParkShiYao2021}.

\subsection{Point vortices and desingularization}

Vortex-patch dynamics are closely connected with the point-vortex model,
which represents the vorticity by a finite collection of Dirac masses:
$$
    \omega(t,x)
    =
    \sum_{j=1}^N
    \Gamma_j\delta_{z_j(t)}(x).
$$
Here $\Gamma_j\in\mathbb R$ is the circulation of the $j$-th vortex and
$z_j(t)\in\mathbb R^2$ is its position. Since the vortex positions
remain distinct, they evolve according to the Kirchhoff--Routh system
$$
    \dot z_j(t)
    =
    \sum_{\substack{k=1\\k\neq j}}^N
    \Gamma_k
    K\bigl(z_j(t)-z_k(t)\bigr),
    \qquad
    j\in\{1,\ldots,N\},
$$
where $K$ is the Biot-Savart kernel defined in \eqref{Biot-Savart_Law}. The point-vortex system is well defined as long as the vortex positions remain
pairwise distinct, that is,
$$
    z_j(t)\neq z_k(t)
    \qquad\text{for }j\neq k.
$$
This separation property need not be preserved by the evolution. Indeed,
finite-time collapse may occur for systems of three or more vortices, as
already observed in Gröbli's classical analysis of the three-vortex problem
\cite{Grobli1877}; the corresponding collapsing configurations involve
circulations of different signs. Nevertheless, such singular trajectories
are exceptional. More precisely, under the non-neutral cluster condition
$$
    \sum_{j\in J}\Gamma_j\neq0
    \qquad
    \text{for every nonempty subset }
    J\subset\{1,\ldots,N\},
$$
Marchioro and Pulvirenti proved that the point-vortex system admits a global
solution for Lebesgue-almost every initial configuration
\cite[Chapter~4]{MarchioroPulvirenti1994}. Equivalently, the set of initial
configurations leading to a finite-time collision has Lebesgue measure zero. 

Of particular relevance to the present work are regular vortex polygons.
Consider $\mathbf m$ identical point vortices of circulation $\Gamma$ placed at
\begin{equation}\label{polyg}
      z_j
    =
    d\,
    e^{i\frac{\pi\sigma}{\mathbf m}}
    e^{i\frac{2\pi(j-1)}{\mathbf m}},
    \qquad
    j\in\{1,\ldots,\mathbf m\},
\end{equation}
where $d>0$ and $\sigma\in\mathbb R$. This configuration is a relative
equilibrium and rotates with angular velocity
$$
    \Omega_{\mathrm{pv}}
    =
    \Gamma\frac{\mathbf m-1}{4\pi d^2};
$$
see, for instance, \cite{G21,T85}. The phase $\sigma$ determines the global
orientation of the polygon but does not affect its angular velocity.

A central problem in vortex dynamics is to replace singular point vortices
by smooth or patch-type vorticity distributions while preserving the
underlying motion. This procedure is usually referred to as
\emph{desingularization}. In the context of relative equilibria, one seeks
families of concentrated vortex patches converging to a prescribed
point-vortex configuration while rotating with an asymptotically compatible
angular velocity.

Early numerical investigations of finite-area vortex equilibria for the
planar Euler equations go back to Deem and Zabusky
\cite{DZ78}. Steadily translating symmetric vortex pairs were
subsequently computed in detail by Pierrehumbert
\cite{Pierrehumbert1980}, while Saffman and Szeto
\cite{SaffmanSzeto1980} studied equilibrium shapes of equal co-rotating
vortices. Numerical investigations of more general, and in particular
asymmetric, vortex interactions were later developed by Dritschel
\cite{Dritschel1995}.

Rigorous constructions were first obtained through variational methods.
Turkington \cite{T85} constructed co-rotating steady vortex
patches with $N$-fold symmetry, which may be viewed as desingularizations
of regular point-vortex polygons, described in \eqref{polyg}. Keady \cite{Keady1985} established
existence and asymptotic estimates for symmetric translating vortex
streets, while Wan \cite{Wan1988} studied the existence and stability of
desingularizations associated with more general rotating point-vortex
systems. Although these variational methods are powerful, they generally
provide less detailed information about the geometry of the patch
boundaries and the local uniqueness of the resulting solutions.

A more direct approach was introduced by Hmidi and Mateu
\cite{HmidiMateu2017}. By desingularizing the contour-dynamics equations
and applying the implicit function theorem, they constructed local families
of co-rotating and counter-rotating vortex pairs bifurcating from the
corresponding point-vortex configurations. This strategy was subsequently
adapted to unequal and asymmetric vortex pairs by Hassainia and Hmidi
\cite{HassainiaHmidi2021}, to the von K\'arm\'an vortex street by
Garc\'ia \cite{GarciaKarman2020}, and to the regular Thomson polygon by
Garc\'ia \cite{G21}. Hassainia and Wheeler
\cite{HassainiaWheeler2022} extended the contour-dynamics
desingularization method to arbitrary steady configurations of finitely many
point vortices satisfying a natural nondegeneracy condition. Their framework
covers uniformly rotating, uniformly translating, and stationary
configurations and, after incorporating the relevant symmetries, applies to
regular Thomson polygons \eqref{polyg}, body-centered polygons, and nested regular
polygons.

In a complementary direction, D\'avila, del Pino, Musso and Wei
\cite{DavilaDelPinoMussoWei2020} developed a gluing method for constructing
smooth Euler solutions whose vorticity remains concentrated near prescribed
collision-free point-vortex trajectories. G\'omez-Serrano, Park and Shi
\cite{GomezSerranoParkShi2025} used a Nash--Moser scheme to construct
nontrivial, compactly supported stationary patch configurations with
sign-changing vorticity and finite kinetic energy.

\subsection{Main result}
The purpose of this paper is to construct uniformly rotating patch-solutions with multiple holes for the two-dimensional incompressible Euler equations in vorticity form \eqref{ESVF}.  More precisely, we establish the existence of piecewise-constant vortex configurations consisting of one outer component
and $\mathbf m\geq2$ concentrated inner components. For
$\varepsilon>0$, we consider vorticities of the form
\begin{equation}
\label{intro_initial_vorticity}
    \omega_{0,\varepsilon}^{\sigma}(x)
    =
    \mathbf 1_{\mathcal D_1^{\varepsilon}}(x)
    -
    \frac{1}{\pi\varepsilon^2}
    \sum_{j=1}^{\mathbf m}
    \mathbf 1_{\mathcal D_{2,j}^{\varepsilon,\sigma}}(x).
\end{equation}
The inner components are defined from a reference domain
$\mathcal D_2^{\varepsilon,\sigma}$ by
\begin{equation}
\label{intro_inner_components}
    \mathcal D_{2,j}^{\varepsilon, \sigma}
    =
    e^{i\frac{\pi\sigma}{\mathbf m}}
    e^{i\frac{2\pi(j-1)}{\mathbf m}}
    \left(
        \varepsilon\mathcal D_2^{\varepsilon,\sigma}
        +
        \varepsilon^\alpha\mathtt e_1
    \right),
    \qquad
    \mathtt e_1:=(1,0).
\end{equation}
They are pairwise disjoint, lie strictly inside
$\mathcal D_1^{\varepsilon}$, and are arranged with
$\mathbf m$-fold rotational symmetry. We shall refer to these inner
components as holes, in the sense that they determine inner interfaces of
the configuration. Notice, however, that the vorticity in
\eqref{intro_initial_vorticity} is not zero on them; the construction is a
sign-changing, two-level vortex-patch configuration. We restrict to the two phase choices
$$
    \sigma=0
    \qquad\text{and}\qquad
    \sigma=1.
$$
These are, up to cyclic relabeling, the two orientations compatible with
reflection symmetry about the real axis. The first corresponds to the
roots-of-unity configuration, whereas the second is obtained by rotating
the polygon through the angle $\pi/\mathbf m$.

The construction is performed in a desingularization regime. Each inner
component has diameter of order $\varepsilon$, whereas its center lies at
distance $\varepsilon^\alpha$ from the origin. Since $\alpha<1$, the inner
components are much smaller than the distance separating their centers.
As $\varepsilon\to0^+$, the outer component converges to the unit disk
$\mathbb D$, while all the inner components shrink and collapse toward the
origin. In the sense of measures,
\begin{equation}
\label{limiting_configuration}
    \omega_{0,\varepsilon}^\sigma
    \rightharpoonup
    \omega_{0,0}
    :=
    \mathbf 1_{\mathbb D}
    -
    \mathbf m\delta_0.
\end{equation}

Thus, the limiting measure consists of a Rankine vortex and a point vortex
of total circulation $-\mathbf m$ located at its center. The leading
angular velocity, however, is inherited from the regular
$\mathbf m$-polygon at the intermediate scale $\varepsilon^\alpha$.
Indeed, the collapsed measure in \eqref{limiting_configuration} is radially
symmetric and, by itself, does not select an angular velocity.

To the best of our knowledge, the present work gives the first construction
of uniformly rotating configurations of the form
\eqref{intro_initial_vorticity} with more than one concentrated inner
component. In contrast with the classical doubly connected constructions,
all the inner components collapse toward the same point in the singular
limit.

Our main result reads as follows.

\begin{theo}
\label{main_theo}
Let $\mathbf m\geq2$, let $\nu\in(0,1)$, and assume that
$$
    \frac{1}{\mathbf m+2}
    <
    \alpha
    <
    \frac12.
$$
For every $\sigma\in\{0,1\}$, there exists $\varepsilon_0>0$ such that, for
each $\varepsilon\in(0,\varepsilon_0)$, there exist a real number
$\Lambda_\varepsilon$ and bounded, simply connected domains $
    \mathcal D_1^{\varepsilon}$, $
    \mathcal D_2^{\varepsilon,\sigma},
$
with boundaries of class $C^{1,\nu}$, having the following properties:

\begin{enumerate}
    \item the outer domain $\mathcal D_1^{\varepsilon}$ is
    $\mathbf m$-fold rotationally symmetric and symmetric with respect to
    the real axis;
\item the domains $\mathcal D_{2,j}^{\varepsilon,\sigma}$ defined by
    \eqref{intro_inner_components} 
are pairwise disjoint and satisfy $
        \overline{\mathcal D_{2,j}^{\varepsilon,\sigma}}
        \subset
        \mathcal D_1^{\varepsilon};$

    \item as $\varepsilon\to0^+$,
the two generating domains converge to the unit disk in the
    $C^{1,\nu}$ topology and $\Lambda_\varepsilon\longrightarrow0
    $;

    \item the vorticity $\omega_{0,\varepsilon}$ defined by
    \eqref{intro_initial_vorticity} generates the uniformly rotating weak
    solution
    $$
        \omega_\varepsilon(t,x)
        =
        \omega_{0,\varepsilon}^{\sigma}
        \left(
            e^{-i\Omega_\varepsilon t}x
        \right),
    $$
    of \eqref{ESVF}, with angular velocity
    $$
        \Omega_\varepsilon
        =
        \frac{1-\mathbf m}{
            4\pi\varepsilon^{2\alpha}
        }
        +
        \frac{
            \Lambda_\varepsilon
        }{
            \varepsilon^\alpha
        }.
    $$
\end{enumerate}
\end{theo}

The numerical configurations displayed in Figures \ref{k3} and \ref{k32} 
illustrate the geometry of the solutions for $\mathbf m=3$ with $\sigma=1$ and $\sigma=0$, respectively. Figure  \ref{k6} illustrate the geometry of the solutions for
$\mathbf m=6$.

A configuration geometrically closer to the one considered in Theorem \ref{main_theo} was studied by Velasco Fuentes \cite{VelascoFuentes2013}, who numerically evolved a Rankine vortex containing two symmetric vorticity holes. In an appropriate parameter regime, the holes remain coherent and rotate around one another, although their shapes and mutual distance undergo small oscillations, so the resulting motion is not an exact relative equilibrium. 

These numerical studies provide evidence for the existence and persistence of rotating  vortex configurations with holes. Nevertheless, they do not address the singular regime considered in the present paper, in which $\mathbf m$ small inner components are contained in a single outer patch,
are centered at distance $\varepsilon^\alpha$ from the origin, and collapse
simultaneously toward the center as $\varepsilon\to0^+$.

\begin{rem}[On the admissible range of $\alpha$]
\label{rem:bounds_alpha}
The restrictions on $\alpha$ are analytical and do not arise from the
geometric separation of the inner components. Indeed, their diameters are of
order $\varepsilon$, whereas the distances between their centers are of order
$\varepsilon^\alpha$. Hence
$$
    \frac{\varepsilon}{\varepsilon^\alpha}
    =
    \varepsilon^{1-\alpha}
    \longrightarrow0
    \qquad\text{as }\varepsilon\to0^+
$$
whenever $\alpha<1$. Thus, throughout the range considered in Theorem
\ref{main_theo}, the inner components remain mutually separated for
sufficiently small $\varepsilon$.

The lower and upper bounds on $\alpha$ arise instead from the limiting
behavior of the rescaled boundary equations at the circular configuration. For more details, see Section \ref{section_trivial_solutions}.
\end{rem}

\begin{rem}[The exceptional pentagonal case]
\label{rem:pentagonal_case}
When $\mathbf m=5$, the admissible upper bound on $\alpha$ can be improved
from $\frac12$ to $\frac23$. This exceptional range is due to an exact
roots-of-unity cancellation in the second component of the nonlinear
equation at the circular configuration: for a regular pentagon, the
coefficient of the second Fourier harmonic in the inner--inner interaction
vanishes. Consequently, the first nonzero remainder is of order
$$
    \varepsilon^{2-3\alpha},
$$
rather than $\varepsilon^{1-2\alpha}$, and therefore tends to zero whenever
$\alpha<\frac23$. This cancellation depends only on the relative positions
of the inner components and is thus independent of the phase $\sigma$. More details can be found in Proposition \ref{F2_at_0}.
\end{rem}

\begin{rem}[Possible continuation toward genuine holes]
Since the inner components are pairwise disjoint and contained in the outer
domain, the vorticity constructed in Theorem \ref{main_theo} takes the values
$$
\omega_{0,\varepsilon}^{\sigma}(x)
=
\begin{cases}
1,
&
x\in
\mathcal D_1^{\varepsilon}
\setminus
\displaystyle\bigcup_{j=1}^{\mathbf m}
\mathcal D_{2,j}^{\varepsilon,\sigma},
\\[3mm]
\displaystyle
1-\frac{1}{\pi\varepsilon^2},
&
x\in
\displaystyle\bigcup_{j=1}^{\mathbf m}
\mathcal D_{2,j}^{\varepsilon,\sigma},
\\[3mm]
0,
&
x\notin\mathcal D_1^{\varepsilon}.
\end{cases}
$$
Thus, for the small values of $\varepsilon$ considered in the local
construction, the inner components are regions of negative vorticity rather
than genuine zero-vorticity holes.

Nevertheless, this observation suggests a possible connection with
multiply connected V-states. Indeed, at the distinguished value $
    \varepsilon_{\mathrm h}
    :=
    \frac{1}{\sqrt{\pi}},
$
one has
$
    1-\frac{1}{\pi\varepsilon_{\mathrm h}^2}=0,
$
and hence
$$
\begin{aligned}
\omega_{0,\varepsilon_{\mathrm h}}^\sigma
&=
\mathbf 1_{\mathcal D_1^{\varepsilon_{\mathrm h},\sigma}}
-
\sum_{j=1}^{\mathbf m}
\mathbf 1_{\mathcal D_{2,j}^{\varepsilon_{\mathrm h},\sigma}}
\\
&=
\mathbf 1_{
\mathcal D_1^{\varepsilon_{\mathrm h},\sigma}
\setminus
\displaystyle\bigcup_{j=1}^{\mathbf m}
\mathcal D_{2,j}^{\varepsilon_{\mathrm h},\sigma}
}.
\end{aligned}
$$
Consequently, if the local solution branch furnished by Theorem
\ref{main_theo} could be continued globally in the solution set up to
$\varepsilon=\varepsilon_{\mathrm h}$, in the spirit of \cite{GarciaHaziot2023}, while preserving the regularity and
admissibility of the configuration  (in particular, without collision of the
inner components, contact with the outer boundary, or degeneration of the
interfaces), then the continuation would produce a uniformly rotating vortex
patch with $\mathbf m$ genuine holes. Establishing such a global continuation
is an interesting problem that lies beyond the scope of the present work.
\end{rem}
\begin{rem}[Fast rotation and degeneration of the branch]
\label{rem:fast_rotation_degenerate_limit}
The angular velocity of the solutions constructed in Theorem
\ref{main_theo} satisfies 
$$
    \Omega_\varepsilon
    \sim
    \frac{1-\mathbf m}{
        4\pi\varepsilon^{2\alpha}
    },
    \qquad
    \varepsilon\to0^+,
$$
and therefore
$$
    \left|
        \Omega_\varepsilon
    \right|
    \longrightarrow+\infty.
$$
Thus, the nonradial configurations constructed above rotate increasingly
fast as the inner components shrink and collapse toward the origin.
This behavior is not inconsistent with the fact that the limiting
configuration
$$
    \omega_{0,0}
    =
    \mathbf 1_{\mathbb D}
    -
    \mathbf m\delta_0,
$$
is stationary. Indeed, this limiting vorticity is radially symmetric and
therefore invariant under every rotation:
$$
    \omega_{0,0}
    \left(
        e^{-i\theta}x
    \right)
    =
    \omega_{0,0}(x),
    \qquad
    \theta\in\mathbb R.
$$
Hence its angular velocity is not geometrically observable: the same
limiting configuration may formally be regarded as rotating with any
angular velocity, while remaining stationary as a vorticity distribution.

The divergence of $\Omega_\varepsilon$ reflects the degeneracy of
the nontrivial solution branch as $\varepsilon\to0^+$. Before reaching the
limit, the polygon formed by the inner components possesses a well-defined
orientation and rotates on the increasingly short time scale
$
    \left|
        \Omega_\varepsilon
    \right|^{-1}
    \sim
    \varepsilon^{2\alpha}.
$
At $\varepsilon=0$, however, all the inner components have collapsed to the
origin and the outer component has become circular. The orientation of the
configuration is then lost, and the nontrivial rotating branch collapses
onto a rotationally invariant state. In this sense, the branch disappears
from the limiting geometric problem, even though its vorticities converge
to the stationary measure $\omega_{0,0}$.

A related degeneracy occurs in the construction of doubly connected
V-states bifurcating from the {renormalized} annulus
$$
    A_b
    :=
    \left\{
        z\in\mathbb C:
        b<|z|<1
    \right\},
    \qquad
    0<b<1,
$$
obtained in \cite{HHMV16}. For each admissible Fourier mode $\mathbf m$,
there are two bifurcation values
$$
    \Omega_{\mathbf m}^{\pm}(b)
   :=
    \frac{1-b^2}{4}
    \pm
    \frac{1}{2\mathbf m}
    \sqrt{
        \left(
            \frac{\mathbf m(1-b^2)}{2}-1
        \right)^2
        -
        b^{2\mathbf m}
    }.
$$
As $b\to0^+$, these values satisfy
$$
    \Omega_{\mathbf m}^{+}(b)
    \longrightarrow
    \frac{\mathbf m-1}{2\mathbf m},
    \qquad
    \Omega_{\mathbf m}^{-}(b)
    \longrightarrow
    \frac{1}{2\mathbf m}.
$$
The first limit is precisely the bifurcation value of the classical
simply connected $\mathbf m$-fold V-states, whereas the second one does not
produce a nontrivial branch in the limiting simply connected problem.

This provides a useful analogy with the present construction: a family of
nontrivial rotating configurations may degenerate into a rotationally
invariant limiting state, and the corresponding rotation rate need not
have a meaningful finite limit.
\end{rem}

\subsection{Outline of the proof}

The proof is based on a reduction of the contour-dynamics equations to a
nonlinear operator equation in the spirit of \cite{HmidiMateu2017}. We parametrize the outer boundary and one
reference inner boundary as radial graphs and exploit the
$\mathbf m$-fold symmetry to reduce the full system of
$\mathbf m+1$ boundary equations to two scalar equations. This yields a
phase-dependent nonlinear operator
$$
    \mathcal F^\sigma(\varepsilon,\Lambda,f)=0,
    \qquad
    f=(f_1,f_2),
$$
acting between symmetry-adapted H\"older spaces.
The angular velocity is written as
$$
    \Omega_{\varepsilon,\Lambda}
    =
    \frac{\Omega_0}{\varepsilon^{2\alpha}}
    +
    \frac{\Lambda}{\varepsilon^\alpha},
    \qquad
    \Omega_0
    =
    \frac{1-\mathbf m}{4\pi}.
$$
The value of $\Omega_0$ is determined by canceling the leading
singular term in the inner-boundary equation and agrees with the angular
velocity coefficient of the regular polygon of $\mathbf m$ point vortices.
We then show that $\mathcal F^\sigma$ admits a continuous extension to the
one-sided parameter value $\varepsilon=0$ and that
$$
    \mathcal F^\sigma(0,0,0)=0.
$$
Since fractional powers of $\varepsilon$ occur in the formulation, we treat
$\varepsilon$ as a continuous parameter belonging to the metric space
$[0,\varepsilon_0)$ and apply the parameter-dependent implicit function
theorem of Dontchev and Rockafellar
\cite[Theorem~5F.4, p.~305]{DontchevRockafellar2014}.
The linearization with respect to the unknowns $(\Lambda,f)$ is explicitly
given by
$$
\begin{aligned}
D_{(\Lambda,f)}
\mathcal F^\sigma(0,0,0)[\lambda,h]
=
\Bigg(
    &\Omega_0h_1',
    -\lambda\sin{(\cdot)}
    +
    \frac{1}{2\pi}
    \left(
        h_2'
        +
        \mathcal H[h_2]
    \right)
\Bigg),
\end{aligned}
$$
where $\mathcal H$ denotes the periodic Hilbert transform. The operator is
independent of $\sigma$ and is diagonal in Fourier variables. The scalar
parameter $\lambda$ supplies the first sine mode in the second component,
whereas the boundary perturbation $h_2$ controls all the higher modes. This
allows us to prove that the linearized operator is an isomorphism and to
conclude by the {above-mentioned version of the} implicit function theorem.

\subsection{Organization of the paper}

The remainder of the paper is organized as follows. In Section
\ref{section_reformulation}, we derive the rotating-boundary equations,
perform the symmetry reduction, and introduce the scaled nonlinear operator
$\mathcal F^\sigma$. In Section \ref{section_trivial_solutions}, we identify
the limiting configuration and determine the scaling coefficient
$\Omega_0$ and the admissible range of $\alpha$. In Section
\ref{sction4}, we introduce the symmetry-adapted Banach spaces,
establish the continuity and differentiability properties of
$\mathcal F^\sigma$. Finally, in Section
\ref{sec:linop} we compute the linearized operator at the limiting configuration, show that it is invertible, and complete the
proof of Theorem \ref{main_theo}. The appendix \ref{appendix_techinal_results} collects the potential theory identities and singular-integral estimates used throughout the paper.

\subsection{Acknowledgments}

V.B. is  supported by the FPI grant from the Spanish Government PREP2022-000038, and partially  by the AEI project PID2022-140494NA-I00 (Spain). C.G. has been supported by RYC2022-035967-I (MCIU/AEI/10.1303/501100011033 andFSE+), and partially by Grants PID2022-140494NA-I00 and PID2022-137228OB-I00 funded
by MCIN/AEI/10.\\13039/501100011033/FEDER, UE,  by the RED2024-153842-T, funded by MCIN/AEI /10.13039\\/501100011033, by Grant C-EXP-265-UGR23 funded by
Consejeria de Universidad, Investigacion e Innovacion \& ERDF/EU Andalusia Program, and
by Modeling Nature Research Unit, project QUAL21-011. Proyecto realizado con la Beca Leonardo a
Investigadores y Creadores Culturales 2024 de la Fundaci\'on BBVA. Z.H. is supported by Grant RYC2023-043706-I funded by MICIU/AEI/10.13039/501100011033 and by ESF+ and by grant PID2022-137228OB-I00 funded by
MCIN/AEI/10.13039/501100011033/FEDER.
J.C.C. and J.M. are supported by grant PID2024-155320NB-I00.

\begin{figure}[h]
 \centering
\begin{subfigure}{.3\textwidth}
  \centering
  \includegraphics[width=1\linewidth]{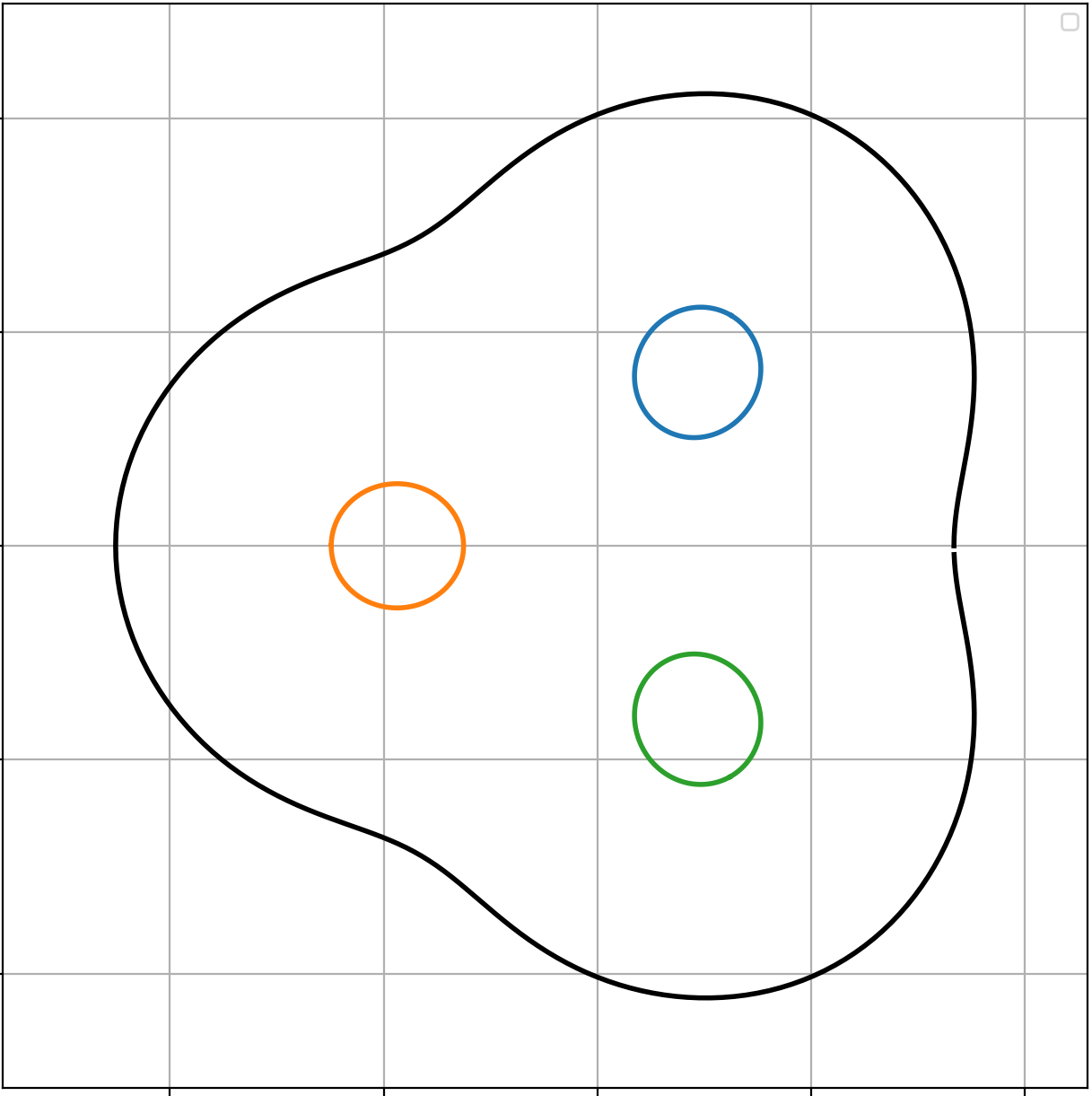}
\end{subfigure}%
\begin{subfigure}{.3\textwidth}
  \centering
  \includegraphics[width=1\linewidth]{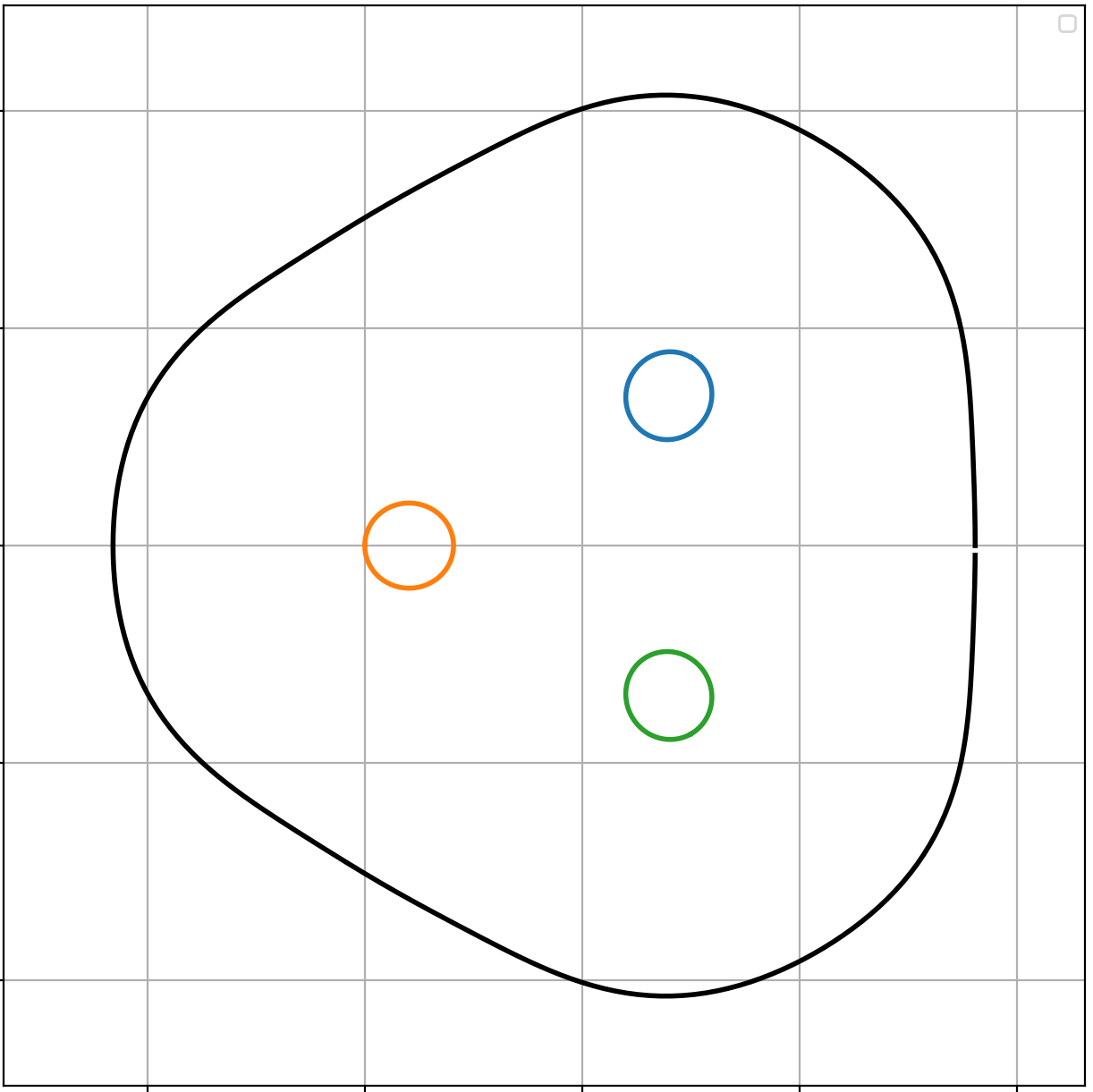}
\end{subfigure}
\begin{subfigure}{.3\textwidth}
  \centering
  \includegraphics[width=1\linewidth]{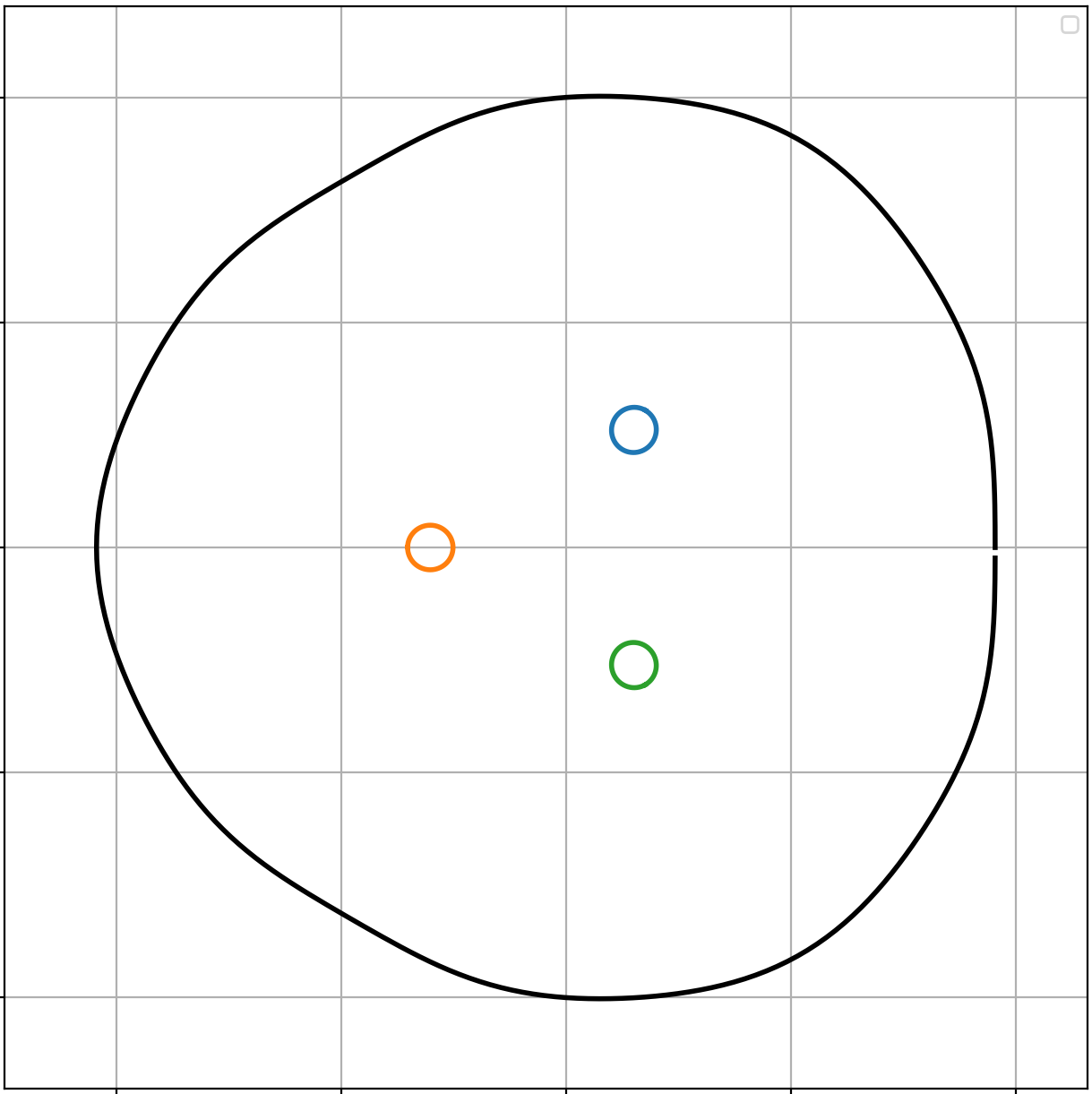}
\end{subfigure}
\caption{Numerically computed rotating configurations for $\sigma=1$ and
$\mathbf m=3$. From left to right,
$\varepsilon=0.05$, $0.10$, and $0.15$. }
\label{k3}
\end{figure}

\begin{figure}[h]
 \centering
\begin{subfigure}{.3\textwidth}
  \centering
  \includegraphics[width=1\linewidth]{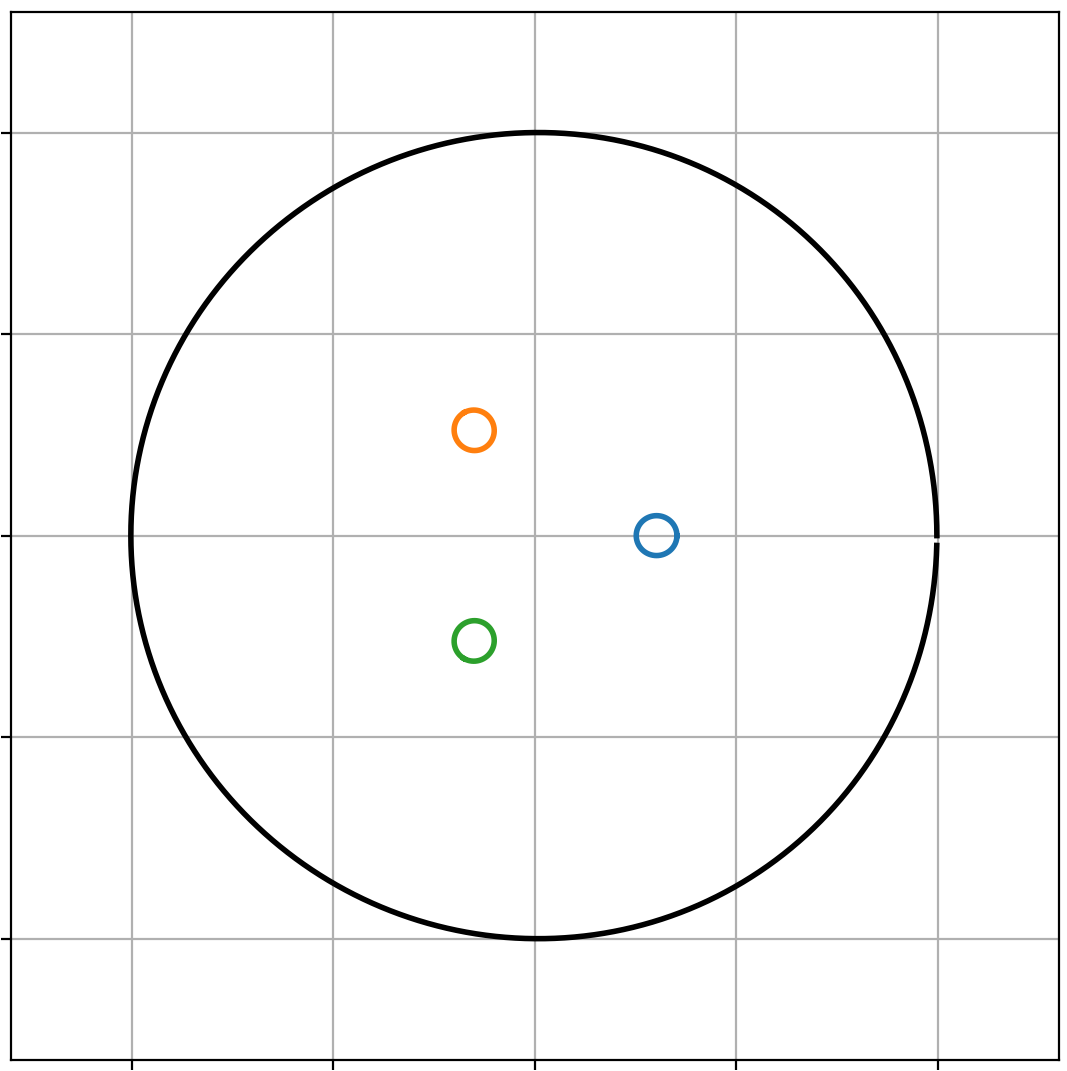}
\end{subfigure}%
\begin{subfigure}{.3\textwidth}
  \centering
  \includegraphics[width=1\linewidth]{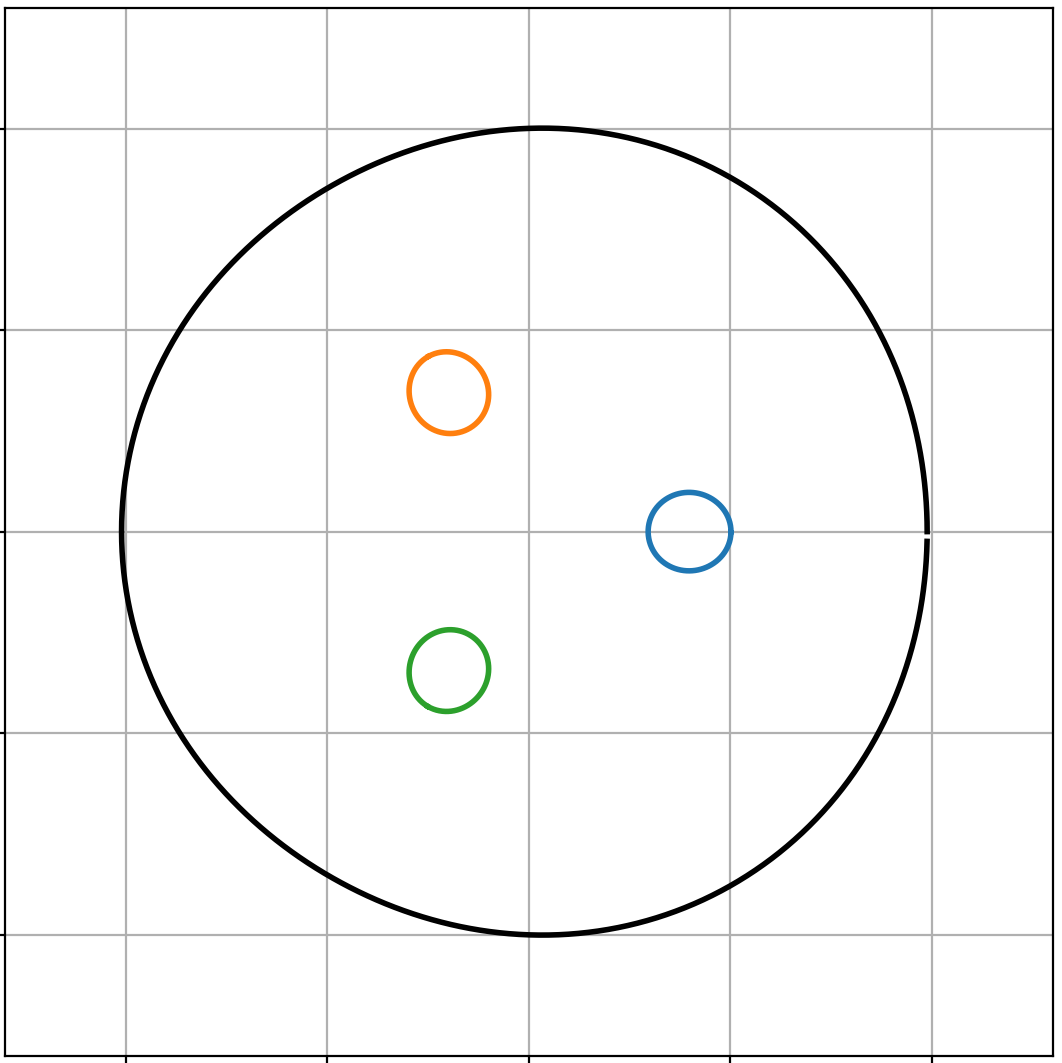}
\end{subfigure}
\begin{subfigure}{.3\textwidth}
  \centering
  \includegraphics[width=1\linewidth]{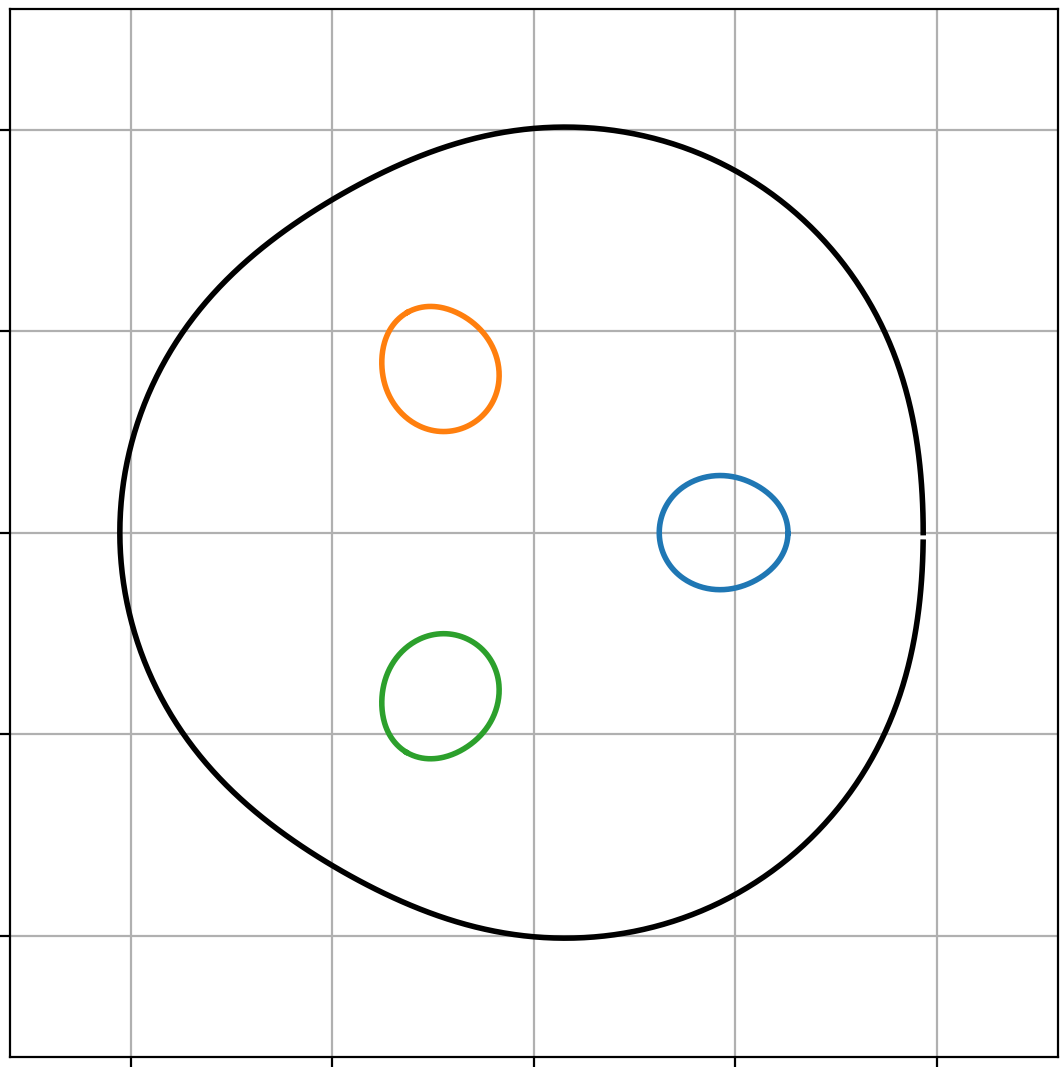}
\end{subfigure}
\caption{Numerically computed rotating configurations for $\sigma=0$ and
$\mathbf m=3$. From left to right,
$\varepsilon=0.05$, $0.10$, and $0.15$.}
\label{k32}
\end{figure}

\begin{figure}[h]
 \centering
\begin{subfigure}{.3\textwidth}
  \centering
  \includegraphics[width=1\linewidth]{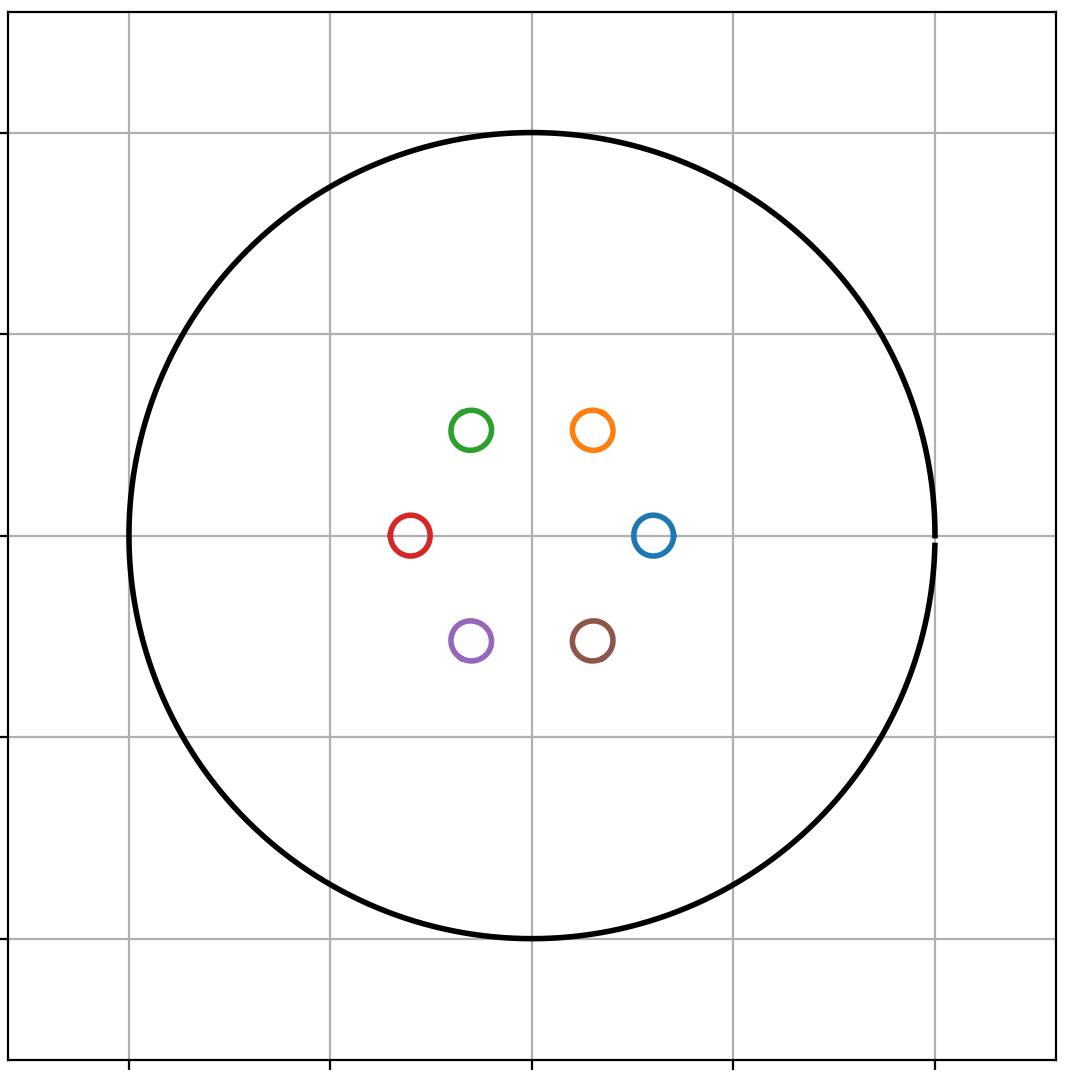}
\end{subfigure}%
\begin{subfigure}{.3\textwidth}
  \centering
  \includegraphics[width=1\linewidth]{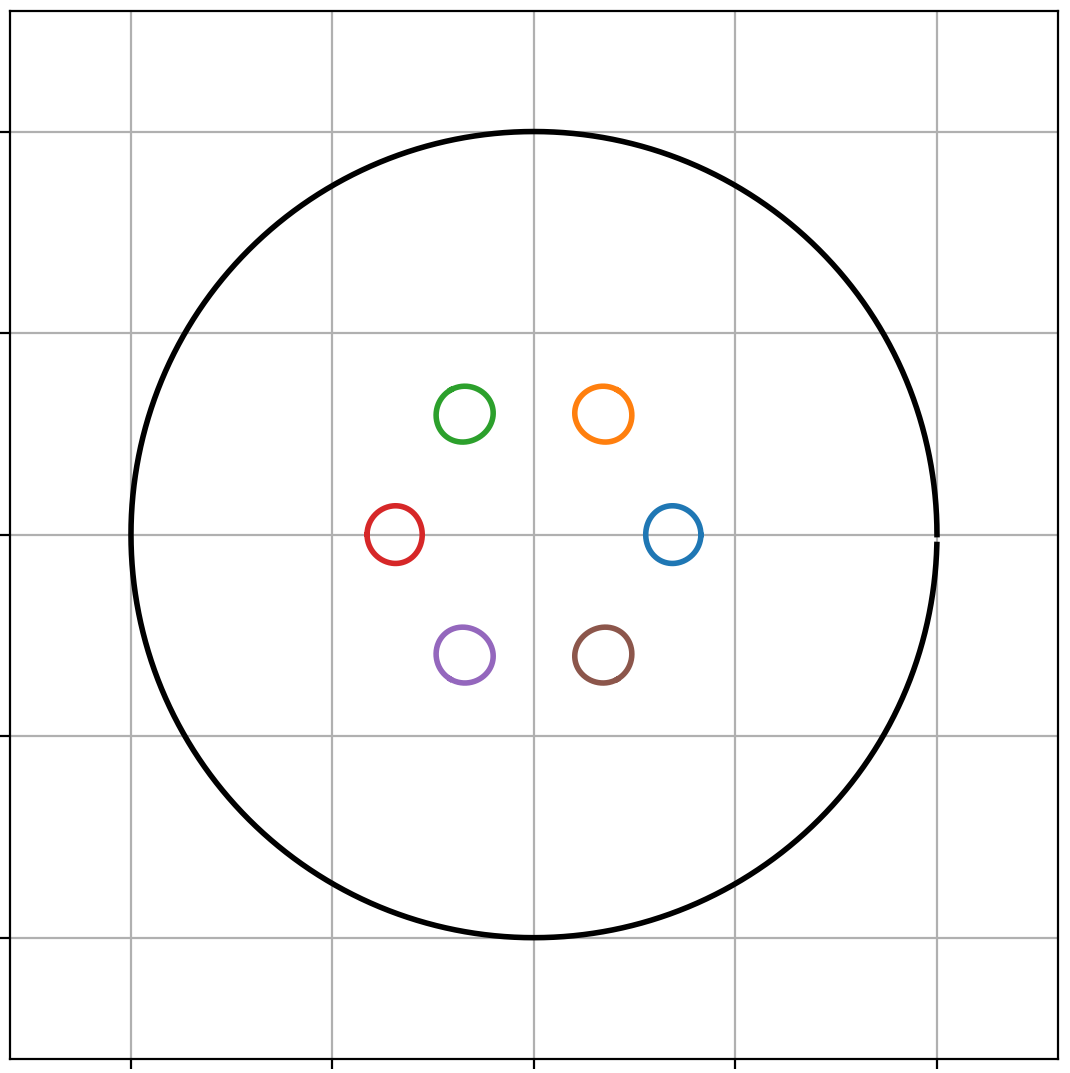}
\end{subfigure}
\begin{subfigure}{.3\textwidth}
  \centering
  \includegraphics[width=1\linewidth]{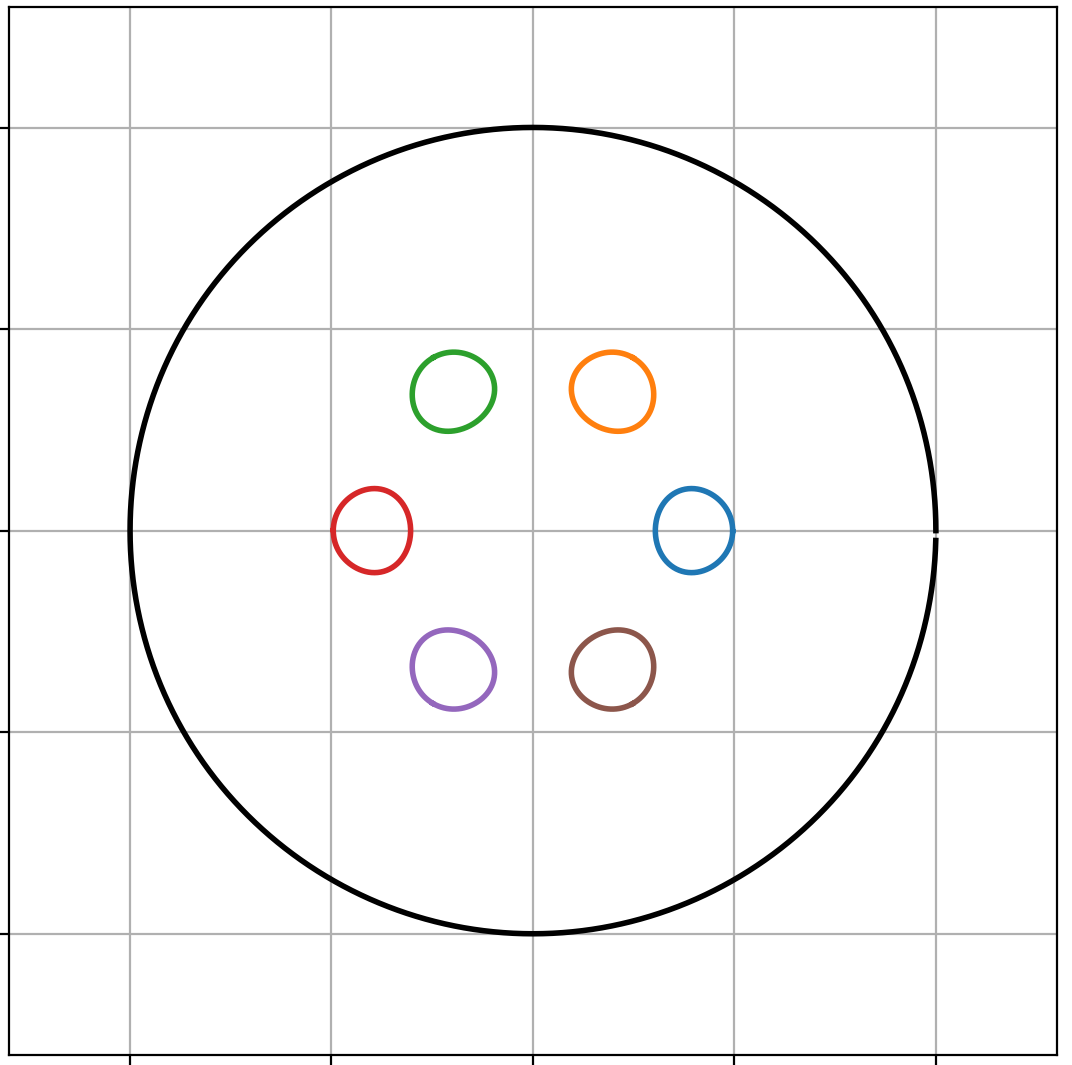}
\end{subfigure}
\caption{Numerically computed rotating configurations for $\sigma=0$ and
$\mathbf m=6$. From left to right,
$\varepsilon=0.05$, $0.07$, and $0.10$.}
\label{k6}
\end{figure}

\section{Formulation of the problem and scaling} \label{section_reformulation}

This section is devoted to reformulating the problem of finding uniformly rotating vortex patches with holes as a problem of finding the zeros of a suitable nonlinear operator. 
\subsection{Geometric setting and rotating-patch formulation}
Let
$$
    \mathbb T := \mathbb R/(2\pi\mathbb Z).
$$
Throughout this section, we identify $\mathbb R^2$ with the complex plane
$\mathbb C$. In particular, rotations through an angle $\theta$ are represented
by multiplication by $e^{i\theta}$. We use the convention
$$
    x^\perp := (-x_2,x_1),
    \qquad x=(x_1,x_2)\in\mathbb R^2.
$$

Fix an integer $\mathbf{m}\geq 2$, and let $\alpha>0$ be a parameter that will be
specified later. For every $\varepsilon>0$, let $\mathcal D_1^\varepsilon$ and
$\mathcal D_2^\varepsilon$ be bounded, simply connected domains with
sufficiently regular boundaries. Let $\sigma\in\mathbb R$ be a phase parameter. For $j\in\{1,\ldots,\mathbf{m}\}$,  define
\begin{equation}
\label{def_D2j}
    \mathcal D_{2,j}^{\varepsilon,\sigma}
    := e^{i\frac{\pi\sigma}{\mathbf{m}}}
    e^{i\frac{2\pi(j-1)}{\mathbf{m}}}\left(
        \varepsilon\mathcal D_2^\varepsilon
        +\varepsilon^\alpha \mathtt e_1
    \right),
    \qquad
    \mathtt e_1:=(1,0).
\end{equation}
When referring to the sets
$\mathcal D_{2,j}^{\varepsilon,\sigma}$ as holes, we restrict attention to
admissible configurations satisfying
$$
    \overline{\mathcal D_{2,j}^{\varepsilon,\sigma}}
    \subset \mathcal D_1^\varepsilon,
    \qquad
    \overline{\mathcal D_{2,j}^{\varepsilon,\sigma}}
    \cap
    \overline{\mathcal D_{2,k}^{\varepsilon,\sigma}}
    =\varnothing,
    \quad j\neq k.
$$

We consider the initial vorticity
$$
\omega_{0, \varepsilon}^\sigma(x) := \mathbf 1_{\mathcal D_1^\varepsilon}(x) - \frac{1}{\pi \varepsilon^2} \sum_{j = 1}^{\mathbf{m}} \mathbf 1_{\mathcal D_{2, j}^{\varepsilon,\sigma}}(x), \qquad x \in \R^2.
$$

We are interested in constructing a weak solution to \eqref{ESVF} with initial datum $\omega_{0, \varepsilon}$ that rotates rigidly about the origin with some constant angular velocity $\Omega \in \R$. In other words, we seek a weak solution to \eqref{ESVF} of the form 
$$\omega(t, x) = \omega_{0, \varepsilon}^\sigma\left(e^{- i \Omega t} x\right), \qquad (t, x) \in [0, \infty) \times \R^2,$$
for some $\Omega \in \R$. Plugging this ansatz into the  Euler equations, the problem reduces to solving the following system of $\mathbf{m}+1$ boundary equations for the unknowns $\Omega, \mathcal{D}_1^\varepsilon, \mathcal D_{2, 1}^{\varepsilon,\sigma}, \dots, \mathcal{D}_{2, \mathbf{m}}^{\varepsilon,\sigma}$:
\begin{equation} \label{Euler_System}
    \begin{cases}
    \left(u_{0, \varepsilon}(x) - \Omega x^\perp\right) \cdot \vec{n}_{\partial \mathcal D_1^\varepsilon}(x) = 0, & x \in \partial \mathcal D_{1}^\varepsilon, \\
    \left(u_{0, \varepsilon}(x) - \Omega x^\perp\right)\cdot \vec{n}_{\partial \mathcal D^{\varepsilon,\sigma}_{2, j}}(x) = 0, & x \in \partial \mathcal D_{2, j}^{\varepsilon,\sigma} \quad j \in \{1, \dots, \mathbf{m}\},
    \end{cases}
\end{equation}
where the symbol $\vec{n}_{\partial \mathcal{D}}(x)$ denotes the outward unit normal vector to $\partial \mathcal{D}$ at $x \in \partial \mathcal D$, while $u_{0, \varepsilon}$ is the initial velocity of the fluid given in terms of $\omega_{0, \varepsilon}$ via Biot--Savart Law \eqref{Biot-Savart_Law}.
\subsection{Reduction by rotational symmetry}

We henceforth assume that the outer domain is $\mathbf{m}$-fold rotationally symmetric:
\begin{equation}\label{assupmtion:mfold D1}
       \mathcal D_1^\varepsilon
    =
    e^{i\frac{2\pi}{\mathbf{m}}}\mathcal D_1^\varepsilon.
\end{equation}
Note that, from \eqref{def_D2j}, one has
\begin{equation}\label{assupmtion: Dj}
 e^{i\frac{2\pi}{\mathbf m}}
\mathcal D_{2,j}^{\varepsilon,\sigma}
=
\mathcal D_{2,j+1}^{\varepsilon,\sigma},
\end{equation}
where the indices are understood modulo $\mathbf{m}$. Hence the family
$\{\mathcal D_{2,j}^{\varepsilon,\sigma}\}_{j=1}^\mathbf{m}$ is $\mathbf{m}$-fold
rotationally symmetric for every $\sigma\in\mathbb R$.
Moreover,
\begin{equation*}
    \mathcal D_{2,j}^{\varepsilon,\sigma+2q}
    =
    \mathcal D_{2,j+q}^{\varepsilon,\sigma},
    \qquad q\in\mathbb Z,
\end{equation*}
so that, up to a cyclic relabeling of the inner components, the phase
parameter is defined modulo~$2$. Thus, 
$$
    \sigma\in\mathbb R/2\mathbb Z.
$$
We may therefore choose the representative
$$
    \sigma\in[0,2).
$$
 The parameter $\sigma$ describes the orientation of the polygon of inner
components relative to the outer boundary.

The following proposition shows that it is enough to impose the boundary
condition on a single inner component.
\begin{pro} \label{reduction_by_symmetry}
Under the assumption that $\mathcal D_1^\varepsilon$ is $\mathbf{m}$-fold rotationally symmetric, for every $k\in\{1,\ldots,\mathbf{m}\}$ and every
$x\in\partial\mathcal D_{2,1}^{\varepsilon,\sigma}$, setting $y:=e^{i\frac{2\pi(k-1)}{\mathbf{m}}} x$, one has
$$
\begin{aligned}
    &\bigl(u_{0,\varepsilon}(y)-\Omega y^\perp\bigr)
    \cdot \vec{n}_{\partial\mathcal D_{2,k}^{\varepsilon,\sigma}}(y)
   =
    \bigl(u_{0,\varepsilon}(x)-\Omega x^\perp\bigr)
    \cdot\vec{n}_{\partial\mathcal D_{2,1}^{\varepsilon,\sigma}}(x).
\end{aligned}
$$
In particular, the relative velocity is tangent to
$\partial\mathcal D_{2,k}^{\varepsilon,\sigma}$ if and only if it is tangent to
$\partial\mathcal D_{2,1}^{\varepsilon,\sigma}$.
\end{pro}
\begin{proof}
The cyclic permutation of  the inner domains \eqref{assupmtion: Dj},
together with  the rotational
symmetry of $\mathcal D_1^\varepsilon$, implies
$$
    \omega_{0,\varepsilon}^\sigma\left( e^{i\frac{2\pi}{\mathbf{m}}} x\right)
    =
    \omega_{0,\varepsilon}^\sigma(x),
    \qquad x\in\mathbb R^2.
$$
 Indeed, using \eqref{assupmtion:mfold D1} and   \eqref{assupmtion: Dj},  we obtain
\begin{align*}
\omega_{0,\varepsilon}^\sigma
\left(
    e^{i\frac{2\pi}{\mathbf m}}x
\right)
&=
\mathbf 1_{
    e^{-i\frac{2\pi}{\mathbf m}}\mathcal D_1^\varepsilon
}(x)
-
\frac{1}{\pi\varepsilon^2}
\sum_{j=1}^{\mathbf m}
\mathbf 1_{
    e^{-i\frac{2\pi}{\mathbf m}}
    \mathcal D_{2,j}^{\varepsilon,\sigma}
}(x)
\\
&=
\mathbf 1_{\mathcal D_1^\varepsilon}(x)
-
\frac{1}{\pi\varepsilon^2}
\sum_{j=1}^{\mathbf m}
\mathbf 1_{\mathcal D_{2,j-1}^{\varepsilon,\sigma}}(x)
\\
&=
\omega_{0,\varepsilon}^\sigma(x).
\end{align*}
Since the Biot-Savart kernel (defined in \eqref{Biot-Savart_Law}) satisfies 
$$
    K\left( e^{i\frac{2\pi}{\mathbf{m}}} x\right)= e^{i\frac{2\pi}{\mathbf{m}}} K(x),
$$
the change of variables $y= e^{i\frac{2\pi}{\mathbf{m}}} z$ gives
\begin{align*}
    u_{0,\varepsilon}\left( e^{i\frac{2\pi}{\mathbf{m}}}x\right)
    &=
    \int_{\mathbb R^2}
    K\left( e^{i\frac{2\pi}{\mathbf{m}}} x-y\right)\omega_{0,\varepsilon}^\sigma(y)\,dy
    \\
    &=
    \int_{\mathbb R^2}
    K\left(e^{i\frac{2\pi}{\mathbf{m}}}(x-z)\right)
    \omega_{0,\varepsilon}^\sigma\left( e^{i\frac{2\pi}{\mathbf{m}}} z\right)\,dz
    \\
    &=
    e^{i\frac{2\pi}{\mathbf{m}}}\int_{\mathbb R^2}
    K(x-z)\omega_{0,\varepsilon}^\sigma(z)\,dz
    \\
    &=
    e^{i\frac{2\pi}{\mathbf{m}}} u_{0,\varepsilon}(x).
\end{align*}
Iterating this identity yields
$$
    u_{0,\varepsilon}\left(e^{i\frac{2\pi \ell }{\mathbf{m}}} x\right)
    =
    e^{i\frac{2\pi \ell}{\mathbf{m}}} u_{0,\varepsilon}(x),
    \qquad \ell\in\mathbb Z.
$$
If $y=e^{i\frac{2\pi (k-1)}{\mathbf{m}}}x$, then rotations preserve outward unit normals, so that
$$
    \vec{n}_{\partial\mathcal D_{2,k}^{\varepsilon,\sigma}}(y)
    =
    e^{i\frac{2\pi (k-1)}{\mathbf{m}}}
    \vec{n}_{\partial\mathcal D_{2,1}^{\varepsilon,\sigma}}(x).
$$
Moreover, rotations commute with the map $x\mapsto x^\perp$. Therefore,
using the invariance of the Euclidean scalar product under rotations,
\begin{align*}
    \bigl(u_{0,\varepsilon}(y)-\Omega y^\perp\bigr)
    \cdot \vec{n}_{\partial\mathcal D_{2,k}^{\varepsilon,\sigma}}(y)
    &=
    e^{i\frac{2\pi(k-1)}{\mathbf{m}}}
    \bigl(u_{0,\varepsilon}(x)-\Omega x^\perp\bigr)
    \cdot
    e^{i\frac{2\pi (k-1)}{\mathbf{m}}}
    \vec{n}_{\partial\mathcal D_{2,1}^{\varepsilon,\sigma}}(x)
    \\
    &=
    \bigl(u_{0,\varepsilon}(x)-\Omega x^\perp\bigr)
    \cdot \vec{n}_{\partial\mathcal D_{2,1}^{\varepsilon,\sigma}}(x).
\end{align*}
This proves the claim.
\end{proof}

Consequently, for a fixed phase $\sigma$, the problem reduces to a system of
two boundary equations for the angular velocity $\Omega$ and the
two generating domains $\mathcal D_1^\varepsilon$ and
$\mathcal D_2^\varepsilon$, the remaining inner components being determined
by \eqref{def_D2j}. Namely, the system \eqref{Euler_System} is equivalent to
\begin{equation}
\label{New_System}
    \begin{cases}
    \bigl(u_{0,\varepsilon}(x)-\Omega x^\perp\bigr)
    \cdot\vec{n}_{\partial\mathcal D_1^\varepsilon}(x)=0,
    &
    x\in\partial\mathcal D_1^\varepsilon,
    \\[2mm]
    \bigl(u_{0,\varepsilon}(x)-\Omega x^\perp\bigr)
    \cdot \vec{n}_{\partial( e^{i\frac{\pi\sigma}{\mathbf{m}}}(\varepsilon  \mathcal D_2^\varepsilon
        +\varepsilon^\alpha \mathtt  e_1))}(x)=0,
    &
    x\in
    \partial\bigl(
        e^{i\frac{\pi\sigma}{\mathbf{m}}}(\varepsilon\mathcal D_2^\varepsilon
        +\varepsilon^\alpha\mathtt e_1)
    \bigr).
    \end{cases}
\end{equation}

\subsection{Stream-function formulation}  Henceforth, we fix $\sigma\in[0,2)$. Unless needed for clarity, the
dependence of the vorticity, velocity, stream function
on $\sigma$ will be suppressed from the notation.
Let $\Psi_\varepsilon$ denote the stream function associated with $\omega_{0, \varepsilon}$. With the convention used in
\eqref{steam_vorticity},
\begin{equation} \label{Psi_eps_def}
\Psi_{\varepsilon}(x) := \left(\frac{1}{2 \pi} \log | \cdot | * \omega_{0, \varepsilon}^\sigma\right)(x), \quad x \in \R^2,
\end{equation}
and
$$
    u_{0,\varepsilon}
    =
    \nabla^\perp\Psi_\varepsilon.
$$
Let
$$
    z_j:\mathbb T\longrightarrow\partial\mathcal D_j^\varepsilon,
    \qquad j\in\{1,2\},
$$
be regular counterclockwise parametrizations, and define
$$
    \gamma_1(\vartheta):=z_1(\vartheta),
    \qquad
    \gamma_2(\vartheta)
    :=
    e^{i\frac{\pi\sigma}{\mathbf{m}}}(\varepsilon z_2(\vartheta)
    +\varepsilon^\alpha\mathtt e_1).
$$
For any regular counterclockwise parametrization $\gamma$ of the boundary of a
domain $\mathcal D$, the outward unit normal is
$$
    \vec{n}_{\partial {\mathcal D}}(\gamma(\vartheta))
    =
    -\frac{\gamma'(\vartheta)^\perp}{|\gamma'(\vartheta)|}.
$$
It follows that
\begin{align}
\label{general_boundary_identity}
\bigl(
    u_{0,\varepsilon}(\gamma(\vartheta))
    -
    \Omega\gamma(\vartheta)^\perp
\bigr)
\cdot
\vec n_{\partial\mathcal D}(\gamma(\vartheta))
& =
\frac{
    \Omega\gamma(\vartheta)\cdot\gamma'(\vartheta)
    -
    \nabla\Psi_\varepsilon(\gamma(\vartheta))
    \cdot\gamma'(\vartheta)
}{
    |\gamma'(\vartheta)|
}
\nonumber\\
& =
\frac{
    \Omega\gamma(\vartheta)\cdot\gamma'(\vartheta)
    -
    \partial_\vartheta
    \bigl[
        \Psi_\varepsilon(\gamma(\vartheta))
    \bigr]
}{
    |\gamma'(\vartheta)|
}.
\end{align}
Thus, the boundary condition is equivalent to the vanishing of the numerator
in \eqref{general_boundary_identity}.
For the inner component, one has
\begin{align*}
\gamma_2(\vartheta)\cdot\gamma_2'(\vartheta)
&=
\left[
e^{i\frac{\pi\sigma}{\mathbf m}}
\left(
    \varepsilon z_2(\vartheta)
    +
    \varepsilon^\alpha\mathtt e_1
\right)
\right]
\cdot
\left[
e^{i\frac{\pi\sigma}{\mathbf m}}
\varepsilon z_2'(\vartheta)
\right]
\\
&=
\left(
    \varepsilon z_2(\vartheta)
    +
    \varepsilon^\alpha\mathtt e_1
\right)
\cdot
\varepsilon z_2'(\vartheta)
\\
&=
\varepsilon^2
z_2(\vartheta)\cdot z_2'(\vartheta)
+
\varepsilon^{\alpha+1}
\bigl(z_2'(\vartheta)\bigr)_1.
\end{align*}
Therefore, \eqref{New_System} is equivalent to
\begin{equation}
\label{Last_System}
    \begin{cases}
    \displaystyle
    \Omega z_1(\vartheta)\cdot z_1'(\vartheta)
    -
    \partial_\vartheta
    \bigl[
        \Psi_\varepsilon(z_1(\vartheta))
    \bigr]
    =0,
    \\[3mm]
    \displaystyle
    \Omega\left[
        \varepsilon^2
        z_2(\vartheta)\cdot z_2'(\vartheta)
        +
        \varepsilon^{\alpha+1}
        \bigl(z_2'(\vartheta)\bigr)_1
    \right]
    -
    \partial_\vartheta
    \left[
        \Psi_\varepsilon\left(e^{i\frac{\pi\sigma}{\mathbf{m}}}(
            \varepsilon z_2(\vartheta)
            +\varepsilon^\alpha\mathtt e_1)
        \right)
    \right]
    =0,
    \end{cases}
    \qquad \vartheta\in\mathbb T.
\end{equation}

\subsection{Radial parametrization and scaled nonlinear operator}

We restrict to domains that are
star-shaped with respect to the origin and write
\begin{equation}\label{Djeps}
    \mathcal D_j^\varepsilon
    =
    \left\{
        re^{i\vartheta}:
        \vartheta\in\mathbb T,\;
        0\leq r<w_j(\vartheta)
    \right\},
    \qquad j\in\{1,2\},
\end{equation}
where 
$$
    w_j(\vartheta)
    :=
    \sqrt{1+2\varepsilon f_j(\vartheta)}.
$$
Here $f_j:\mathbb T\to\mathbb R$ is sufficiently regular and
$1+2\varepsilon f_j>0$. The ${\bf m}$-fold symmetry of the outer domain is encoded
by the condition
$$
    f_1\left(\vartheta+\frac{2\pi}{{\bf m}}\right)
    =
    f_1(\vartheta).
$$
Thus, the boundary $\partial \mathcal D_j^\varepsilon $ is parametrized by
$$
    z_j(\vartheta)
    =
    w_j(\vartheta)e^{i\vartheta}.
$$
This parametrization is inspired by \cite{BHM22}, and  has the useful property
$$
    w_j'(\vartheta)
    =
    \varepsilon
    w_j(\vartheta)^{-1}
    f_j'(\vartheta),
$$
and hence
$$
    z_j'(\vartheta)
    =
    \varepsilon
    w_j(\vartheta)^{-1}
    f_j'(\vartheta)e^{i\vartheta}
    +
    iw_j(\vartheta)e^{i\vartheta}.
$$
In particular,
\begin{equation}
\label{radial_scalar_product}
    z_j(\vartheta)\cdot z_j'(\vartheta)
    =
    w_j(\vartheta)w_j'(\vartheta)
    =
    \varepsilon f_j'(\vartheta).
\end{equation}
Moreover,
\begin{equation}
\label{first_component_z2}
    \bigl(z_2'(\vartheta)\bigr)_1
    =
    \varepsilon
    w_2(\vartheta)^{-1}
    f_2'(\vartheta)\cos\vartheta
    -
    w_2(\vartheta)\sin\vartheta.
\end{equation}
The choice $w_j^2=1+2\varepsilon f_j$ also linearizes the area constraint:
$$
    |\mathcal D_j^\varepsilon|
    =
    \frac12\int_0^{2\pi}w_j(\vartheta)^2\,d\vartheta
    =
    \pi
    +
    \varepsilon
    \int_0^{2\pi}f_j(\vartheta)\,d\vartheta.
$$
Thus, imposing zero average on $f_j$ preserves the area
$|\mathcal D_j^\varepsilon|=\pi$.

In order to make the dependence of the stream function on the boundary perturbations $f := \left(f_1, f_2\right)$ explicit, we perform a change of variables in the integrals over the inner components, so that from \eqref{Psi_eps_def} it follows that $\Psi_\varepsilon$ can be written as
\begin{align}
\label{Psi_representation}
    \Psi_{\varepsilon}(z)
    &=
    \frac{1}{2\pi}
    \int_0^{2\pi}
    \int_0^{w_1(\eta)}
    \log\left|z-re^{i\eta}\right|
    r\,dr\,d\eta
    \nonumber\\
    &\quad
    -
    \frac{1}{2\pi^2}
    \sum_{j=1}^{{\bf m}}
    \int_0^{2\pi}
    \int_0^{w_2(\eta)}
    \log\left|
        z-
        e^{i \phi_j}\left(
            \varepsilon re^{i\eta}
            +\varepsilon^\alpha
        \right)
    \right|
    r\,dr\,d\eta.
\end{align}
We now introduce the scaled angular velocity
$$
    \Omega_{\varepsilon,\Lambda}
    :=
    \frac{\Omega_0}{\varepsilon^{2\alpha}}
    +
    \frac{\Lambda}{\varepsilon^\alpha},
$$
where $\Omega_0$ is a constant to be determined from the limiting problem.
Using \eqref{radial_scalar_product} and
\eqref{first_component_z2}, define
\begin{equation}\label{def F}
   \mathcal F^\sigma(\varepsilon,\Lambda,f)
    :=
    \bigl(
        \mathcal F_1^\sigma(\varepsilon,\Lambda,f),
        \mathcal F_2^\sigma(\varepsilon,\Lambda,f)
    \bigr),
\end{equation}
where $f:=(f_1,f_2)$ and 
\begin{align}
\label{F1}
  \mathcal  F_1^\sigma(\varepsilon,\Lambda,f)(\vartheta)
    &:=
    \frac{1}{\varepsilon^{1-2\alpha}}
    \Bigg\{
        \varepsilon
        \left(\frac{\Omega_0}{\varepsilon^{2 \alpha}} + \frac{\Lambda}{\varepsilon^\alpha} \right)
        f_1'(\vartheta)
        -
        \partial_\vartheta
        \left[
            \Psi_{\varepsilon}
            \left(
                w_1(\vartheta)e^{i\vartheta}
            \right)
        \right]
    \Bigg\},
\end{align}
and
\begin{align}
\label{F2}
   \mathcal F_2^\sigma(\varepsilon,\Lambda,f)(\vartheta)
    &:=
    \frac{1}{\varepsilon}
    \Bigg\{
        \left(\frac{\Omega_0}{\varepsilon^{2 \alpha}} + \frac{\Lambda}{\varepsilon^\alpha} \right)
        \Big[
            \varepsilon^3 f_2'(\vartheta)
           +
            \varepsilon^{\alpha+1}
            \Big(
                \varepsilon
                w_2(\vartheta)^{-1}
                f_2'(\vartheta)\cos\vartheta
                -
                w_2(\vartheta)\sin\vartheta
            \Big)
        \Big]
        \nonumber\\
    &\hspace{2.2cm}
        -
        \partial_\vartheta
        \left[
            \Psi_{\varepsilon}
            \left(
                e^{i\frac{\pi\sigma}{\mathbf m}}
    \left(
        \varepsilon w_2(\vartheta)e^{i\vartheta}
        +
        \varepsilon^\alpha
    \right)
            \right)
        \right]
    \Bigg\}.
\end{align}
Substituting \eqref{Psi_representation} into \eqref{F1}--\eqref{F2} gives the
corresponding fully expanded integral formulas:
\begin{align} \label{F12}
 \mathcal   F_1^\sigma(\varepsilon, \Lambda, f)(\vartheta) & =  \left(\Omega_0 + \varepsilon^\alpha\Lambda \right) \varepsilon f_1'(\vartheta)  + \mathcal S_1(\varepsilon,f_1)(\vartheta)+ \mathcal I_{1,2}(\varepsilon,f)(\vartheta),
\end{align}
\begin{align}
\label{F22}
\mathcal F_2^\sigma(\varepsilon,\Lambda,f)(\vartheta)
&=
\left(\Omega_0 + \varepsilon^\alpha \Lambda \right)
\Big[
    \varepsilon^{2-2\alpha} f_2'(\vartheta)
    +\varepsilon^{1-\alpha} w_2(\vartheta)^{-1}
        f_2'(\vartheta)\cos\vartheta
        -\varepsilon^{-\alpha}
        w_2(\vartheta)\sin\vartheta
\Big]
\nonumber\\
&\quad +\mathcal S_2(\varepsilon,f_2)(\vartheta)
+
\mathcal I_{2,1}(\varepsilon,f)(\vartheta)
+  \mathcal I_{2,2}(\varepsilon,f_2)(\vartheta),
\end{align}
where
\begin{align*}
    \mathcal S_1(\varepsilon,f_1)&:=- \frac{\partial_\vartheta}{2\pi \varepsilon^{1-2\alpha}} \int_{0}^{2\pi} \int_{0}^{w_1(\eta)} \log \left|w_1(\vartheta) e^{i \vartheta} - r e^{i \eta}\right|  r \, dr \, d\eta,\\
   \mathcal S_2(\varepsilon,f_2)&:= \frac{\partial_\vartheta}{2\pi^2\varepsilon}
\int_0^{2\pi}
\int_0^{w_2(\eta)}
\log\left| w_2(\vartheta)e^{i\vartheta}  
    - r e^{i\eta}
\right|
r\,dr\,d\eta,\\
    \mathcal I_{1,2}(\varepsilon,f)&:= \frac{\partial_\vartheta}{2 \pi^2 \varepsilon^{1-2\alpha}} \sum_{j = 1}^{\mathbf{m}} \int_{0}^{2 \pi} \int_{0}^{w_2(\eta)} \log \left|w_1(\vartheta) e^{i \vartheta} - e^{i \frac{\pi(2j-2+\sigma)}{\mathbf m}} \left(\varepsilon r e^{i \eta} + \varepsilon^\alpha \right)\right|  r \, dr \, d\eta,\\
    \mathcal I_{2,1}(\varepsilon,f)&:=-\frac{\partial_\vartheta}{2\pi \varepsilon}
\int_0^{2\pi}
\int_0^{w_1(\eta)}
\log\left|
    e^{i\frac{\pi\sigma}{\mathbf m}}
    \left(
        \varepsilon w_2(\vartheta)e^{i\vartheta}
        +
        \varepsilon^\alpha
    \right)
    -
    r e^{i\eta}
\right|
r\,dr\,d\eta
    \\
    \mathcal I_{2,2}(\varepsilon,f_2)&:= \frac{\partial_\vartheta}{2\pi^2\varepsilon}
\sum_{j=2}^{\mathbf m}
\int_0^{2\pi}
\int_0^{w_2(\eta)}
\log\left|
    \varepsilon w_2(\vartheta)e^{i\vartheta}
        +
        \varepsilon^\alpha   
    -
    e^{i\frac{2\pi(j-1)}{\mathbf m}}
    \left(
        \varepsilon r e^{i\eta}
        +
        \varepsilon^\alpha
    \right)
\right|
r\,dr\,d\eta
\end{align*}
Consequently, for admissible configurations,
$$
   \mathcal  F^\sigma(\varepsilon,\Lambda,f)=0
$$
is equivalent to the reduced free-boundary system \eqref{Last_System}. The
parameters $\alpha$ and $\Omega_0$ will be chosen so that the operator admits
a finite and nontrivial limiting formulation as $\varepsilon\to0^{{+}}$, suitable
for the subsequent application of the implicit function theorem.

\section{The limiting configuration and the choice of scaling parameters}
\label{section_trivial_solutions}
The purpose of this section is to identify the singular base configuration
for the subsequent implicit-function argument and to determine the admissible
values of the scaling parameters $\alpha$ and $\Omega_0$. More precisely, we
compute the limit
$$
    \lim_{\varepsilon\to0^+}
   \mathcal  F^\sigma(\varepsilon,\Lambda,0),
$$
and choose $\Omega_0$ so that
$$
   \mathcal  F^\sigma(0,0,0)=0.
$$

For the moment, the notation $\mathcal  F^\sigma(0,\Lambda,0)$ denotes the above
one-sided limit. In the following section, we introduce the functional
framework and show that this definition is compatible with the appropriate extension of the nonlinear operator to $\varepsilon=0$.

For $f=0$ and $\varepsilon>0$, the outer component is the unit disk and the
inner components are $\mathbf m$ disks of radius $\varepsilon$, whose centers
form a regular polygon of radius $\varepsilon^\alpha$. As
$\varepsilon\to0^+$, these components shrink and their centers collapse to
the origin. The resulting singular limit consists formally of a Rankine
vortex together with a point vortex of total circulation $-\mathbf m$ at the
origin. The leading coefficient $\Omega_0$, however, is determined by the
angular velocity of the regular polygon formed by the $\mathbf m$
concentrated vortices.

\begin{pro}
\label{F1_at_0}
Let $\sigma\in\mathbb R$. If
$$
    \alpha>\frac{1}{\mathbf m+2},
$$
then, for every $\Lambda\in\mathbb R$,
$$
   \mathcal  F_1^\sigma(0,\Lambda,0){(\vartheta)}=0, \qquad {\vartheta \in \R}.
$$
\end{pro}

\begin{proof}
{Fix $\Lambda \in \R$ and $\vartheta \in \T$.} For $f=0$, one has $w_1=w_2=1$ and
$$
\mathcal   F_1^\sigma(\varepsilon, \Lambda, 0)(\vartheta)  =    \mathcal S_1(\varepsilon,0)+ \mathcal I_{1,2}(\varepsilon,0),
$$
where
\begin{align*}
       \mathcal S_1(\varepsilon,0)&=-\frac{1}{\varepsilon^{1-2\alpha}}  \frac{\partial_\vartheta}{2\pi} \int_{0}^{2\pi} \int_{0}^{1} \log \left|e^{i \vartheta} - r e^{i \eta} \right| \, r \, dr \, d \eta \\
        \mathcal I_{1,2}(\varepsilon,0)&=\frac{1}{\varepsilon^{1-2\alpha}}\frac{\partial_\vartheta}{2\pi^2} \sum_{j = 1}^{\mathbf{m}} \int_{0}^{2\pi} \int_{0}^{1} \log \left|e^{i\vartheta} - e^{i \phi_j} \left(\varepsilon r e^{i \eta} + \varepsilon^\alpha\right) \right| \, r \, dr \, d\eta,
    \end{align*}
    where
    $$
    \phi_j
    :=
    \frac{\pi(2j-2+\sigma)}{\mathbf m},
    \qquad
    j\in\{1,\ldots,\mathbf m\}.
$$
    Using Lemma \ref{integral}
we can eliminate the dependence on the radius of each inner disk,
    \begin{align*}
       -\varepsilon^{1-2\alpha} \mathcal S_1(\varepsilon,0)(\vartheta) &=\frac{1}{2\pi} \int_{0}^{2\pi} \int_{0}^{1} \log \left|e^{i \vartheta} - r e^{i \eta} \right| \, r \, dr \, d \eta\\ & =  \frac{1}{2\pi} \int_{0}^{2\pi} \int_{0}^{1} \log\left| 1 - r e^{i(\eta - \vartheta)}\right| \, r \, dr \, d \eta\\ 
        & = \frac{1}{2\pi} \int_{0}^{1} \int_{0}^{2\pi} \log\left|1 - re^{i\tau}\right| \, d\tau \, r \, dr  = 0,
    \end{align*}
    As for the second integral one may write
    \begin{align*}
        \varepsilon^{1-2\alpha} \mathcal I_{1,2}(\varepsilon,0)(\vartheta) &=\frac{1}{2\pi} \sum_{j = 1}^{\mathbf{m}} \partial_\vartheta \log\left|e^{i \vartheta} - e^{i \phi_j} \varepsilon^\alpha \right|\\ &\quad + \frac{\partial_\vartheta}{2\pi^2} \sum_{j = 1}^{\mathbf{m}} \int_{0}^{2\pi} \int_{0}^{1} \log \left|1 - \frac{e^{i \phi_j} \varepsilon r }{e^{i \vartheta} - e^{i \phi_j}\varepsilon^\alpha} e^{i\eta}\right| \, r \, dr \, d\eta .
    \end{align*}
For all $r\in[0,1]$ and for  sufficiently small
$\varepsilon>0$ one has
$$
\left|\frac{e^{i \phi_j} \varepsilon r }{e^{i \vartheta} - e^{i \phi_j}\varepsilon^\alpha} \right|\leq \frac{\varepsilon}{1 - \varepsilon^\alpha}<1.
$$
From Lemma \ref{integral} it follows that
$$
\varepsilon^{1-2\alpha} \mathcal I_{1,2}(\varepsilon,0)(\vartheta)
=
\frac{1}{2\pi}
\partial_\vartheta
\sum_{j=1}^{\mathbf m}
\log\left|
e^{i\vartheta}
-
\varepsilon^\alpha
e^{i\phi_j}
\right|.
$$
The numbers
$$
    e^{i\phi_j},
    \qquad j\in \{1,\ldots,\mathbf m\},
$$
are the roots of $z^{\mathbf m}=e^{i\pi\sigma}$. Consequently,
$$
\prod_{j=1}^{\mathbf m}
\left(
z-
\varepsilon^\alpha
e^{i\phi_j}
\right)
=
z^{\mathbf m}
-
\varepsilon^{\mathbf m\alpha}e^{i\pi\sigma}.
$$
Taking $z=e^{i\vartheta}$ gives
$$
\varepsilon^{1-2\alpha} \mathcal I_{1,2}(\varepsilon,0)(\vartheta)
=
\frac{1}{2\pi}
\partial_\vartheta
\log\left|
e^{i\mathbf m\vartheta}
-
\varepsilon^{\mathbf m\alpha}e^{i\pi\sigma}
\right|.
$$
Since
$$
\left|
e^{i\mathbf m\vartheta}
-
\varepsilon^{\mathbf m\alpha}e^{i\pi\sigma}
\right|^2
=
1
-
2\varepsilon^{\mathbf m\alpha}
\cos(\mathbf m\vartheta-\pi\sigma)
+
\varepsilon^{2\mathbf m\alpha},
$$
we finally obtain
$$
\varepsilon^{1-2\alpha} \mathcal I_{1,2}(\varepsilon,0)(\vartheta)
=
\frac{\mathbf m}{2\pi}
\varepsilon^{\mathbf m\alpha}
\frac{
    \sin(\mathbf m\vartheta-\pi\sigma)
}{
    1
    -
    2\varepsilon^{\mathbf m\alpha}
    \cos(\mathbf m\vartheta-\pi\sigma)
    +
    \varepsilon^{2\mathbf m\alpha}
}.
$$
Therefore,
$$
\mathcal{F}_1^\sigma(\varepsilon,\Lambda,0)(\vartheta)
=
\frac{\mathbf m}{2\pi}
\varepsilon^{(\mathbf m+2)\alpha-1}
\frac{
    \sin(\mathbf m\vartheta-\pi\sigma)
}{
    1
    -
    2\varepsilon^{\mathbf m\alpha}
    \cos(\mathbf m\vartheta-\pi\sigma)
    +
    \varepsilon^{2\mathbf m\alpha}
}.
$$
In particular,
$$
    \left\|\mathcal{F}_1^\sigma(\varepsilon,\Lambda,0)\right\|_{L^\infty}
    \leq
    C_{\mathbf m}
    \varepsilon^{(\mathbf m+2)\alpha-1}.
$$
Thus
$$
 \mathcal  F_1^\sigma(0,\Lambda,0){(\vartheta)}=0,
$$
whenever
$$
    \alpha>\frac{1}{\mathbf m+2}.
$$
  That ends the proof of this proposition.
\end{proof}

We next analyze the second component of the nonlinear operator. The leading
singular term is canceled by choosing $\Omega_0$ equal to the angular
velocity coefficient of a regular polygon of $\mathbf m$ identical point
vortices. An exact computation of the inner--inner interaction also reveals
an exceptional cancellation in the pentagonal case $\mathbf m=5$.

\begin{pro}
\label{F2_at_0}
Fix a phase $\sigma\in\mathbb R$, and let
$$
    \Omega_0=\frac{1-\mathbf m}{4\pi}.
$$
Then the following assertions hold.

\begin{enumerate}
    \item If $\mathbf m\neq5$ and
    $$
        0<\alpha<\frac12,
    $$
    then, for every $\Lambda\in\mathbb R$,
    $$
       \mathcal  F_2^\sigma(0,\Lambda,0)(\vartheta)
        =
        -\Lambda\sin\vartheta,
        \qquad \vartheta\in\mathbb T.
    $$

    \item If $\mathbf m=5$ and
    $$
        0<\alpha<\frac23,
    $$
    then the same conclusion holds:
    $$
       \mathcal  F_2^\sigma(0,\Lambda,0)(\vartheta)
        =
        -\Lambda\sin\vartheta,
        \qquad \vartheta\in\mathbb T.
    $$
\end{enumerate}
\end{pro}

\begin{proof}
Fix $\Lambda\in\mathbb R$ and $\vartheta\in\mathbb T$. Throughout the proof,
$\varepsilon>0$ is assumed to be sufficiently small. At $f=0$, one has 
   $ w_1=w_2=1.$
Using the expression of $\mathcal{F}_2^\sigma$, given by \eqref{F22}, we obtain
\begin{align}
\label{F2_at_trivial}
\mathcal  F_2^\sigma(\varepsilon,\Lambda,0)(\vartheta)
&=
-\varepsilon^{-\alpha}\Omega_0\sin\vartheta
-\Lambda\sin\vartheta+\mathcal S_2(\varepsilon,0)(\vartheta)
+
\mathcal I_{2,1}(\varepsilon,0)(\vartheta)
+  \mathcal I_{2,2}(\varepsilon,0)(\vartheta),
\end{align}
where
\begin{align*}
\mathcal S_2(\varepsilon,0)(\vartheta)=
\frac{\partial_\vartheta}{2\pi^2 \varepsilon}
\int_0^{2\pi}\int_0^1
\log\left| e^{i\vartheta}
- r e^{i\eta}\right|
r\,dr\,d\eta.
\end{align*}
\begin{align*}
\mathcal I_{2,1}(\varepsilon,0)(\vartheta)
=
-\frac{\partial_\vartheta}{2\pi \varepsilon}
\int_0^{2\pi}\int_0^1
\log\left|
e^{i\frac{\pi\sigma}{\mathbf m}}
\left(
    \varepsilon e^{i\vartheta}
    +\varepsilon^\alpha
\right)
-r e^{i\eta}
\right|
r\,dr\,d\eta,
\end{align*}
and
\begin{align*}
\mathcal I_{2,2}(\varepsilon,0)(\vartheta)
=
\frac{\partial_\vartheta}{2\pi^2 \varepsilon}
\sum_{j=2}^{\mathbf m}
\int_0^{2\pi}\int_0^1
\log\left|
\varepsilon e^{i\vartheta}
+\varepsilon^\alpha
-
e^{i\frac{2\pi(j-1)}{\mathbf m}}
\left(
    \varepsilon r e^{i\eta}
    +\varepsilon^\alpha
\right)
\right|
r\,dr\,d\eta.
\end{align*}
First, notice that, similarly to $\mathcal S_1(\varepsilon,0)(\vartheta)$, by  Lemma \ref{integral}, we get
\begin{equation}
\label{J1_exact}
  \mathcal S_2(\varepsilon,0)(\vartheta)
    =0.
\end{equation}
Let's now compute the outer-domain contribution $\mathcal I_{2,1}(\varepsilon,0)(\vartheta)$. Note that, by setting $z = r e^{i \eta}$, one has $dA(z) = r \, dr \, d \eta$; moreover, the integration domain 
    $$\left\{(r, \eta) : 0 \leq r \leq 1, \, 0 \leq \eta \leq 2\pi\right\},$$
    coincides with the unit disk $\D \subseteq \C$. Hence, we may rewrite the integral as
    $$\varepsilon\, \mathcal I_{2,1}(\varepsilon,0)(\vartheta)= - \frac{\partial_\vartheta}{2\pi} \int_{\D} \log\left|z - a_{\varepsilon}^{\sigma}(\vartheta)\right| dA(z), \quad\textnormal{with}\quad  a_{\varepsilon}^{\sigma}(\vartheta)
    :=
    e^{i\frac{\pi\sigma}{\mathbf m}}
    \left(
        \varepsilon e^{i\vartheta}
        +\varepsilon^\alpha
    \right).$$
Since
$$
    \left|a_{\varepsilon}^{\sigma}(\vartheta)\right|
    =
    \left|
        \varepsilon e^{i\vartheta}
        +\varepsilon^\alpha
    \right|
    \leq
    \varepsilon+\varepsilon^\alpha<1,
$$
for sufficiently small $\varepsilon>0$, the point
$a_{\varepsilon}^{\sigma}(\vartheta)$ belongs to the unit disk. By the
logarithmic-potential formula for the unit disk in Lemma \ref{integral_disk}, we obtain
\begin{align*}
\varepsilon\, \mathcal I_{2,1}(\varepsilon,0)(\vartheta)
&=
-\frac{\partial_\vartheta}{2\pi}
\int_{\mathbb D}
\log\left|
z-a_{\varepsilon}^{\sigma}(\vartheta)
\right|
\,dA(z)
=
-\frac14
\partial_\vartheta
\left(
    \left|
        \varepsilon e^{i\vartheta}
        +\varepsilon^\alpha
    \right|^2
    -1
\right).
\end{align*}
Since
$$
    \left|
        \varepsilon e^{i\vartheta}
        +\varepsilon^\alpha
    \right|^2
    =
    \varepsilon^2
    +\varepsilon^{2\alpha}
    +2\varepsilon^{\alpha+1}\cos\vartheta,
$$
we obtain the exact identity
\begin{equation}
\label{exact_I1_F2}
     \mathcal I_{2,1}(\varepsilon,0)(\vartheta)
    =
    \frac{\varepsilon^{\alpha}}{2}
    \sin\vartheta.
\end{equation}

Concerning $ \mathcal I_{2,2}(\varepsilon,0)(\vartheta)$, we write 
$$
 \mathcal I_{2,2}(\varepsilon,0)(\vartheta)=\frac{\partial_\vartheta}{2\pi^2\varepsilon}
\sum_{j=2}^{\mathbf m} J_j(\varepsilon,\vartheta),
$$
where 
\begin{align*}
J_j(\varepsilon,\vartheta)
:=
\int_0^{2\pi}\int_0^1
\log\left|
\varepsilon e^{i\vartheta}
+\varepsilon^\alpha
-
e^{i\frac{2\pi(j-1)}{\mathbf m}}
\left(
    \varepsilon r e^{i\eta}
    +\varepsilon^\alpha
\right)
\right|
r\,dr\,d\eta.
\end{align*}
Let $j\in\{2,\ldots,\mathbf m\}$. The source integration domain is the
disk of radius $\varepsilon$ centered at
$$
    \varepsilon^\alpha
    e^{i\frac{2\pi(j-1)}{\mathbf m}}.
$$
The distance between the target point
$\varepsilon e^{i\vartheta}+\varepsilon^\alpha$ and this center satisfies
\begin{align*}
&
\left|
\varepsilon e^{i\vartheta}
+
\varepsilon^\alpha
-
\varepsilon^\alpha
e^{i\frac{2\pi(j-1)}{\mathbf m}}
\right|
\geq
\varepsilon^\alpha
\left|
1-e^{i\frac{2\pi(j-1)}{\mathbf m}}
\right|
-\varepsilon.
\end{align*}
Since $\alpha<1$, one has
$\varepsilon^{\alpha-1}\to+\infty$ as $\varepsilon\to0^+$. As there are only
finitely many indices $j\in\{2,\ldots,\mathbf m\}$, it follows that, for
sufficiently small $\varepsilon>0$,
$$
\left|
\varepsilon e^{i\vartheta}
+
\varepsilon^\alpha
-
\varepsilon^\alpha
e^{i\frac{2\pi(j-1)}{\mathbf m}}
\right|
>\varepsilon,
$$
uniformly in $\vartheta$ and $j$.
Consequently, the target point lies outside the corresponding closed source
disk. The logarithmic kernel is therefore harmonic throughout that disk, and
the mean-value property yields
\begin{equation}
\label{Jj_exact}
\begin{aligned}
J_j(\varepsilon,\vartheta)
&=
\pi
\log\left|
\varepsilon e^{i\vartheta}
+
\varepsilon^\alpha
\left(
    1-e^{i\frac{2\pi(j-1)}{\mathbf m}}
\right)
\right|, \qquad j\in\{2,\ldots,\mathbf m\}.
\end{aligned}
\end{equation}
Combining \eqref{J1_exact} and \eqref{Jj_exact}, we obtain
\begin{align}
\label{I2_reduced}
\varepsilon \mathcal I_{2,2}(\varepsilon,0)(\vartheta)
&=
\frac{1}{2\pi}
\partial_\vartheta
\Bigg[
    \log\varepsilon
    +
    \sum_{j=2}^{\mathbf m}
    \log\left|
        \varepsilon e^{i\vartheta}
        +
        \varepsilon^\alpha
        \left(
            1-e^{i\frac{2\pi(j-1)}{\mathbf m}}
        \right)
    \right|
\Bigg]
\nonumber\\
&=
\frac{1}{2\pi}
\partial_\vartheta
\sum_{j=2}^{\mathbf m}
\log\left|
1+\varepsilon^{1-\alpha}e^{i\vartheta}
-e^{i\frac{2\pi(j-1)}{\mathbf m}}
\right|.
\end{align}
All factors depending only on $\varepsilon$ disappear after differentiation
with respect to $\vartheta$.
Set
$$
    s
    :=
    \varepsilon^{1-\alpha}e^{i\vartheta}.
$$
The numbers
$$
    e^{i\frac{2\pi(j-1)}{\mathbf m}},
    \qquad j\in\{1,\ldots,\mathbf m\},
$$
are the $\mathbf m$-th roots of unity. Therefore,
$$
    \prod_{j=1}^{\mathbf m}
    \left(
        z-e^{i\frac{2\pi(j-1)}{\mathbf m}}
    \right)
    =
    z^{\mathbf m}-1.
$$
Removing the factor corresponding to $j=1$, for which the root is $1$, gives
$$
    \prod_{j=2}^{\mathbf m}
    \left(
        z-e^{i\frac{2\pi(j-1)}{\mathbf m}}
    \right)
    =
    \frac{z^{\mathbf m}-1}{z-1}.
$$
Taking $z=1+s$, we obtain
$$
    \prod_{j=2}^{\mathbf m}
    \left(
        1+s-e^{i\frac{2\pi(j-1)}{\mathbf m}}
    \right)
    =
    \frac{(1+s)^{\mathbf m}-1}{s}.
$$
It follows from \eqref{I2_reduced} that
\begin{align*}
\varepsilon\, \mathcal I_{2,2}(\varepsilon,0)(\vartheta)
&=
\frac{1}{2\pi}
\partial_\vartheta
\log\left|
    \frac{(1+s)^{\mathbf m}-1}{s}
\right|.
\end{align*}
Since
$$
    |s|=\varepsilon^{1-\alpha},
$$
is independent of $\vartheta$, one has
$$
    \partial_\vartheta\log|s|=0.
$$
Thus,
$$
\varepsilon\, \mathcal I_{2,2}(\varepsilon,0)(\vartheta)
=
\frac{1}{2\pi}
\partial_\vartheta
\log\left|
    (1+s)^{\mathbf m}-1
\right|.
$$
Using
$$
    \partial_\vartheta s=is,
$$
we obtain the exact identity
\begin{equation}
\label{exact_I2_F2}
\varepsilon\, \mathcal I_{2,2}(\varepsilon,0)(\vartheta)
=
\frac{1}{2\pi}
\operatorname{Re}
\left[
    \frac{
        i\mathbf m s(1+s)^{\mathbf m-1}
    }{
        (1+s)^{\mathbf m}-1
    }
\right],
\qquad
s=\varepsilon^{1-\alpha}e^{i\vartheta}.
\end{equation}
Substituting \eqref{exact_I1_F2} and \eqref{exact_I2_F2} into
\eqref{F2_at_trivial}, and dividing by $\varepsilon$, we obtain
\begin{align*}
\mathcal  F_2^\sigma(\varepsilon,\Lambda,0)(\vartheta)
&=
-\left[\Lambda
+
\frac{\Omega_0}{\varepsilon^\alpha}
-
\frac{\varepsilon^\alpha}{2}\right]\sin\vartheta
+
\frac{\varepsilon^{-\alpha}}{2\pi }
\operatorname{Re}
\left[
    \frac{
        i\mathbf m e^{i\vartheta}(1+\varepsilon^{1-\alpha}e^{i\vartheta})^{\mathbf m-1}
    }{
        (1+\varepsilon^{1-\alpha}e^{i\vartheta})^{\mathbf m}-1
    }
\right].
\end{align*}

Let us now expand the exact formula of $\mathcal F_2^\sigma(\varepsilon,\Lambda,0)$.
Define
$$
    G_{\mathbf m}(s)
    :=
    \frac{
        \mathbf m s(1+s)^{\mathbf m-1}
    }{
        (1+s)^{\mathbf m}-1
    }.
$$
The singularity of $G_{\mathbf m}$ at $s=0$ is removable, and
$G_{\mathbf m}(0)=1$. Hence $G_{\mathbf m}$ is analytic in a neighborhood of
the origin.
To compute its first coefficients, observe that
\begin{align*}
\frac{(1+s)^{\mathbf m}-1}{\mathbf m s}
&=
1
+
\frac{\mathbf m-1}{2}s
+
\frac{(\mathbf m-1)(\mathbf m-2)}{6}s^2
+
\frac{(\mathbf m-1)(\mathbf m-2)(\mathbf m-3)}{24}s^3
+
O\left(s^4\right),
\end{align*}
whereas
\begin{align*}
(1+s)^{\mathbf m-1}
&=
1
+
(\mathbf m-1)s
+
\frac{(\mathbf m-1)(\mathbf m-2)}{2}s^2
+
\frac{(\mathbf m-1)(\mathbf m-2)(\mathbf m-3)}{6}s^3
+
O\left(s^4\right).
\end{align*}
Dividing these two expansions gives
\begin{equation}
\label{expansion_Gm}
\begin{aligned}
G_{\mathbf m}(s)
&=
1
+
\frac{\mathbf m-1}{2}s
+
\frac{(\mathbf m-1)(\mathbf m-5)}{12}s^2
-
\frac{(\mathbf m-1)(\mathbf m-3)}{8}s^3
+
O\left(s^4\right).
\end{aligned}
\end{equation}
The remainder is uniform for $s$ in a sufficiently small fixed disk. Since
$$
    \operatorname{Re}\left(is^k\right)
    =
    -\varepsilon^{k(1-\alpha)}
    \sin(k\vartheta),
$$
equations \eqref{exact_I2_F2} and \eqref{expansion_Gm} imply
\begin{align*}
\mathcal  F_2^\sigma(\varepsilon,\Lambda,0)(\vartheta)
&=
-\left[\Lambda
+
\frac{1}{\varepsilon^\alpha}
\left(
    \Omega_0+\frac{\mathbf m-1}{4\pi}
\right)
-
\frac{\varepsilon^\alpha}{2}\right]\sin\vartheta
\\
&\quad +
\frac{(\mathbf m-1)(5-\mathbf m)}{24\pi}
\varepsilon^{1-2\alpha}\sin(2\vartheta)
+
O\left(
    \varepsilon^{2-3\alpha}
\right).
\end{align*}
With the choice
$$
    \Omega_0
    =
    \frac{1-\mathbf m}{4\pi},
$$
the singular first harmonic cancels. Hence
\begin{align}
\label{F2_m_not_5}
\mathcal  F_2^\sigma(\varepsilon,\Lambda,0)(\vartheta)
&=
-\Lambda\sin\vartheta
+
\frac{\varepsilon^\alpha}{2}\sin\vartheta
+
\frac{(\mathbf m-1)(5-\mathbf m)}{24\pi}
\varepsilon^{1-2\alpha}\sin(2\vartheta)
+
O\left(
    \varepsilon^{2-3\alpha}
\right).
\end{align}
If $\mathbf m\neq5$, the coefficient of the second harmonic is nonzero.
Under the assumption $0<\alpha<1/2$, one has
$$
    1-2\alpha>0,
    \qquad
    2-3\alpha>0.
$$
Therefore, all the terms following $-\Lambda\sin\vartheta$ in
\eqref{F2_m_not_5} converge uniformly to zero, and hence
$$
   \mathcal  F_2^\sigma(0,\Lambda,0)(\vartheta)
    =
    -\Lambda\sin\vartheta.
$$
When $\mathbf m=5$,  we obtain
\begin{align*}
\mathcal  F_2^\sigma(\varepsilon,\Lambda,0)(\vartheta)
&=
-\Lambda\sin\vartheta
+
\frac{\varepsilon^\alpha}{2}\sin\vartheta
+
\frac{1}{2\pi}
\varepsilon^{2-3\alpha}\sin(3\vartheta)
+
O\left(
    \varepsilon^{3-4\alpha}
\right).
\end{align*}
If $\alpha<2/3$, then
$$
   \mathcal  F_2^\sigma(0,\Lambda,0)(\vartheta)
    =
    -\Lambda\sin\vartheta.
$$
This completes the proof.
\end{proof}
\begin{rem}
The upper bounds on $\alpha$ obtained in Proposition \ref{F2_at_0} are sharp
for the present limiting configuration.
If $\mathbf m\neq5$ and $\alpha=\frac12$, then
$$
\lim_{\varepsilon\to0^+}
\mathcal  F_2^\sigma(\varepsilon,\Lambda,0)(\vartheta)
=
-\Lambda\sin\vartheta
+
\frac{(\mathbf m-1)(5-\mathbf m)}{24\pi}
\sin(2\vartheta).
$$
Thus the desired limiting identity fails at the critical exponent. For
$\alpha>\frac12$, the second harmonic is unbounded as
$\varepsilon\to0^+$.
If $\mathbf m=5$ and $\alpha=\frac23$, then
$$
\lim_{\varepsilon\to0^+}
\mathcal  F_2^\sigma(\varepsilon,\Lambda,0)(\vartheta)
=
-\Lambda\sin\vartheta
+
\frac{1}{2\pi}\sin(3\vartheta).
$$
For $\alpha>\frac23$, the third harmonic is unbounded. Hence the exceptional
pentagonal range $\alpha<\frac23$ is also sharp.
\end{rem}

\begin{cor}
\label{cor_trivial_sol}
Fix $\sigma\in[0,2)$, and set
$$
    \Omega_0
    =
    \frac{1-\mathbf m}{4\pi}.
$$
Assume that one of the following conditions holds:
$$
    \mathbf m\neq5
    \quad\text{and}\quad
    \frac{1}{\mathbf m+2}
    <
    \alpha
    <
    \frac12,
$$
or
$$
    \mathbf m=5
    \quad\text{and}\quad
    \frac17
    <
    \alpha
    <
    \frac23.
$$
Then, for every $\Lambda\in\mathbb R$,
$$
   \mathcal  F^\sigma(0,\Lambda,0)(\vartheta)
    =
    \bigl(
        0,
        -\Lambda\sin\vartheta
    \bigr), \qquad {\vartheta\in\mathbb T}.
$$
In particular,
$$
   \mathcal  F^\sigma(0,0,0){(\vartheta)}=0, \qquad {\vartheta\in\mathbb T}.
$$
\end{cor}

\begin{proof} The proof of the corollary follows immediately from 
 Propositions \ref{F1_at_0} and  \ref{F2_at_0}.
\end{proof}

\section{Functional framework and regularity of the nonlinear operator}\label{sction4}

From now on, we restrict attention to configurations that are symmetric with
respect to the real axis. Accordingly, we assume that
$$
    \overline{\mathcal D_1^\varepsilon}
    =
    \mathcal D_1^\varepsilon,
    \qquad
    \overline{\mathcal D_2^\varepsilon}
    =
    \mathcal D_2^\varepsilon,
$$
The family of inner components is invariant under
complex conjugation if and only if $\sigma\in\mathbb Z$. Consequently, up to
cyclic relabeling, the two reflection-symmetric configurations correspond to
$$
    \sigma=0
    \qquad\text{and}\qquad
    \sigma=1.
$$
The first is the roots-of-unity configuration, whereas the second is obtained
from it by a rotation through the angle $\pi/\mathbf{m}$. 

For each fixed $\sigma\in\{0,1\}$, we denote the corresponding nonlinear
operator by
$$
   \mathcal  F^\sigma
    =
    \bigl(\mathcal F_1^\sigma,\mathcal F_2^\sigma\bigr).
$$
When no confusion can arise, we suppress the superscript $\sigma$.

The existence proof will be based on an application of the implicit function
theorem to the equation
$$
   \mathcal  F^\sigma(\varepsilon,\Lambda,f)=0.
$$
We therefore introduce symmetry-adapted Banach spaces for the boundary
perturbations and establish the regularity properties of the nonlinear
operator on a neighborhood of the trivial configuration.
\subsection{Function spaces and well-posedness of the nonlinear operator}\label{section_spaces}
Let $\nu\in(0,1)$. We define
\begin{align*}
\mathcal X_{1,\mathbf m}
:=
\Bigg\{
    f\in C^{1,\nu}(\mathbb T):
    \;&f(-\vartheta)=f(\vartheta),\;\;
    f\left(\vartheta+\frac{2\pi}{\mathbf m}\right)=f(\vartheta), \;\;\int_0^{2\pi}f(\vartheta)\,d\vartheta=0
\Bigg\},
\end{align*}
and
\begin{align*}
\mathcal X_2
:=
\Bigg\{
    f\in C^{1,\nu}(\mathbb T):
    \;&f(-\vartheta)=f(\vartheta), \;\;\int_0^{2\pi}f(\vartheta)\,d\vartheta=0,\;\;\int_0^{2\pi}f(\vartheta)\cos\vartheta\,d\vartheta=0
\Bigg\}.
\end{align*}
Similarly, we set
\begin{align*}
\mathcal Y_{1,\mathbf m}
:=
\Bigg\{
    g\in C^{0,\nu}(\mathbb T):
    \;&g(-\vartheta)=-g(\vartheta),\;\; g\left(\vartheta+\frac{2\pi}{\mathbf m}\right)=g(\vartheta)
\Bigg\},
\end{align*}
and
$$
\mathcal Y_2
:=
\left\{
    g\in C^{0,\nu}(\mathbb T):
    g(-\vartheta)=-g(\vartheta)
\right\}.
$$
The corresponding Fourier expansions are
\begin{align*}
    f_1(\vartheta)
    &=
    \sum_{k=1}^{\infty}
    f_{1,k}\cos(k\mathbf m\vartheta),
  \qquad
    f_1\in\mathcal X_{1,\mathbf m},
    \\
    f_2(\vartheta)
    &=
    \sum_{n=2}^{\infty}
    f_{2,n}\cos(n\vartheta),
   \qquad
    f_2\in\mathcal X_2,
    \\
    g_1(\vartheta)
    &=
    \sum_{k=1}^{\infty}
    g_{1,k}\sin(k\mathbf m\vartheta),
   \qquad
    g_1\in\mathcal Y_{1,\mathbf m},
    \\
    g_2(\vartheta)
    &=
    \sum_{n=1}^{\infty}
    g_{2,n}\sin(n\vartheta),
   \qquad
    g_2\in\mathcal Y_2,
\end{align*}
where $f_{1,k}$, $f_{2,n}$, $g_{1,k}$ and $g_{2,n}$ are real numbers. It is worth pointing out that the zero-average conditions preserve the areas of the generating domains. The absence of the first cosine mode in $\mathcal X_2$ removes the translational mode from the linearized inner-boundary problem.
We then define
$$
    X
    :=
    \mathcal X_{1,\mathbf m}\times\mathcal X_2,
    \qquad
    Y
    :=
    \mathcal Y_{1,\mathbf m}\times\mathcal Y_2,
$$
endowed with the norms
$$
    \|(f_1,f_2)\|_X
    :=
    \|f_1\|_{C^{1,\nu}}
    +
    \|f_2\|_{C^{1,\nu}},
$$
and
$$
    \|(g_1,g_2)\|_Y
    :=
    \|g_1\|_{C^{0,\nu}}
    +
    \|g_2\|_{C^{0,\nu}}.
$$
These are closed subspaces of the corresponding H\"older spaces and are
therefore Banach spaces.

We denote by $B_X(0,r)$ the ball of radius $r>0$ and center $0$ in $X$,
$$
    B_X(0,r)
    :=
    \left\{
        f\in X:\|f\|_X<r
    \right\}.
$$
Let $r>0$ and $\Lambda_0>0$, and define the open neighborhood
$$
    U
    :=
    (-\Lambda_0,\Lambda_0)\times B_X(0,r)
    \subset \R\times X.
$$
Since the nonlinear operator contains the fractional power
$\varepsilon^\alpha$, we regard $\varepsilon$ as a one-sided parameter and
work on the metric space
$
    [0,\varepsilon_0).
$
We use a continuously parameter-dependent implicit function theorem with a metric parameter
space; more precisely, we apply the specialization to single-valued maps
between Banach spaces of
\cite[Theorem~5F.4, p.~305]{DontchevRockafellar2014}.
This approach does not require extending the nonlinear operator to negative
values of $\varepsilon$.

\begin{rem}[One-sided parameter and auxiliary even extension]\label{remark-IFT}
The parameter $\varepsilon$ has a geometric meaning only for
$\varepsilon\geq0$: it determines the size of the inner components, whereas
$\varepsilon^\alpha$ determines their distance from the origin. Accordingly,
we formulate the implicit-function argument on the one-sided parameter space $
    [0,\varepsilon_0).
$
After constructing the continuous extension of $\mathcal{F}^\sigma$ to
$\varepsilon=0$, one could define an auxiliary even extension by
$$
     F^\sigma(\varepsilon,\Lambda,f)
    :=
    \mathcal F^\sigma(|\varepsilon|,\Lambda,f),
    \qquad
    |\varepsilon|<\varepsilon_0.
$$
This extension does not provide additional physical solutions. Indeed, local
uniqueness in the implicit function theorem implies that the corresponding
solution satisfies
$$
    \Lambda(-\varepsilon)=\Lambda(\varepsilon),
    \qquad
    f(-\varepsilon)=f(\varepsilon).
$$
Hence the parameters $\varepsilon$ and $-\varepsilon$ determine the same
vortex-patch configuration. The two-sided curve is therefore only a symmetric
double parametrization of the physical one-sided branch emerging from the
singular point-vortex configuration.
\end{rem}

Throughout the remainder of this section, for each fixed
$\sigma\in\{0,1\}$, we shall prove that there exists
$\varepsilon_0>0$ sufficiently small such that the nonlinear operator
$   \mathcal  F^\sigma$ can be extended from
    $[0,\varepsilon_0)\times U$ to
     $Y$
as a continuous function in all its variables and continuously Fréchet differentiable with respect to the unknowns $(\Lambda,f)\in U$. 

We begin by verifying that the nonlinear operator $\mathcal F^\sigma$ preserves the
symmetries encoded in the spaces $X$ and $Y$.

\begin{lem}[Symmetries {of $\mathcal{F}^\sigma$}]
\label{lemma:symmetries_F}
Fix $\sigma\in\{0,1\}$, and let
$(\varepsilon, \Lambda, f) \in \R \times \R \times X$. Then
$$
   \mathcal F_j^\sigma(\varepsilon,\Lambda,f)(-\vartheta)
    =
    -\mathcal F_j^\sigma(\varepsilon,\Lambda,f)(\vartheta),
    \qquad j\in\{1,2\}, \qquad {\vartheta \in \T,}
$$
and
$$
   \mathcal  F_1^\sigma(\varepsilon,\Lambda,f)
    \left(\vartheta+\frac{2\pi}{\mathbf m}\right)
    =\mathcal
    F_1^\sigma(\varepsilon,\Lambda,f)(\vartheta), \qquad {\vartheta \in \T,}
$$
\end{lem}

\begin{proof}
Since $f_1$ and $f_2$ are even, the generating domains
$\mathcal D_1^\varepsilon$ and $\mathcal D_2^\varepsilon$ are symmetric with
respect to the real axis. Moreover, because $\sigma\in\{0,1\}$, complex
conjugation permutes the family of inner components. Hence the vorticity and
the associated stream function satisfy
$$
    \Psi_\varepsilon\left(\overline z\right)
    =
    \Psi_\varepsilon(z),
$$
and, by $\mathbf m$-fold rotational symmetry,
$$
    \Psi_\varepsilon
    \left(
        e^{i\frac{2\pi}{\mathbf m}}z
    \right)
    =
    \Psi_\varepsilon(z).
$$
Recall that
$$
    z_j(\vartheta)
    =
    w_j(\vartheta)e^{i\vartheta}, \quad j\in\{1,2\}.
$$
For the outer boundary, 
since $w_1$ is even,
$$
    z_1(-\vartheta)
    =
    \overline{z_1(\vartheta)}.
$$
Therefore,
$$
    \vartheta
    \longmapsto
    \Psi_\varepsilon
    \bigl(z_1(\vartheta)\bigr),
$$
is even, and its derivative is odd. In addition,
$$
    z_1(\vartheta)\cdot z_1'(\vartheta)
    =
    \varepsilon f_1'(\vartheta),
$$
is odd. It follows from the definition of $\mathcal F_1^\sigma$, given by \eqref{F1}, that
$$
    \mathcal F_1^\sigma(-\vartheta)
    =
    -\mathcal F_1^\sigma(\vartheta).
$$

For the representative inner component, write
$$
    \gamma_2(\vartheta)
    :=
    e^{i\frac{\pi\sigma}{\mathbf m}}
    \left(
        \varepsilon z_2(\vartheta)
        +
        \varepsilon^\alpha\mathtt e_1
    \right).
$$
Since $z_2(-\vartheta)=\overline{z_2(\vartheta)}$, one has
$$
    \gamma_2(-\vartheta)
    =
    e^{i\frac{2\pi\sigma}{\mathbf m}}
    \overline{\gamma_2(\vartheta)}.
$$
The map
$$
    z
    \longmapsto
    e^{i\frac{2\pi\sigma}{\mathbf m}}\overline z,
$$
is a symmetry of the complete configuration. Consequently,
$$
    \Psi_\varepsilon
    \bigl(\gamma_2(-\vartheta)\bigr)
    =
    \Psi_\varepsilon
    \bigl(\gamma_2(\vartheta)\bigr).
$$
Thus the stream-function contribution to $\mathcal F_2^\sigma$, given by \eqref{F2}, is odd after
differentiation with respect to $\vartheta$.

Finally,
$$
    z_2(\vartheta)\cdot z_2'(\vartheta)
    =
    \varepsilon f_2'(\vartheta),
$$
is odd, and
$$
    \bigl(z_2'(\vartheta)\bigr)_1
    =
    \varepsilon
    w_2(\vartheta)^{-1}
    f_2'(\vartheta)\cos\vartheta
    -
    w_2(\vartheta)\sin\vartheta,
$$
is also odd. Hence
$$
   \mathcal  F_2^\sigma(-\vartheta)
    =
    -\mathcal F_2^\sigma(\vartheta).
$$

To prove the periodicity of the first component, observe that
$$
    z_1\left(\vartheta+\frac{2\pi}{\mathbf m}\right)
    =
    e^{i\frac{2\pi}{\mathbf m}}z_1(\vartheta).
$$
The rotational invariance of the stream function, together with the
$\frac{2\pi}{\mathbf m}$-periodicity of $f_1'$, yields
$$
  \mathcal F_1^\sigma
    \left(\vartheta+\frac{2\pi}{\mathbf m}\right)
    =\mathcal 
    F_1^\sigma(\vartheta).
$$
This concludes the proof of the Lemma.
\end{proof}

Next, we establish the well-definedness and continuity of the nonlinear
operator for positive values of the parameter $\varepsilon$ as well as the extension to the singular value $\varepsilon=0$.  As will become apparent below, the periodic Hilbert transform naturally arises at this limiting value. For the reader's convenience, we first recall its
definition and its action on Fourier modes.
For $h\in C^{1,\nu}(\mathbb T)$, we define the periodic Hilbert transform by
\begin{equation}
\label{definition_Hilbert_transform}
    \mathcal H[h](\vartheta)
    :=
    \frac{1}{2\pi}
    \operatorname{p.v.}
    \int_0^{2\pi}
    \cot\left(
        \frac{\vartheta-\eta}{2}
    \right)
    h(\eta)\,d\eta.
\end{equation}
With this convention, for every $n\in\mathbb N^*$,
\begin{equation}
\label{Hilbert_Fourier_action}
    \mathcal H[\cos(n\vartheta)]
    =
    \sin(n\vartheta),
    \qquad
    \mathcal H[\sin(n\vartheta)]
    =
    -\cos(n\vartheta).
\end{equation}

\begin{pro}
\label{Fwelldef}
Fix $\sigma\in\{0,1\}$, and assume that
\[
    \frac{1}{\mathbf m+2}
    <
    \alpha
    <
    \frac12,
    \qquad
    \Omega_0
    =
    \frac{1-\mathbf m}{4\pi}.
\]
 There exist $\varepsilon_0>0$,
$\Lambda_0>0$, and $r>0$ 
such that the nonlinear operator
$   \mathcal  F^\sigma$ can be extended from
    $[0,\varepsilon_0)\times U$ to
     $Y$
as a continuous, with
\begin{equation}
\label{definition_F_at_zero}
\mathcal F^\sigma(0,\Lambda,f)
=
\left(
    \Omega_0 f_1',
     \frac{1}{2\pi}
        f_2'
        +
        \frac{1}{2\pi}\mathcal H[f_2]-\Lambda\sin(\cdot)
\right).
\end{equation}
\end{pro}

\begin{proof}
Throughout the proof, the phase $\sigma\in\{0,1\}$ is fixed and is
occasionally suppressed from the notation. All constants denoted by $C>0$
may change from one line to another, but they are uniform for
$(\Lambda,f)\in U$ and $\varepsilon\in[0,\varepsilon_0)$, unless otherwise
specified.

In view of Lemma \ref{lemma:symmetries_F}, it is enough to show that both
components of $\mathcal  F^\sigma$ belong to $C^{0,\nu}(\mathbb T)$, depend continuously on $(\varepsilon,\Lambda,f)$, and admit continuous extensions to the singular value $\varepsilon=0$. 
We begin with the first component $\mathcal F_1^\sigma$, given by \eqref{F12}. Its local contribution is
$$
\left(
    \Omega_0+\varepsilon^\alpha\Lambda
\right)f_1',
$$
which belongs to $C^{0,\nu}(\mathbb T)$. Moreover,
\begin{equation}
\label{local_F1_limit}
\left\|
    \left(
        \Omega_0+\varepsilon^\alpha\Lambda
    \right)f_1'
    -
    \Omega_0f_1'
\right\|_{C^{0,\nu}}
\leq
C\varepsilon^\alpha.
\end{equation}
We next consider the outer self-interaction
$$
\begin{aligned}
\mathcal S_1(\varepsilon,f_1)(\vartheta)
=-
\frac{\varepsilon^{2\alpha-1}}{2\pi}
\partial_\vartheta
\int_0^{2\pi}\int_0^{w_1(\eta)}
\log\left|
    w_1(\vartheta)e^{i\vartheta}
    -
    r e^{i\eta}
\right|
r\,dr\,d\eta.
\end{aligned}
$$
After the change of variables
$$
    r=\rho w_1(\eta),
$$
we obtain
$$
\begin{aligned}
-\mathcal S_1(\varepsilon,f_1)(\vartheta)
=
\frac{\varepsilon^{2\alpha-1}}{2\pi}
\int_0^{2\pi}\int_0^1
\frac{
    \operatorname{Re}
    \left[
        \overline{\Delta(\rho,\vartheta,\eta)}
        z_1'(\vartheta)
    \right]
}{
    |\Delta(\rho,\vartheta,\eta)|^2
}
\rho w_1(\eta)^2
\,d\rho\,d\eta,
\end{aligned}
$$
where
$$
    z_1(\vartheta)
    =
    w_1(\vartheta)e^{i\vartheta},
    \qquad
    \Delta(\rho,\vartheta,\eta)
    :=
    z_1(\vartheta)-\rho z_1(\eta).
$$
Using
$$
    w_1(\eta)^2
    =
    w_1(\eta)^2-w_1(\vartheta)^2
    +
    w_1(\vartheta)^2
$$
and
$$
    w_1(\eta)^2-w_1(\vartheta)^2
    =
    2\varepsilon
    \bigl(
        f_1(\eta)-f_1(\vartheta)
    \bigr),
$$
we obtain
$$
\begin{aligned}
-\mathcal S_1(\varepsilon,f_1)(\vartheta)
&=
\frac{\varepsilon^{2\alpha}}{\pi}
\int_0^{2\pi}\int_0^1
\frac{
    \operatorname{Re}
    \left[
        \overline{\Delta(\rho,\vartheta,\eta)}
        z_1'(\vartheta)
    \right]
}{
    |\Delta(\rho,\vartheta,\eta)|^2
}
\rho
\bigl(
    f_1(\eta)-f_1(\vartheta)
\bigr)
\,d\rho\,d\eta
\\
&\quad
+
\frac{\varepsilon^{2\alpha-1}}{2\pi}
w_1(\vartheta)^2
\int_0^{2\pi}\int_0^1
\frac{
    \operatorname{Re}
    \left[
        \overline{\Delta(\rho,\vartheta,\eta)}
        z_1'(\vartheta)
    \right]
}{
    |\Delta(\rho,\vartheta,\eta)|^2
}
\rho\,d\rho\,d\eta.
\end{aligned}
$$
For every $\rho\in[0,1)$, periodicity in $\eta$ gives
$$
\begin{aligned}
\int_0^{2\pi}
\frac{
    \operatorname{Re}
    \left[
        \overline{\Delta(\rho,\vartheta,\eta)}
        \rho z_1'(\eta)
    \right]
}{
    |\Delta(\rho,\vartheta,\eta)|^2
}
\,d\eta
&=
-
\int_0^{2\pi}
\partial_\eta
\log|\Delta(\rho,\vartheta,\eta)|
\,d\eta
&=0.
\end{aligned}
$$
The identity may first be applied for $\rho\leq1-\delta$ and then extended
to $\rho=1$ by letting $\delta\to0^+$.
Set
$$
    p_1(\vartheta)
    :=
    \frac{f_1'(\vartheta)}{w_1(\vartheta)}
    e^{i\vartheta}.
$$
Since
$$
    z_1'(\vartheta)
    =
    i z_1(\vartheta)
    +
    \varepsilon p_1(\vartheta),
$$
we have
$$
    z_1'(\vartheta)-\rho z_1'(\eta)
    =
    i\Delta(\rho,\vartheta,\eta)
    +
    \varepsilon
    \bigl(
        p_1(\vartheta)-\rho p_1(\eta)
    \bigr).
$$
Moreover,
$$
    \operatorname{Re}
    \left[
        i|\Delta(\rho,\vartheta,\eta)|^2
    \right]
    =0.
$$
Consequently,
$$
\begin{aligned}
-\mathcal S_1(\varepsilon,f_1)(\vartheta)
&=
\frac{\varepsilon^{2\alpha}}{\pi}
\operatorname{Re}
\left\{
    z_1'(\vartheta)
    \int_0^{2\pi}\int_0^1
    \frac{\rho}{
        \Delta(\rho,\vartheta,\eta)
    }
    \bigl(
        f_1(\eta)-f_1(\vartheta)
    \bigr)
    \,d\rho\,d\eta
\right\}
\\
&\quad
+
\frac{\varepsilon^{2\alpha}}{2\pi}
w_1(\vartheta)^2
\operatorname{Re}
\left\{
    p_1(\vartheta)
    \int_0^{2\pi}\int_0^1
    \frac{\rho}{
        \Delta(\rho,\vartheta,\eta)
    }
    \,d\rho\,d\eta
    -
    \int_0^{2\pi}\int_0^1
    \frac{\rho^2 p_1(\eta)}{
        \Delta(\rho,\vartheta,\eta)
    }
    \,d\rho\,d\eta
\right\}.
\end{aligned}
$$
This identity exposes the factor $\varepsilon^{2\alpha}$ and removes the apparent singularity.
Define, for $k\in\{1,2\}$,
$$
    L_k(\rho,\vartheta,\eta)
    :=
    \frac{\rho^k}{
        z_1(\vartheta)-\rho z_1(\eta)
    }.
$$
After reducing $\varepsilon_0$ and $r$, there exists $c>0$ such that
$$
\left|
    z_1(\vartheta)-\rho z_1(\eta)
\right|^2
\geq
c\left[
    (1-\rho)^2
    +
    4\rho
        \sin^2\left(
            \frac{\vartheta-\eta}{2}
        \right)
\right].
$$
Using  the elementary inequalities
$$
    \frac{1}{(a^2+b^2)^{1/2}}
    \leq
    \frac{C}{
        a^\nu b^{1-\nu}
    },
\qquad 
    \frac{1}{a^2+b^2}
    \leq
    \frac{C}{
        a^\nu b^{2-\nu}
    },
    \qquad a,b>0,
$$
we obtain
\begin{equation*}
    |L_{k}(\rho,\vartheta,\eta)|
    \leq
    C\left|
        \sin\left(
            \frac{\vartheta-\eta}{2}
        \right)
    \right|^{\nu-1}
    \frac{
        \rho^{(1+\nu)/2}
    }{
        (1-\rho)^\nu
    }.
\end{equation*}
Moreover,
$$
    \partial_\vartheta L_{k}
    =
    -
    \frac{
        \rho^k z_1'(\vartheta)
    }{
        \Delta(\rho,\vartheta,\eta)^2
    },
$$
and hence
\begin{equation*}
    |\partial_\vartheta
    L_{k}(\rho,\vartheta,\eta)|
    \leq
    C
    \left|
        \sin\left(
            \frac{\vartheta-\eta}{2}
        \right)
    \right|^{\nu-2}
    \frac{
        \rho^{\nu/2}
    }{
        (1-\rho)^\nu
    }.
\end{equation*}
Since
$$
    \rho
    \longmapsto
    \frac{\rho^{(1+\nu)/2}}{(1-\rho)^\nu},
    \qquad
    \rho
    \longmapsto
    \frac{\rho^{\nu/2}}{(1-\rho)^\nu}
$$
belong to $L^1(0,1)$, Lemma \ref{lemmasurvey} applies to the real and
imaginary parts of $L_{k}$. Consequently, for $k\in\{1,2\}$, the operators
$$
    h
    \longmapsto
    \int_{\mathbb T}\int_0^1
    L_{k}(\rho,\vartheta,\eta)
    h(\eta)
    \,d\rho\,d\eta,
$$
map $L^\infty(\mathbb T)$ continuously into
$C^{0,\nu}(\mathbb T)$, with bounds uniform on the prescribed
neighborhood. Since
$$
    z_1',
    \quad
    p_1,
    \quad
    w_1^2,
    \quad
    f_1
    \in C^{0,\nu}(\mathbb T),
$$
the preceding representation yields
\begin{equation}\label{outer_self_extension_bound}
    \|\mathcal S_1(\varepsilon,f_1)\|_{C^{0,\nu}}
    \leq
    C\varepsilon^{2\alpha}.
\end{equation}
Let's now study the interaction between the outer boundary and the
inner components. 
Since $f_2$ has zero average,
\[
    |\mathcal D_{2}^{\varepsilon}|
    =
    \pi.
\]
After writing the inner integrals as integrals over
$\mathcal D_{2}^{\varepsilon}$, given by \eqref{Djeps}, their contribution to $\mathcal F_1^\sigma$ is
\begin{align}
\label{outer_inner_F1}
 \nonumber \mathcal I_{1,2}(\varepsilon,f)(\vartheta)
&=\frac{\varepsilon^{2\alpha-1}}{2 \pi^2} \sum_{j = 1}^{\mathbf{m}} \partial_\vartheta\int_{0}^{2 \pi} \int_{0}^{w_2(\eta)} \log \left|w_1(\vartheta) e^{i \vartheta} - e^{i \frac{\pi(2j-2+\sigma)}{\mathbf m}} \left(\varepsilon r e^{i \eta} + \varepsilon^\alpha \right)\right|  r \, dr \, d\eta \\ 
&=
\frac{\varepsilon^{2\alpha-1}}{2\pi^2}
\partial_\vartheta
\sum_{j=1}^{\mathbf m}
\int_{\mathcal D_{2}^{\varepsilon}}
\log\left|
    z_{1}(\vartheta)
    -
    \varepsilon^\alpha e^{i\phi_j}
    -
    \varepsilon e^{i\phi_j}y
\right|
\,dA(y),
\end{align}
where
$$
    \phi_j
    =
    \frac{\pi(2j-2+\sigma)}{\mathbf m},
    \qquad
    j\in\{1,\ldots,\mathbf m\}.
$$

The outer boundary is uniformly separated from the inner components.
Therefore, Taylor's formula in the variable $\varepsilon y$ gives
\begin{align}
\label{point_mass_reduction_F1}
&
\int_{\mathcal D_{2}^{\varepsilon}}
\log\left|
    z_{1}(\vartheta)
    -
    \varepsilon^\alpha e^{i\phi_j}
    -
    \varepsilon e^{i\phi_j}y
\right|
\,dA(y)=
\pi
\log\left|
    z_{1}(\vartheta)
    -
    \varepsilon^\alpha e^{i\phi_j}
\right|
+
\mathcal E_{j,\varepsilon}(\vartheta),
\end{align}
where
\begin{equation}
\label{error_point_mass_F1}
\left\|
    \partial_\vartheta
    \mathcal E_{j,\varepsilon}
\right\|_{C^{0,\nu}}
\leq
C\varepsilon.
\end{equation}
It follows that the contribution of the errors in
\eqref{point_mass_reduction_F1} is bounded by
$C\varepsilon^{2\alpha}$.
 Since the points $e^{i\phi_j}$ are the roots of
    $z^{\mathbf m}=e^{i\pi\sigma}$, then
\[
\prod_{j=1}^{\mathbf m}
\left(
    z-\varepsilon^\alpha e^{i\phi_j}
\right)
=
z^{\mathbf m}
-
\varepsilon^{\mathbf m\alpha}e^{i\pi\sigma}.
\]
Therefore, the point-mass part of \eqref{outer_inner_F1} is
\begin{align*}
&
\frac{\varepsilon^{2\alpha-1}}{2\pi}
\partial_\vartheta
\log\left|
    z_{1}(\vartheta)^{\mathbf m}
    -
    \varepsilon^{\mathbf m\alpha}e^{i\pi\sigma}
\right|
=
\frac{\varepsilon^{2\alpha-1}}{2\pi}
\partial_\vartheta
\left[
    \mathbf m\log w_{1}(\vartheta)
    +
    \log\left|
        1-
        \varepsilon^{\mathbf m\alpha}
        e^{i\pi\sigma}
        z_{1}(\vartheta)^{-\mathbf m}
    \right|
\right].
\end{align*}
Since
\[
    \frac{w_{1}'}{w_{1}}
    =
    \varepsilon
    \frac{f_1'}{w_{1}^2},
\]
the first term is bounded in $C^{0,\nu}$ by
$C\varepsilon^{2\alpha}$. The second one is bounded by
\[
    C
    \varepsilon^{2\alpha-1}
    \varepsilon^{\mathbf m\alpha}
    =
    C
    \varepsilon^{(\mathbf m+2)\alpha-1}.
\]
Consequently,
\begin{equation}
\label{outer_inner_F1_final_bound}
\left\|
    \mathcal I_{1,2}(\varepsilon,f)
\right\|_{C^{0,\nu}}
\leq
C\left(
    \varepsilon^{2\alpha}
    +
    \varepsilon^{(\mathbf m+2)\alpha-1}
\right).
\end{equation}
Combining \eqref{local_F1_limit},
\eqref{outer_self_extension_bound}, and
\eqref{outer_inner_F1_final_bound}, we obtain
\begin{align}
\label{estimate_F1_extension}
\left\|
  \mathcal  F_1^\sigma(\varepsilon,\Lambda,f)
    -
    \Omega_0f_1'
\right\|_{C^{0,\nu}}
&\leq
C\left(
    \varepsilon^\alpha
    +
    \varepsilon^{2\alpha}
    +
    \varepsilon^{(\mathbf m+2)\alpha-1}
\right).
\end{align}

Next, we shall deal with the second component $\mathcal F_2^\sigma$, given by \eqref{F22}. We first identify the contribution of the inner component interacting with
itself. For $t\in[0,\varepsilon_0)$,  set
\[
\mathcal J_t(\vartheta)
:=
\int_0^{2\pi}
\int_0^{ \sqrt{1+2t f_2(\eta)}}
\log\left|
    \sqrt{1+2t f_2(\vartheta)} e^{i\vartheta}-re^{i\eta}
\right|
r\,dr\,d\eta.
\]
The self-interaction contribution to $\mathcal F_2^\sigma$ is
\[
  \mathcal S_2(\varepsilon,f_2)=  \frac{1}{2\pi^2\varepsilon}
    \partial_\vartheta\mathcal J_\varepsilon.
\]
Since $\mathcal J_0$ is constant in $\vartheta$, this can be written as
\[
\begin{aligned}
\mathcal S_2(\varepsilon,f_2)
&=
\frac{1}{2\pi^2\varepsilon}
\partial_\vartheta
\left(
    \mathcal J_\varepsilon-\mathcal J_0
\right)
\\
&=
\frac{1}{2\pi^2}
\int_0^1
\partial_\vartheta
\dot{\mathcal J}_{s\varepsilon}
\,ds.
\end{aligned}
\]
Differentiating $\mathcal J_t$ with respect to $t$ gives
\begin{align}
\label{shape_derivative_self}
\dot{\mathcal J}_t(\vartheta)
&=
\int_0^{2\pi}
f_2(\eta)
\log\left|
     \sqrt{1+2t f_2(\vartheta)} e^{i\vartheta}- \sqrt{1+2t f_2(\eta)} e^{i\eta}
\right|
\,d\eta
\nonumber\\
&\quad
+
\frac{f_2(\vartheta)}{ \sqrt{1+2t f_2(\vartheta)} } \int_0^{2\pi}
\int_0^{ \sqrt{1+2t f_2(\eta)} }
\operatorname{Re}
\left[
    \frac{e^{i\vartheta}}{
         \sqrt{1+2t f_2(\vartheta)} e^{i\vartheta}-re^{i\eta}
    } 
\right]
r\,dr\,d\eta,
\end{align}
The same chord--arc and kernel estimates used for $\mathcal S_1(\varepsilon,f_1)$  imply that
\[
    t
    \longmapsto
    \partial_\vartheta\dot{\mathcal J}_t
\]
is continuous from a neighborhood of $0$ into
$C^{0,\nu}(\mathbb T)$. More precisely,
\begin{equation}
\label{continuity_shape_derivative_self}
\left\|
    \partial_\vartheta\dot{\mathcal J}_t
    -
    \partial_\vartheta\dot{\mathcal J}_0
\right\|_{C^{0,\nu}}
\leq
Ct.
\end{equation}
At $t=0$, formula \eqref{shape_derivative_self} becomes
\begin{align*}
\dot{\mathcal J}_0(\vartheta)
&=
\int_0^{2\pi}
f_2(\eta)
\log\left|
    e^{i\vartheta}-e^{i\eta}
\right|
\,d\eta
+
f_2(\vartheta)
\int_0^{2\pi}
\int_0^1
\operatorname{Re}
\left[
    \frac{
        e^{i\vartheta}
    }{
        e^{i\vartheta}-re^{i\eta}
    }
\right]
r\,dr\,d\eta.
\end{align*}
By Lemma \ref{integral_double},
\[
\int_0^{2\pi}
\int_0^1
\operatorname{Re}
\left[
    \frac{
        e^{i\vartheta}
    }{
        e^{i\vartheta}-re^{i\eta}
    }
\right]
r\,dr\,d\eta
=
\pi.
\]
Therefore,
\[
\dot{\mathcal J}_0(\vartheta)
=
\int_0^{2\pi}
\log\left|
    e^{i\vartheta}-e^{i\eta}
\right|
f_2(\eta)\,d\eta
+
\pi f_2(\vartheta).
\]
It follows from \eqref{continuity_shape_derivative_self} that
\begin{equation}
\label{self_interaction_limit}
\left\|
    \mathcal S_2(\varepsilon,f_2)
    -\frac{1}{2\pi^2} \partial_\vartheta
    \int_0^{2\pi}
\log\left|
    e^{i\vartheta}-e^{i\eta}
\right|
f_2(\eta)\,d\eta
-\frac{1}{2\pi} 
 f_2'(\vartheta)
\right\|_{C^{0,\nu}}
\leq
C\varepsilon.
\end{equation}
Let us now discuss the contribution of the outer domain to the second component. It can be written as
\begin{align*}
    \mathcal I_{2,1}(\varepsilon,f)(\vartheta)
=
-\frac{1}{2\pi\varepsilon}
\partial_\vartheta
\int_{\mathcal D_{1}^{\varepsilon}}
\log\left|
    A_\varepsilon^\sigma(\vartheta)-y
\right|
\,dA(y).
\end{align*}
where
\[
    A_\varepsilon^\sigma(\vartheta)
    :=
    e^{i\frac{\pi\sigma}{\mathbf m}}
    \left(
        \varepsilon z_{2}(\vartheta)
        +
        \varepsilon^\alpha
    \right).
\]
We decompose
\[
    \mathcal D_{1}^{\varepsilon}
    =
    \mathbb D
    +
    \left(
        \mathcal D_{1}^{\varepsilon}-\mathbb D
    \right),
\]
where the second term is understood as a signed radial-domain difference.
For the unit disk, Lemma \ref{integral_disk} gives
\[
\begin{aligned}
&
-\frac{1}{2\pi\varepsilon}
\partial_\vartheta
\int_{\mathbb D}
\log\left|
    A_\varepsilon^\sigma(\vartheta)-y
\right|
\,dA(y)
=
-\frac{1}{4\varepsilon}
\partial_\vartheta
\left|
    \varepsilon^\alpha
    +
    \varepsilon z_{2}(\vartheta)
\right|^2.
\end{aligned}
\]
Consequently, this term is bounded in $C^{0,\nu}$ by
$C\varepsilon^\alpha$.
 The radial thickness of
$\mathcal D_{1}^{\varepsilon}-\mathbb D$ is $O(\varepsilon)$, while
\[
    \left|
        \partial_\vartheta
        A_\varepsilon^\sigma(\vartheta)
    \right|
    \leq C\varepsilon.
\]
Since the inner component remains uniformly separated from the outer
boundary, the contribution of the signed-domain difference is bounded by
$C\varepsilon$. Hence,
\begin{equation}
\label{outer_contribution_F2}
\left\|
    \mathcal I_{2,1}(\varepsilon,f)
\right\|_{C^{0,\nu}}
\leq
C\left(
    \varepsilon^\alpha+\varepsilon
\right).
\end{equation}
The interaction of the representative inner component with the remaining
components is
\begin{align}
\label{distinct_inner_interaction}
\mathcal I_{2,2}(\varepsilon,f_2)(\vartheta)
&=
\frac{1}{2\pi^2\varepsilon}
\partial_\vartheta
\sum_{j=2}^{\mathbf m}
\int_{\mathcal D_{2,\varepsilon}}
\log\left|
    1-\mathcal R_j
    +
    \varepsilon^{1-\alpha}
    \left(
        z_2(\vartheta)-\mathcal R_jy
    \right)
\right|
\,dA(y),
\end{align}
where \[
   \mathcal R_j
    =
    e^{i\frac{2\pi(j-1)}{\mathbf m}},
    \qquad
    j\in\{1,\ldots,\mathbf m\},
\]
Since $1-\mathcal R_j\neq0$ for $j\geq2$, Taylor's formula yields
\begin{align}
\label{polygon_taylor}
&
\log\left|
    1-\mathcal R_j
    +
    \varepsilon^{1-\alpha}
    \left(
        z_2(\vartheta)-\mathcal R_jy
    \right)
\right|
\nonumber\\
&\qquad
=
\log|1-\mathcal R_j|
+
\varepsilon^{1-\alpha}
\operatorname{Re}
\left[
    \frac{
        z_2(\vartheta)-\mathcal R_jy
    }{
        1-\mathcal R_j
    }
\right]
+
\varepsilon^{2-2\alpha}
\mathcal P_{j,\varepsilon}(\vartheta,y),
\end{align}
where
\[
\left\|
    \partial_\vartheta
    \mathcal P_{j,\varepsilon}
\right\|_{C^{0,\nu}}
\leq C.
\]
The term involving \(y\) in the first-order expansion is independent of
\(\vartheta\), and its derivative vanishes. Since
\[
    |\mathcal D_{2,\varepsilon}|=\pi,
\]
we obtain
\begin{align*}
\mathcal I_{2,2}(\varepsilon,f_2)(\vartheta)
&=
\frac{1}{2\pi\varepsilon^\alpha}
\operatorname{Re}
\left[
    z_2'(\vartheta)
    \sum_{j=2}^{\mathbf m}
    \frac{1}{1-\mathcal R_j}
\right]
+
\mathcal R_\varepsilon(\vartheta).
\end{align*}
By Lemma \ref{lem:C1},
\[
    \sum_{j=2}^{\mathbf m}
    \frac{1}{1-\mathcal R_j}
    =
    \frac{\mathbf m-1}{2}.
\]
Therefore,
\begin{equation}
\label{polygon_leading_term}
\mathcal I_{2,2}(\varepsilon,f_2)(\vartheta)
=
\frac{\mathbf m-1}{4\pi}
\varepsilon^{-\alpha}
\operatorname{Re}
\left[
    z_2'(\vartheta)
\right]
+
\mathcal P_\varepsilon(\vartheta),
\end{equation}
with
\begin{equation}
\label{polygon_remainder}
\left\|
    \mathcal P_\varepsilon
\right\|_{C^{0,\nu}}
\leq
C
\frac{\varepsilon^{2-2\alpha}}{\varepsilon}
=
C\varepsilon^{1-2\alpha}.
\end{equation}
On the other hand,
\[
\operatorname{Re}
\left[
    z_2'(\vartheta)
\right]
=
\varepsilon
\frac{f_2'(\vartheta)}{w_{2}(\vartheta)}
\cos\vartheta
-
w_{2}(\vartheta)\sin\vartheta.
\]
The local rotational contribution to $\mathcal F_2^\sigma$ is
\begin{align}
\label{local_rotation_F2_extension}
&
\left(
    \Omega_0\varepsilon^{2-2\alpha}
    +
    \Lambda\varepsilon^{2-\alpha}
\right)
f_2'
+
\left(
    \Omega_0\varepsilon^{-\alpha}
    +
    \Lambda
\right)
\operatorname{Re}
\left[
    z_2'
\right].
\end{align}
Since
\[
    \Omega_0
    +
    \frac{\mathbf m-1}{4\pi}
    =
    0,
\]
the singular terms in
\eqref{polygon_leading_term} and
\eqref{local_rotation_F2_extension} cancel exactly.
 After this cancellation, the remaining local terms satisfy
\begin{align}
\label{remaining_local_F2}
&
\left\|
\left(
    \Omega_0\varepsilon^{2-2\alpha}
    +
    \Lambda\varepsilon^{2-\alpha}
\right)
f_2'
+
\Lambda
\operatorname{Re}
\left[
    z_2'
\right]
+
\Lambda\sin\vartheta
\right\|_{C^{0,\nu}}
\leq
C\varepsilon.
\end{align}
Combining \eqref{polygon_remainder} and
\eqref{remaining_local_F2}, we obtain
\begin{equation}
\label{local_polygon_combined}
\left\|
    \mathcal I_{2,2}(\varepsilon,f_2)
    -
    \left(\Omega_0 + \varepsilon^\alpha \Lambda \right)
\Big[
    \varepsilon^{2-2\alpha} f_2'
    +\varepsilon^{-\alpha} \operatorname{Re}
\left[
    z_2'
\right]
\Big]
    +
    \Lambda\sin\vartheta
\right\|_{C^{0,\nu}}
\leq
C\left(
    \varepsilon+\varepsilon^{1-2\alpha}
\right).
\end{equation}
Using \eqref{F22}, \eqref{self_interaction_limit},
\eqref{outer_contribution_F2}, and
\eqref{local_polygon_combined}, we obtain
\begin{equation}\label{estimate_F2_extension}
\left\|
  \mathcal  F_2^\sigma(\varepsilon,\Lambda,f)
    -
    \left(
        \mathscr L[f_2]-\Lambda\sin\vartheta
    \right)
\right\|_{C^{0,\nu}}
\leq
C\left(
    \varepsilon^\alpha
    +
    \varepsilon
    +
    \varepsilon^{1-2\alpha}
\right),
\end{equation}
where
\begin{align*}
\mathscr L[f_2](\vartheta)&:=\frac{1}{2\pi^2} \partial_\vartheta
    \int_0^{2\pi}
\log\left|
    e^{i\vartheta}-e^{i\eta}
\right|
f_2(\eta)\,d\eta
+\frac{1}{2\pi} 
 f_2'(\vartheta)\\
 &= \frac{1}{2\pi}\mathcal{H}[f_2] (\vartheta)+\frac{1}{2\pi} 
 f_2'(\vartheta).
 \end{align*}
Since
\[
    \alpha>0,
    \qquad
    1-2\alpha>0,
    \qquad
    (\mathbf m+2)\alpha-1>0,
\]
the right-hand sides of
\eqref{estimate_F1_extension} and
\eqref{estimate_F2_extension} converge to zero as
$\varepsilon\to0^+$. Thus setting
\begin{equation}
\label{definition_F_at_zero}
\mathcal F^\sigma(0,\Lambda,f)
:=
\left(
    \Omega_0 f_1',
    \mathscr L [f_2]-\Lambda\sin\vartheta
\right),
\qquad
f=(f_1,f_2)\in X.
\end{equation}
gives a continuous extension of
$\mathcal F^\sigma$ to $\varepsilon=0$.
Finally, if
\[
    f_2(\vartheta)
    =
    \sum_{n=2}^{\infty}
    f_{2,n}\cos(n\vartheta),
\]
then the identity
\[
    \log\left|
        e^{i\vartheta}-e^{i\eta}
    \right|
    =
    -\sum_{n=1}^{\infty}
    \frac{\cos(n(\vartheta-\eta))}{n}
\]
implies
\begin{equation}
\label{Fourier_limiting_inner_operator}
    \mathscr L[f_2](\vartheta)
    =
    -
    \sum_{n=2}^{\infty}
    \frac{n-1}{2\pi}
    f_{2,n}\sin(n\vartheta).
\end{equation}
Hence
\[
    \Omega_0f_1'\in\mathcal Y_{1,\mathbf m},
    \qquad
    \mathscr L[f_2]-\Lambda\sin\vartheta
    \in\mathcal Y_2.
\]
Therefore,
\begin{equation}\label{definition_F_at_zero}
\mathcal F^\sigma(0,\Lambda,f)
:=
\left(
    \Omega_0 f_1',
    \mathscr L [f_2]-\Lambda\sin\vartheta
\right)\in Y.
\end{equation}
The convergence estimates are uniform on $U$, and the expression in
\eqref{definition_F_at_zero} depends continuously on
$(\Lambda,f)\in U$. This proves the joint continuity of
\[
  \mathcal  F^\sigma:
    [0,\varepsilon_0)\times U
    \longrightarrow Y
\]
and completes the proof.
\end{proof}

\subsection{Differentiability with respect to $(\Lambda,f)$}\label{section_diff}

We next establish the differentiability of the nonlinear operator with
respect to the unknowns $(\Lambda,f)$. 

\begin{pro}
\label{pro:F_differentiable} 
Under the assumptions of Proposition \ref{Fwelldef}, for every fixed
$\sigma\in\{0,1\}$ and every $\varepsilon\in[0,\varepsilon_0)$, the map
\[
    (\Lambda,f)
    \longmapsto
    \mathcal F^\sigma(\varepsilon,\Lambda,f)
\]
is Fréchet differentiable on $U$. Moreover, the operator-valued map
\[
    D_{(\Lambda,f)}\mathcal F^\sigma:
    [0,\varepsilon_0)\times U
    \longrightarrow
    \mathcal L(\mathbb R\times X,Y)
\]
is continuous.

At $\varepsilon=0$, one has
\begin{align}
\label{derivative_F_at_zero}
&D_{(\Lambda,f)}
\mathcal F^\sigma(0,\Lambda,f)
\bigl[
    \widehat\Lambda,h
\bigr]
=
\left(
    \Omega_0h_1',
    \frac{1}{2\pi}\mathcal{H}[h_2] +\frac{1}{2\pi} 
 h_2'-\widehat\Lambda\sin(\cdot)
\right).
\end{align}

\end{pro}

\begin{proof}
Throughout the proof, $\sigma\in\{0,1\}$ is fixed. For notational
convenience, set
$$
    \zeta_j^\sigma
    :=
    e^{i\frac{\pi(2j-2+\sigma)}{\mathbf m}},
    \qquad
    j\in\{1,\ldots,\mathbf m\},
$$
and keep the notation
$$\mathcal R_j
    =
    e^{i\frac{2\pi(j-1)}{\mathbf m}}, \qquad j \in \{1, \ldots, \mathbf{m}\}.$$
Let $
    f=(f_1,f_2)\in B_X(0,r),$
    $
    h=(h_1,h_2)\in X,$
 and recall that
$$
    w_k(\vartheta)
    =
    \sqrt{1+2\varepsilon f_k(\vartheta)},
    \qquad
    z_k(\vartheta)
    =
    w_k(\vartheta)e^{i\vartheta}.
$$
Thus,
$$
    D_{f_k}w_k[h_k]
    =
    \varepsilon\frac{h_k}{w_k},\qquad
    D_{f_k}z_k[h_k](\vartheta)
    =
    \varepsilon Q_{k}(\vartheta),
    \qquad
    Q_{k}(\vartheta)
    :=
    \frac{h_k(\vartheta)}{w_k(\vartheta)}
    e^{i\vartheta}.
$$
We divide the proof into three steps.

\medskip

\noindent
\textbf{Differentiability for $\varepsilon>0$.}
The derivative with respect to $\Lambda$ is given explicitly by
\begin{equation}
\label{partial_Lambda_F}
\begin{aligned}
D_\Lambda \mathcal F^\sigma
(\varepsilon,\Lambda,f)(\vartheta)
=
\Bigg(
    &\varepsilon^\alpha f_1'(\vartheta), \varepsilon^{2-\alpha}f_2'(\vartheta)
    +
    \varepsilon
    \frac{f_2'(\vartheta)}{w_2(\vartheta)}
    \cos\vartheta
    -
    w_2(\vartheta)\sin\vartheta
\Bigg).
\end{aligned}
\end{equation}
This defines an element of $Y$. Moreover, the right-hand side depends
continuously on $(\varepsilon,\Lambda,f)$ as an element of $Y$.

    For the G\^{a}teaux derivative with respect to $f$ we consider the direction $h=(h_1,h_2)$ and compute the directional derivative component-wise as
    $$D_f\mathcal F_k^\sigma(\varepsilon,\Lambda,f)[h]=\frac{\mathrm{d}}{\mathrm{d}t}\mathcal F_k^\sigma(\varepsilon,\Lambda, f+th)\big|_{t=0}. $$
For a nonvanishing complex-valued function $Z_f$, we use the elementary
identity
$$
    D_f\log|Z_f|[h]
    =
    \operatorname{Re}
    \left(
        \frac{D_fZ_f[h]}{Z_f}
    \right).
$$
Moreover,
$$
    D_{f_j}\bigl(w_j^2\bigr)[h_j]
    =
    2\varepsilon h_j.
$$
These identities yield the directional derivatives below. To write them
compactly, define
\begin{align*}
\Delta_{11}(\rho,\vartheta,\eta)
&:=
z_1(\vartheta)-\rho z_1(\eta),
\\
\Delta_{12,j}^\sigma(\rho,\vartheta,\eta)
&:=
z_1(\vartheta)
-
\zeta_j^\sigma
\left(
    \varepsilon\rho z_2(\eta)
    +\varepsilon^\alpha
\right),
\\
\Delta_{21}^\sigma(\rho,\vartheta,\eta)
&:=
e^{i\frac{\pi\sigma}{\mathbf m}}
\left(
    \varepsilon z_2(\vartheta)
    +\varepsilon^\alpha
\right)
-
\rho z_1(\eta),
\\
\Delta_{22,j}(\rho,\vartheta,\eta)
&:=
\varepsilon z_2(\vartheta)
+\varepsilon^\alpha
-
\mathcal R_j
\left(
    \varepsilon\rho z_2(\eta)
    +\varepsilon^\alpha
\right).
\end{align*}
With this notation, the directional derivative of the first component is
\begin{align*}
\bigl[
D_f\mathcal F_1^\sigma
(\varepsilon,\Lambda,f)[h]
\bigr](\vartheta)
&=
\left(
    \Omega_0+\varepsilon^\alpha\Lambda
\right)
h_1'(\vartheta)
-
\frac{\varepsilon^{2\alpha}}{2\pi}
\partial_\vartheta
\left(
    I_{11}(\vartheta)
    +
    I_{12}(\vartheta)
\right)
\nonumber\\
&\quad
+
\frac{\varepsilon^{2\alpha}}{2\pi^2}
\partial_\vartheta
\sum_{j=1}^{\mathbf m}
\left(
    I_{13,j}^\sigma[h](\vartheta)
    +
    I_{14,j}^\sigma[h](\vartheta)
\right).
\end{align*}
where
\begin{align*}
I_{11}[h](\vartheta)
&:=
\int_0^{2\pi}\int_0^1
\operatorname{Re}
\left[
    \frac{
        Q_1(\vartheta)
        -
        \rho Q_1(\eta)
    }{
        \Delta_{11}(\rho,\vartheta,\eta)
    }
\right]
\rho w_1(\eta)^2
\,d\rho\,d\eta,
\\
I_{12}[h](\vartheta)&:=
2\int_0^{2\pi}\int_0^1
\log|
    \Delta_{11}(\rho,\vartheta,\eta)
|
\rho h_1(\eta)
\,d\rho\,d\eta,
\end{align*}
and
\begin{align*}
I_{13,j}^\sigma[h](\vartheta)
&:=
\int_0^{2\pi}\int_0^1
\operatorname{Re}
\left[
    \frac{
        Q_1(\vartheta)
        -
        \zeta_j^\sigma
        \varepsilon\rho Q_2(\eta)
    }{
        \Delta_{12,j}^\sigma
        (\rho,\vartheta,\eta)
    }
\right]
\rho w_2(\eta)^2
\,d\rho\,d\eta,
\\
I_{14,j}^\sigma[h](\vartheta)
&:=
2\int_0^{2\pi}\int_0^1
\log|
    \Delta_{12,j}^\sigma
    (\rho,\vartheta,\eta)
|
\rho h_2(\eta)
\,d\rho\,d\eta.
\end{align*}
For the second component one has
\begin{align}
\label{DfF2_correct}
\bigl[
D_f\mathcal F_2^\sigma
(\varepsilon,\Lambda,f)[h]
\bigr](\vartheta)
&=
\varepsilon^{1-\alpha}
\left(
    \Omega_0+\varepsilon^\alpha\Lambda
\right)
\Bigg[
    \varepsilon^{1-\alpha}h_2'(\vartheta)
    +
    \frac{h_2'(\vartheta)}{w_2(\vartheta)}
    \cos\vartheta
    -
    \varepsilon
    f_2'(\vartheta)
    \frac{h_2(\vartheta)}{w_2(\vartheta)^3}
    \cos\vartheta
\nonumber\\
&\quad
    -
    \frac{h_2(\vartheta)}{w_2(\vartheta)}
    \sin\vartheta
\Bigg]
-
\frac{1}{2\pi}
\partial_\vartheta
\left(
    I_{21}^\sigma[h](\vartheta)
    +
    I_{22}^\sigma[h](\vartheta)
\right)
\nonumber\\
&\quad
+
\frac{1}{2\pi^2}
\partial_\vartheta
\sum_{j=1}^{\mathbf m}
\left(
    I_{23,j}[h](\vartheta)
    +
    I_{24,j}[h](\vartheta)
\right),
\end{align}
where
\begin{align*}
I_{21}^\sigma[h](\vartheta)
&:=
\int_0^{2\pi}\int_0^1
\operatorname{Re}
\left[
    \frac{
        e^{i\frac{\pi\sigma}{\mathbf m}}
        \varepsilon Q_2(\vartheta)
        -
        \rho Q_1(\eta)
    }{
        \Delta_{21}^\sigma
        (\rho,\vartheta,\eta)
    }
\right]
\rho w_1(\eta)^2
\,d\rho\,d\eta,
\\
I_{22}^\sigma[h](\vartheta)
&:=
2\int_0^{2\pi}\int_0^1
\log|
    \Delta_{21}^\sigma
    (\rho,\vartheta,\eta)
|
\rho h_1(\eta)
\,d\rho\,d\eta,
\end{align*}
and
\begin{align*}
I_{23,j}[h](\vartheta)
&:=
\varepsilon
\int_0^{2\pi}\int_0^1
\operatorname{Re}
\left[
    \frac{
        Q_2(\vartheta)
        -
        \mathcal R_j\rho Q_2(\eta)
    }{
        \Delta_{22,j}
        (\rho,\vartheta,\eta)
    }
\right]
\rho w_2(\eta)^2
\,d\rho\,d\eta,
\\
I_{24,j}[h](\vartheta)
&:=
2\int_0^{2\pi}\int_0^1
\log|
    \Delta_{22,j}
    (\rho,\vartheta,\eta)
|
\rho h_2(\eta)
\,d\rho\,d\eta.
\end{align*}
We explain the argument for the outer self-interaction
$I_{11}$, which contains the principal boundary singularity. The other terms can be treated analogously. 

  Adding and subtracting an $\eta$-derivative and then integrating by parts in
$\eta$, using periodicity, gives
\begin{equation}
\label{derivative_If}
\partial_\vartheta I_{11}[h](\vartheta)
=
\int_0^{2\pi}\int_0^1
K_{f_1}[h_1](\rho,\vartheta,\eta)
\,d\rho\,d\eta,
\end{equation}
where
$$
    K_{f_1}[h_1]
    :=
    \rho w_1(\eta)^2 A_{f_1}[h_1]
    +
    2\varepsilon\rho f_1'(\eta)B_{f_1}[h_1],
$$
\begin{equation}
\label{definition_Af}
A_{f_1}[h_1]
:=
\operatorname{Re}
\left[
    \frac{
         \Delta_{11}(\rho,\vartheta,\eta)
        \bigl(
            Q_1'(\vartheta)
            -
            \rho Q_1'(\eta)
        \bigr)
        -
        \bigl( Q_1(\vartheta)-\rho Q_1(\eta)\bigr)
        \bigl(
            z_1'(\vartheta)
            -
            \rho z_1'(\eta)
        \bigr)
    }{
         \Delta_{11}(\rho,\vartheta,\eta)^2
    }
\right],
\end{equation}
and
$$
    B_{f_1}[h_1](\rho,\vartheta,\eta)
    :=
   \operatorname{Re}
\left[
    \frac{
        Q_1(\vartheta)
        -
        \rho Q_1(\eta)
    }{
        \Delta_{11}(\rho,\vartheta,\eta)
    }
\right].
$$
For $f$ in the prescribed ball, the uniform chord--arc estimate gives
$$
    | \Delta_{11}(\rho,\vartheta,\eta)|
    \geq
    c\, \left[
        (1-\rho)^2
        +
        4\rho
        \sin^2\left(
            \frac{\vartheta-\eta}{2}
        \right)
    \right]^{1/2}.
$$
Consequently, the kernels in
\eqref{derivative_If} satisfy estimates of the form
$$
    |K_{f_1}[h_1](\rho,\vartheta,\eta)|
    \leq
    C\|h_1\|_{C^{1,\nu}}
     \left|
        \sin\left(
            \frac{\vartheta-\eta}{2}
        \right)
    \right|^{\nu-1}
    G_1(\rho),
$$
and
$$
    |\partial_\vartheta
    K_{f_1}[h_1](\rho,\vartheta,\eta)|
    \leq
    C\|h_1\|_{C^{1,\nu}}
     \left|
        \sin\left(
            \frac{\vartheta-\eta}{2}
        \right)
    \right|^{\nu-2}
    G_2(\rho),
$$
where $G_1,G_2\in L^1(0,1)$ are the radial weights appearing in the proof
of Proposition \ref{Fwelldef}. Lemma \ref{lemmasurvey} therefore gives
\begin{equation}
\label{bound_Df_I11}
\left\|
    \partial_\vartheta I_{11}[h]
\right\|_{C^{0,\nu}}
\leq
C\|h_1\|_{C^{1,\nu}}.
\end{equation}
Let us now check the continuity of the derivative in operator norm.
 Let
$f,g$ belong to the prescribed ball. By adding and subtracting crossed terms we get

\begin{equation*}
\begin{split}
    K_{f_1}(\rho,\vartheta,\eta)-K_{g_1}(\rho,\vartheta,\eta) 
    &=\rho w_1^2[f_1](\eta)[A_{f_1}(\rho,\vartheta,\eta)-A_{g_1}(\rho,\vartheta,\eta)] \\
    &\quad+ \rho (w_1^2[f_1](\eta)-w_1^2[g_1](\eta))A_{g_1}(\rho,\vartheta,\eta)  \\
    &\quad +2\varepsilon \rho (f_1'(\eta)-g_1'(\eta))B_{f_1}(\rho,\vartheta,\eta)\\
    &\quad+2\varepsilon \rho g_1'(\eta)[B_{f_1}(\rho,\vartheta,\eta)-B_{g_1}(\rho,\vartheta,\eta)].
\end{split}
\end{equation*}
Here we use the notation $w_1[f_1](\vartheta):=\sqrt{1+2\varepsilon f_1(\vartheta)}$ to stress the dependence on $f$.
We treat the last term; the remaining three terms are handled similarly. 
The quotient identity
$$
 \frac{
        Q_1[{f_1}]
    }{
        \Delta_{11}[f_1]
    }
-
 \frac{
        Q_1[{g_1}]
    }{
        \Delta_{11}[g_1]
    }
=
 \frac{
        Q_1[{f_1}]-Q_1[{g_1}]
    }{
        \Delta_{11}[f_1]
    }
    +
    \frac{
        Q_1[g_1](\Delta_{11}[g_1]-\Delta_{11}[f_1])
    }{
        \Delta_{11}[f_1]\Delta_{11}[g_1],
    }
$$
implies
\begin{equation}
\label{difference_Bf_Bg}
B_{f_1}[h_1]-B_{g_1}[h_1]
=
\operatorname{Re}
\left[
    \frac{
        Q_1[{f_1}]-Q_1[{g_1}]
    }{
        \Delta_{11}[f_1]
    }
    +
    \frac{
        Q_1[g_1](\Delta_{11}[g_1]-\Delta_{11}[f_1])
    }{
        \Delta_{11}[f_1]\Delta_{11}[g_1]
    }
\right].
\end{equation}
The relevant differences satisfy
\begin{align}
\label{difference_zfg}
&
\left|
    \bigl(
        z_1[f_1]-z_1[g_1]
    \bigr)(\vartheta)
    -
    \rho
    \bigl(
        z_1[f_1]-z_1[g_1]
    \bigr)(\eta)
\right|
\nonumber\\
&\qquad
\leq
C\varepsilon
\|f_1-g_1\|_{C^{1,\nu}}
\left[
        (1-\rho)^2
        +
        4\rho
        \sin^2\left(
            \frac{\vartheta-\eta}{2}
        \right)
    \right]^{1/2},
\end{align}
and
\begin{align}
\label{difference_qfg}
&
\left|
    \bigl(
        Q_1[f_1]-Q_1[g_1]
    \bigr)(\vartheta)
    -
    \rho
    \bigl(
        Q_1[f_1]-Q_1[g_1]
    \bigr)(\eta)
\right|
\nonumber\\
&\qquad
\leq
C\varepsilon
\|h_1\|_{C^{1,\nu}}
\|f_1-g_1\|_{C^{1,\nu}}
\left[
        (1-\rho)^2
        +
        4\rho
        \sin^2\left(
            \frac{\vartheta-\eta}{2}
        \right)
    \right]^{1/2}.
\end{align}
Notice that the estimate \eqref{difference_zfg} does not contain
$\|h\|_{C^{1,\nu}}$.

Combining
\eqref{difference_Bf_Bg}--\eqref{difference_qfg} with the chord--arc
estimate gives
$$
    |B_{f_1}[h_1]-B_{g_1}[h_1]|
    \leq
    C\varepsilon
    \|h_1\|_{C^{1,\nu}}
    \|f_1-g_1\|_{C^{1,\nu}}.
$$
A similar calculation, using the correct denominator
$\Delta_{11}^2$ in \eqref{definition_Af}, gives kernel estimates of the form
\begin{align*}
|K_{f_1}[h_1]-K_{g_1}[h_1]|
&\leq
C\varepsilon
\|h_1\|_{C^{1,\nu}}
\|f_1-g_1\|_{C^{1,\nu}}
\left|
        \sin\left(
            \frac{\vartheta-\eta}{2}
        \right)
    \right|^{\nu-1}
G_1(\rho),
\\
|\partial_\vartheta
(K_{f_1}[h_1]-K_{g_1}[h_1])|
&\leq
C\varepsilon
\|h_1\|_{C^{1,\nu}}
\|f_1-g_1\|_{C^{1,\nu}}
\left|
        \sin\left(
            \frac{\vartheta-\eta}{2}
        \right)
    \right|^{\nu-2}
G_2(\rho).
\end{align*}
Lemma \ref{lemmasurvey} therefore yields
$$
\left\|
    \partial_\vartheta I_{11}[f_1]
    -
    \partial_\vartheta I_{11}[g_1]
\right\|_{C^{0,\nu}}
\leq
C\varepsilon
\|f_1-g_1\|_{C^{1,\nu}}
\|h_1\|_{C^{1,\nu}}.
$$
Taking the supremum over
$\|h_1\|_{C^{1,\nu}}\leq1$ proves continuity in operator norm for this
self-interaction term.
The inner self-interaction satisfies the same type of estimate. For the
outer--inner and distinct inner--inner interactions, continuity follows
from the smooth dependence of the separated kernels and boundary
parametrizations on $(\varepsilon,\Lambda,f)$. It follows that
\[
D_f\mathcal F^\sigma
\in
C\left(
    (0,\varepsilon_0)\times U;
    \mathcal L(X,Y)
\right).
\]
Together with \eqref{partial_Lambda_F}, this gives
\begin{equation}
\label{positive_epsilon_derivative_continuity}
D_{(\Lambda,f)}\mathcal F^\sigma
\in
C\left(
    (0,\varepsilon_0)\times U;
    \mathcal L(E,Y)
\right).
\end{equation}
\medskip

\noindent
\textbf{Derivative at $\varepsilon=0$.}
By Proposition \ref{Fwelldef},
\[
    \mathcal F^\sigma(0,\Lambda,f)
    =
    \left(
        \Omega_0f_1',
        \mathscr Lf_2-\Lambda\sin\vartheta
    \right).
\]
This map is affine in $(\Lambda,f)$ and is therefore Fréchet
differentiable. Its derivative is precisely
\[
D_{(\Lambda,f)}
\mathcal F^\sigma(0,\Lambda,f)
\bigl[
    \widehat\Lambda,(h_1,h_2)
\bigr]
=
\left(
    \Omega_0h_1',
    \mathscr Lh_2-\widehat\Lambda\sin\vartheta
\right).
\]

\medskip

\noindent
\textbf{Continuity of the derivative at $\varepsilon=0$.}
We first consider the $\Lambda$-derivative. From
\eqref{partial_Lambda_F},
\begin{align*}
&
D_\Lambda\mathcal F^\sigma
(\varepsilon,\Lambda,f)
-
\left(
    0,-\sin\vartheta
\right)
=
\Bigg(
    \varepsilon^\alpha f_1',
    \varepsilon^{2-\alpha}f_2'
    +
    \varepsilon
    \frac{f_2'}{w_2}
    \cos\vartheta
    -
    (w_2-1)\sin\vartheta
\Bigg).
\end{align*}
Since
\[
    \|w_2-1\|_{C^{1,\nu}}
    \leq
    C\varepsilon\|f_2\|_{C^{1,\nu}},
\]
we obtain
\begin{equation}
\label{Lambda_derivative_extension}
\left\|
D_\Lambda\mathcal F^\sigma
(\varepsilon,\Lambda,f)
-
\left(
    0,-\sin\vartheta
\right)
\right\|_Y
\leq
C\left(
    \varepsilon^\alpha+\varepsilon
\right).
\end{equation}

We next examine the derivative with respect to $f$. 
For the first component, the local contribution satisfies
\[
\begin{aligned}
&
D_{(\Lambda,f)}
\left[
    \left(
        \Omega_0+\varepsilon^\alpha\Lambda
    \right)f_1'
\right]
\bigl[
    \widehat\Lambda,h
\bigr]
-
\Omega_0h_1'
=
\varepsilon^\alpha
\left(
    \widehat\Lambda f_1'
    +
    \Lambda h_1'
\right).
\end{aligned}
\]
The differentiated outer self-interaction is bounded by
\[
    C\varepsilon^{2\alpha}\|h_1\|_{C^{1,\nu}},
\]
as follows from \eqref{bound_Df_I11}. Differentiating the regularized
outer--inner decomposition, used in the proof of Proposition \ref{Fwelldef}, and using the cancellation over the regular
polygon gives
\[
\begin{aligned}
&
\left\|
D_{(\Lambda,f)}
\mathcal F_1^\sigma
(\varepsilon,\Lambda,f)
\bigl[
    \widehat\Lambda,h
\bigr]
-
\Omega_0h_1'
\right\|_{C^{0,\nu}}
\leq
C\left(
    \varepsilon^\alpha
    +
    \varepsilon^{2\alpha}
    +
    \varepsilon^{(\mathbf m+2)\alpha-1}
\right)
\big(|\widehat\Lambda|+\|h\|_{C^{1,\nu}}\big).
\end{aligned}
\]
\begin{equation}
\label{derivative_extension_F1}
\end{equation}
The proof of continuity of the second component contribution follows the same line of Proposition \ref{Fwelldef}. leading to
\begin{align}
\label{derivative_extension_F2}
&
\left\|
D_{(\Lambda,f)}
\mathcal F_2^\sigma
(\varepsilon,\Lambda,f)
\bigl[
    \widehat\Lambda,h
\bigr]
-
\left(
    \mathscr L[h_2]
    -
    \widehat\Lambda\sin\vartheta
\right)
\right\|_{C^{0,\nu}}
\leq
C\left(
    \varepsilon^\alpha
    +
    \varepsilon
    +
    \varepsilon^{1-2\alpha}
\right)
\big(|\widehat\Lambda|+\|h\|_{C^{1,\nu}}\big),
\end{align}
with
$$
    \mathscr L[h_2]=\frac{1}{2\pi}\mathcal{H}[h_2] +\frac{1}{2\pi} 
 h_2'.
$$
Therefore, the derivative admits a continuous extension to
$\varepsilon=0$. Together with
\eqref{positive_epsilon_derivative_continuity}, this proves that
\[
D_{(\Lambda,f)}\mathcal F^\sigma
\in
C\left(
    [0,\varepsilon_0)\times U;
    \mathcal L(\mathbb R\times X,Y)
\right).
\]
This completes the proof.

\end{proof}

\section{Invertibility of the linearized operator and proof of the main theorem}\label{sec:linop}

In this section, we show that the linearization of
the nonlinear operator at the singular limiting configuration
$
    (\varepsilon,\Lambda,f)=(0,0,0),
$  admits a diagonal Fourier representation. 
This is what will ultimately allow
us to prove  that
$$
    D_{(\Lambda,f)}\mathcal F^\sigma(0,0,0):
    \mathbb R\times X\longrightarrow Y,
$$ is an isomorphism and complete the proof of Theorem \ref{main_theo}.
Indeed, by Proposition \ref{pro:F_differentiable}, the opertor
$$
    \mathcal L
    :=
    D_{(\Lambda,f)}
    \mathcal F^\sigma(0,0,0):
    \mathbb R\times X
    \longrightarrow Y
$$
is the bounded linear operator, given by
\begin{equation}
\label{lin-op}
    \mathcal L[\lambda,h]
    =
    \bigl(
        \mathcal L_1[\lambda,h],
        \mathcal L_2[\lambda,h]
    \bigr),
\end{equation}
where
\begin{equation}
\label{linearized_operator_components}
\begin{aligned}
    \mathcal L_1[\lambda,h](\vartheta)
    &=
    \Omega_0 h_1'(\vartheta),
    \\[1mm]
    \mathcal L_2[\lambda,h](\vartheta)
    &=
    -\lambda\sin\vartheta
    +
    \frac{1}{2\pi}
    \left(
        h_2'(\vartheta)
        +
        \mathcal H[h_2](\vartheta)
    \right).
\end{aligned}
\end{equation}
 To make the Fourier structure of $\mathcal{L}$ explicit, let
$$
    h_1(\vartheta)
    =
    \sum_{k=1}^{\infty}
    h_{1,k}\cos(k\mathbf m\vartheta),
    \qquad
    h_2(\vartheta)
    =
    \sum_{n=2}^{\infty}
    h_{2,n}\cos(n\vartheta).
$$
Using
$$
    \Omega_0
    =
    \frac{1-\mathbf m}{4\pi}\quad \textnormal{and}\quad
    \mathcal H[\cos(n\vartheta)]
    =
    \sin(n\vartheta),
$$
we obtain
\begin{align}
\label{linearized_operator_Fourier}
    \mathcal L_1[\lambda,h](\vartheta)
    &=
    \frac{\mathbf m-1}{4\pi}
    \sum_{k=1}^{\infty}
    k\mathbf m\,a_k
    \sin(k\mathbf m\vartheta),
    \nonumber\\
    \mathcal L_2[\lambda,h](\vartheta)
    &=
    -\lambda\sin\vartheta
    -
    \frac{1}{2\pi}
    \sum_{n=2}^{\infty}
    (n-1)b_n\sin(n\vartheta).
\end{align}
Thus, $\mathcal L$ is diagonal with respect to the Fourier bases defining
the spaces $X$ and $Y$. The scalar parameter $\lambda$ generates the first
sine mode in the second component, whereas the boundary perturbation $h_2$
generates all the higher sine modes.

We now establish the final property required for the application of the
implicit function theorem.

\begin{pro}
\label{pro:linearized_isomorphism}
The linear operator
$$
    \mathcal L:
    \mathbb R\times X
    \longrightarrow Y,
$$
defined in \eqref{lin-op} is a bounded isomorphism.
\end{pro}

\begin{proof}
The boundedness of $\mathcal L$ follows directly from
\eqref{linearized_operator_components}, since the periodic Hilbert transform is bounded on
$C^{0,\nu}(\mathbb T)$.
We next prove that $\mathcal L$ is bijective and identify its inverse.
Let
$$
    g=(g_1,g_2)\in Y.
$$
By the definitions of $\mathcal Y_{1,\mathbf m}$ and $\mathcal Y_2$, we may
write
$$
    g_1(\vartheta)
    =
    \sum_{k=1}^{\infty}
    g_{1,k}\sin(k\mathbf m\vartheta),
\qquad
    g_2(\vartheta)
    =
    g_{2,1}\sin\vartheta
    +
    \sum_{n=2}^{\infty}
    g_{2,n}\sin(n\vartheta).
$$
We seek
$$
    (\lambda,h)
    =
    (\lambda,h_1,h_2)
    \in\mathbb R\times X
$$
such that
$$
    \mathcal L[\lambda,h]=g.
$$
For the first component, define
\begin{equation}
\label{inverse_h1}
    h_1(\vartheta)
    :=
    \frac{4\pi}{\mathbf m-1}
    \sum_{k=1}^{\infty}
    \frac{g_{1,k}}{k\mathbf m}
    \cos(k\mathbf m\vartheta).
\end{equation}
This is the unique zero-average, even, $\frac{2\pi}{\mathbf m}$-periodic
primitive of $g_1/\Omega_0$. In particular,
$
    h_1\in\mathcal X_{1,\mathbf m},
$
and
$$
\begin{aligned}
    \Omega_0h_1'(\vartheta)
    &=
    \frac{1-\mathbf m}{4\pi}
    \left(
        -\frac{4\pi}{\mathbf m-1}
        \sum_{k=1}^{\infty}
        g_{1,k}\sin(k\mathbf m\vartheta)
    \right)
    =
    g_1(\vartheta).
\end{aligned}
$$
Moreover, the periodic primitive estimate gives
\begin{equation}
\label{estimate_h1_inverse}
    \|h_1\|_{C^{1,\nu}}
    \leq
    C\|g_1\|_{C^{0,\nu}}.
\end{equation}
We now consider the second component. We first write
$$
g_2(\vartheta)=g_{2,1}\sin(\vartheta)+\sum_{n=2}^{\infty}g_{2,n}\sin(n\vartheta)=:g_{2,1}\sin(\vartheta)+g_{\geq 2}(\vartheta).
$$
Then, we can define the preimage of $g_2$ as 
\begin{equation}
\label{inverse_lambda}
    \lambda
    :=
    -g_{2,1},
\end{equation}
and
\begin{equation}
\label{inverse_h2}
    h_2(\vartheta)
    :=
    -2\pi
    \sum_{n=2}^{\infty}
    \frac{g_{2,n}}{n-1}
    \cos(n\vartheta).
\end{equation}
To verify that $h_2\in C^{1,\nu}(\mathbb T)$,  introduce the periodic Fourier multiplier
$$
    \mathcal S[g_{2,\geq2}](\vartheta)
    :=
    \sum_{n=2}^{\infty}
    \frac{g_{2,n}}{n-1}
    \sin(n\vartheta).
$$
The multiplier $(n-1)^{-1}$ is of order $-1$. Equivalently, its periodic
convolution kernel has at most a logarithmic singularity and belongs to
$L^1(\mathbb T)$. It follows that
$$
    \mathcal S:
    C^{0,\nu}(\mathbb T)
    \longrightarrow
    C^{0,\nu}(\mathbb T),
$$
is bounded. Using the Fourier action of the Hilbert transform, we may write
$$
    h_2
    =
    2\pi\mathcal H
    \bigl[
        \mathcal S[g_{2,\geq2}]
    \bigr],
$$
and
$$
    h_2'
    =
    2\pi
    \left(
        g_{2,\geq2}
        +
        \mathcal S[g_{2,\geq2}]
    \right).
$$
Since both $\mathcal H$ and $\mathcal S$ are bounded on
$C^{0,\nu}(\mathbb T)$, we conclude that
$
    h_2\in C^{1,\nu}(\mathbb T)
$
and
\begin{equation}
\label{estimate_h2_inverse}
    \|h_2\|_{C^{1,\nu}}
    \leq
    C
    \|g_{2,\geq2}\|_{C^{0,\nu}}
    \leq
    C\|g_2\|_{C^{0,\nu}}.
\end{equation}
It remains to verify that the functions defined above solve the second equation. By \eqref{inverse_h2}  and \eqref{Hilbert_Fourier_action},
$$
    \frac{1}{2\pi}
    \left(
        h_2'
        +
        \mathcal H[h_2]
    \right)(\vartheta)
    =
    \sum_{n=2}^{\infty}
    g_{2,n}\sin(n\vartheta).
$$
Together with \eqref{linearized_operator_Fourier} and \eqref{inverse_lambda}, this gives
$$
\begin{aligned}
    \mathcal L_2[\lambda,h](\vartheta)
    &=
    -\lambda\sin\vartheta
    +
    \frac{1}{2\pi}
    \left(
        h_2'(\vartheta)
        +
        \mathcal H[h_2](\vartheta)
    \right)
    \\
    &=
    g_{2,1}\sin\vartheta
    +
    \sum_{n=2}^{\infty}
    g_{2,n}\sin(n\vartheta)
    \\
    &=
    g_2(\vartheta).
\end{aligned}
$$
Therefore, $\mathcal L$ is surjective.
The Fourier representation \eqref{linearized_operator_Fourier} also shows
that the preimage is unique. Indeed, if
$$
    \mathcal L[\lambda,h]=0,
$$
then the first component gives $h_1=0$, while the first sine mode of the
second component gives $\lambda=0$. Since $n-1\neq0$ for every $n\geq2$,
all the Fourier coefficients of $h_2$ also vanish. Hence,
$$
    (\lambda,h)=(0,0).
$$
Thus, $\mathcal L$ is injective.

Finally, \eqref{estimate_h1_inverse},
\eqref{estimate_h2_inverse}, and the estimate
$$
    |g_{2,1}|
    \leq
    C\|g_2\|_{C^{0,\nu}}
$$
yield
$$
    |\lambda|+\|h\|_X
    \leq
    C\|g\|_Y.
$$
Therefore, $\mathcal L^{-1}:Y\to\mathbb R\times X$ is bounded, and
$\mathcal L$ is a bounded isomorphism.
\end{proof}

\begin{proof}[Proof of Theorem \ref{main_theo}]
Fix $\sigma\in\{0,1\}$ and regard the parameter $\varepsilon$ as belonging
to the one-sided metric space
$$
    [0,\varepsilon_0).
$$
We apply the parameter-dependent implicit function theorem
\cite[Theorem~5F.4, p.~305]{DontchevRockafellar2014}. Let
$$
    E:=\mathbb R\times X,
$$
and recall that
$$
    U
    =
    (-\Lambda_0,\Lambda_0)
    \times
    B_X(0,r)
    \subset E,
$$
is an open neighborhood of $(0,0)$. By Corollary
\ref{cor_trivial_sol},
$$
    \mathcal F^\sigma(0,0,0)=0.
$$
Moreover, Proposition \ref{Fwelldef} and Proposition
\ref{pro:F_differentiable} show that
$$
    \mathcal F^\sigma:
    [0,\varepsilon_0)\times U
    \longrightarrow Y,
$$
is continuous in all its variables and continuously Fréchet differentiable
with respect to $(\Lambda,f)$. More precisely,
$$
    D_{(\Lambda,f)}\mathcal F^\sigma:
    [0,\varepsilon_0)\times U
    \longrightarrow
    \mathcal L(E,Y),
$$
is continuous. Finally, Proposition \ref{pro:linearized_isomorphism} shows that
$$
    D_{(\Lambda,f)}
    \mathcal F^\sigma(0,0,0)
    =
    \mathcal L:
    E\longrightarrow Y,
$$
is a bounded isomorphism. All the hypotheses of the
parameter-dependent implicit function theorem are therefore satisfied. Consequently, there exist $\varepsilon_\ast\in(0,\varepsilon_0)$ and a
unique continuous map
$$
    [0,\varepsilon_\ast)
    \longrightarrow
    U,
    \qquad
    \varepsilon
    \longmapsto
    \bigl(
        \Lambda(\varepsilon),
        f(\varepsilon)
    \bigr),
$$
such that
$$
    \Lambda(0)=0,
    \qquad
    f(0)=0,
$$
and
\begin{equation}
\label{solution_branch_F}
    \mathcal F^\sigma
    \left(
        \varepsilon,
        \Lambda^\sigma(\varepsilon),
        f(\varepsilon)
    \right)
    =
    0,
    \qquad
    \varepsilon\in[0,\varepsilon_\ast).
\end{equation}
This solution is locally unique in the symmetry class encoded by the spaces
$X$ and $Y$.

Write
$$
    f(\varepsilon)
    =
    \left(
        f_1(\varepsilon),
        f_2(\varepsilon)
    \right),
$$
and define
$$
    w_j(\varepsilon,\vartheta)
    :=
    \sqrt{
        1+
        2\varepsilon
        f_j(\varepsilon,\vartheta)
    },
    \qquad j\in\{1,2\}.
$$
For $\varepsilon>0$, these radial functions determine the generating
domains
$$
    \mathcal D_j^{\varepsilon}
    =
    \left\{
        r e^{i\vartheta}:
        \vartheta\in\mathbb T,\;
        0\leq r<
        w_j(\varepsilon,\vartheta)
    \right\},
    \qquad j\in\{1,2\},
$$
and the remaining inner components are obtained through the rotations
defined in \eqref{def_D2j}. The associated angular velocity is
$$
    \Omega_\varepsilon
    =
    \frac{\Omega_0}{\varepsilon^{2\alpha}}
    +
    \frac{
        \Lambda(\varepsilon)
    }{
        \varepsilon^\alpha
    }.
$$

By the equivalence between the nonlinear equation
\eqref{solution_branch_F} and the reduced boundary system
\eqref{Last_System}, these domains define a uniformly rotating vortex
configuration with angular velocity $\Omega_\varepsilon$. Applying
the argument for each $\sigma\in\{0,1\}$ gives the two
reflection-symmetric configurations asserted in Theorem
\ref{main_theo}. This completes the proof.
\end{proof}

\appendix

\section{Technical results and integral formulas} \label{appendix_techinal_results}
The following lemma provides a Hölder regularity estimate for a class of
weakly singular integral operators on the torus. Its proof follows by adapting
the arguments in \cite{GHM22-1,G26} to the periodic setting.

\begin{lem}[A Hölder estimate for weakly singular kernels]
\label{lemmasurvey}
Let
$$
    K:[0,1]\times\mathbb T\times\mathbb T
    \longrightarrow\mathbb C
$$
be a measurable kernel, differentiable with respect to $\vartheta$ away
from the diagonal $\vartheta=\eta$. Assume that, for some
$\nu\in(0,1)$,
$$
    |K(\rho,\vartheta,\eta)|
    \leq
    C_0
    \left|
        \sin\left(\frac{\vartheta-\eta}{2}\right)
    \right|^{\nu-1}
    g_1(\rho),
$$
and
$$
    |\partial_\vartheta K(\rho,\vartheta,\eta)|
    \leq
    C_0
    \left|
        \sin\left(\frac{\vartheta-\eta}{2}\right)
    \right|^{\nu-2}
    g_2(\rho),
$$
where $g_1,g_2\in L^1(0,1)$.

Then the operator
$$
    \mathcal K h(\vartheta)
    :=
    \int_0^{2\pi}\int_0^1
    K(\rho,\vartheta,\eta)h(\eta)
    \,d\rho\,d\eta
$$
defines a bounded linear map
$$
    \mathcal K:
    L^\infty(\mathbb T)
    \longrightarrow
    C^{0,\nu}(\mathbb T).
$$
More precisely,
$$
    \|\mathcal Kh\|_{C^{0,\nu}}
    \leq
    C C_0
    \left(
        \|g_1\|_{L^1(0,1)}
        +
        \|g_2\|_{L^1(0,1)}
    \right)
    \|h\|_{L^\infty},
$$
where $C>0$ depends only on $\nu$.
\end{lem}

Next, we give the computations of some integrals that appear during the proofs.

\begin{lem} \label{integral}
    Let $a, b \in \C$ not both zero. Then,
   $$\frac{1}{2\pi}\int_{0}^{2\pi} \log \left|a - b e^{i \tau}\right| \, d\tau = \max \left\{\log|a|, \log|b|\right\}.$$
\end{lem}
\begin{proof}
    For simplicity, set
   $$I := \frac{1}{2\pi}\int_{0}^{2\pi} \log \left|a - b e^{i \tau}\right| \, d\tau.$$
    Recall that, for every $z \in \C$ with $|z| < 1$, the following expansion holds true:
    \begin{equation} \label{expansion_log}
        \log|1-z| = - \sum_{k = 1}^{\infty} \frac{1}{k} \operatorname{Re}\left[z^k\right].
    \end{equation}
    Case $|a| > |b|$. By collecting $a$ inside the logarithm and using \eqref{expansion_log}, we obtain 
    \begin{align*}
        I & = \frac{1}{2\pi} \int_{0}^{2\pi} \log\left|a \left(1 - \frac{b}{a} e^{i \tau} \right)\right| \, d\tau = \log|a| + \frac{1}{2\pi} \int_{0}^{2\pi} \log\left|1 - \frac{b}{a} e^{i \tau}\right| \, d \tau \\
        & = \log |a| - \frac{1}{2\pi} \sum_{k = 1}^{\infty} \frac{1}{k}  \int_{0}^{2\pi} \operatorname{Re}\left[\left(\frac{b}{a}\right)^ke^{i k \tau}\right] \, d\tau =  \log|a|,
    \end{align*}
    since 
$$\int_{0}^{2\pi} \operatorname{Re}\left[\left(\frac{b}{a}\right)^ke^{i k \tau}\right] \, d\tau = 0  \,\,\,\text{   for    }\,\,\, k\ge 1.$$
    Case $|a| < |b|$. By collecting $-be^{i \tau}$ inside the logarithm and using \eqref{expansion_log}, we obtain
    \begin{align*}
        I & = \frac{1}{2\pi} \int_{0}^{2\pi} \log \left|- b e^{i \tau} \left(1 - \frac{a}{b} e^{-i\tau} \right)\right| \, d\tau = \log|b| + \frac{1}{2\pi} \int_{0}^{2\pi} \log\left| 1 - \frac{a}{b} e^{-i\tau} \right| \, d\tau \\
        & = \log|b| - \frac{1}{2\pi} \sum_{k = 1}^{\infty} \frac{1}{k}  \int_{0}^{2\pi} \operatorname{Re}\left[\left(\frac{a}{b}\right)^ke^{-ik\tau}\right] \, d\tau = \log |b| .
    \end{align*}
    Case $|a| = |b|$. In this scenario, there is $\vartheta \in [0, 2\pi)$ such that $b = a e^{i \vartheta}$. Thus, factoring out $a$ inside the logarithm, performing the change of variable $\sigma = \tau + \vartheta$, and using the $2\pi$-periodicity of the integrand, we obtain
    \begin{align*}
        I &= \frac{1}{2\pi} \int_{0}^{2\pi} \log\left|a\left(1 - e^{i (\tau + \vartheta)}\right)\right| \, d \tau = \log|a| + \frac{1}{2\pi} \int_{0}^{2\pi} \log\left|1 - e^{i (\tau+\vartheta)}\right| \, d\tau \\
        & = \log|a| + \frac{1}{2\pi} \int_{\vartheta}^{2\pi + \vartheta} \log\left|1 - e^{i \sigma}\right| \, d\sigma = \log|a| + \frac{1}{2\pi} \int_{0}^{2\pi} \log\left|1 - e^{i \sigma}\right| \, d\sigma.
    \end{align*}
    But, by Euler's identity,
   $$\left|1 - e^{i \sigma}\right| = \left|e^{i \frac{\sigma}{2}}\left(e^{-i\frac{\sigma}{2}} - e^{i \frac{\sigma}{2}}\right)\right| = 2 \left|\sin\left(\frac{\sigma}{2}\right)\right|, \qquad \sigma \in (0, 2\pi),$$
    and therefore
    \begin{align*}
        \frac{1}{2\pi} \int_{0}^{2\pi} \log\left|1 - e^{i \sigma}\right| \, d\sigma & = \frac{1}{2\pi} \int_{0}^{2\pi} \log\left| 2 \sin\left(\frac{\sigma}{2}\right)\right| \, d\sigma = \log 2 + \frac{1}{\pi} \int_{0}^{\pi} \log \left(\sin \sigma\right) d \sigma \\
        & = \log 2 + \frac{1}{\pi} \left(-\pi \log 2 \right) = 0.
    \end{align*}
    Consequently,
   $$I = \log|a|.$$
    The cases $a = 0$ and $b = 0$ are trivial.
\end{proof}

\begin{lem} \label{integral_disk}
    Let $a \in \C$ with $|a| \leq 1$. Then,
   $$\int_{\mathbb{D}} \log\left|z-a\right| dA(z) = \frac{\pi}{2} \left(|a|^2-1\right).$$
\end{lem}
\begin{proof}
    Passing to radial coordinates and applying Fubini's theorem, it results
    \begin{align} \label{sum_I_j_a}
        \int_{\D} \log\left|z-a\right| \, dA(z) & = \int_{0}^{2\pi} \int_{0}^{1} \log\left|re^{i \vartheta} - a\right| \, r \, dr \, d\vartheta \nonumber \\
        & = \int_{0}^{1} \int_{0}^{2\pi} \log\left|a - re^{i \vartheta}\right| \, d\vartheta \, r \, dr \nonumber\\
        & = I_1(a) + I_2(a),
    \end{align}
    where
   $$I_1(a) := \int_{0}^{|a|} \int_{0}^{2\pi} \log\left|a - re^{i \vartheta}\right| \, d\vartheta \, r \, dr, \qquad I_2(a) := \int_{|a|}^{1} \int_{0}^{2\pi} \log\left|a - re^{i \vartheta}\right| \, d\vartheta \, r \, dr.$$
    Thanks to Lemma \ref{integral}, one has
    \begin{align} \label{I_1_a}
        I_1(a) = 2 \pi \log |a| \int_{0}^{|a|} r \, dr = 2 \pi \log|a| \left[\frac{r^2}{2}\right]_{r = 0}^{r = |a|} =  \pi |a|^2 \log |a|,
    \end{align}
    while
    \begin{align} \label{I_2_a}
    I_2(a) & = 2\pi \int_{|a|}^1 r \log r \, dr = 2\pi \left[\frac{r^2}{2} \left(\log r - \frac{1}{2}\right)\right]_{r = |a|}^{r = 1} \nonumber \\
    & = - \frac{\pi}{2} - \pi |a|^2\left(\log |a| - \frac{1}{2}\right) = - \pi |a|^2 \log |a| + \frac{\pi}{2} \left(|a|^2 - 1\right).
    \end{align}
    The conclusion follows by \eqref{I_1_a} and \eqref{I_2_a} into \eqref{sum_I_j_a}.
\end{proof}

\begin{lem} \label{integral_cos_den}
    Let $a, b \in \R$ with $a > |b|$. Then, 
   $$\int_{0}^{2\pi} \frac{d \tau}{a-b\cos \tau} = \frac{2\pi}{\sqrt{a^2-b^2}}.$$
\end{lem}
\begin{proof}
    Using the change of variables
   $$t = \tan\left(\frac{\tau}{2}\right),$$
    we have
   $$\cos\tau = \frac{1-t^2}{1+t^2}, \qquad d\tau = \frac{2}{1+t^2} \, dt.$$
    Hence, 
   $$a-b\cos \tau = a - b \frac{1-t^2}{1+t^2} = \frac{(a-b)+(a+b)t^2}{1+t^2}. $$
    Therefore, 
    \begin{align*}
        \int_{0}^{2\pi} \frac{d \tau}{a-b\cos \tau} &= 2 \int_{-\infty}^{\infty} \frac{dt}{(a-b)+(a+b)t^2} = \frac{2\pi}{\sqrt{(a-b)(a+b)}} = \frac{2\pi}{\sqrt{a^2-b^2}}.
    \end{align*}
\end{proof}

\begin{lem} \label{integral_double}
    One has
   $$\int_{0}^{2\pi} \int_{0}^{1} \frac{1 - r \cos \tau}{1 + r^2 - 2r \cos \tau} \, r \, dr \, d\tau = \pi.$$
\end{lem}
\begin{proof}
    First we notice that 
   $$1-r\cos \tau = \frac{1}{2}\left(1+r^2-2r\cos \tau \right) + \frac{1}{2}\left(1-r^2\right).$$
    Consequently, 
   $$\frac{1 - r \cos \tau}{1 + r^2 - 2r \cos \tau} = \frac{1}{2} + \frac{1-r^2}{2\left(1 + r^2 - 2r \cos \tau\right)}.$$
    Integrating with respect to $\tau$ we get 
   $$\int_{0}^{2\pi} \frac{1 - r \cos \tau}{1 + r^2 - 2r \cos \tau} \, d\tau = \pi + \frac{1-r^2}{2} \int_{0}^{2\pi} \frac{d \tau}{1 + r^2 - 2r \cos \tau}.$$
    But, by Lemma \ref{integral_cos_den}, 
    \begin{align*}
        \int_{0}^{2\pi} \frac{d \tau}{1 + r^2 - 2r \cos \tau} & = \frac{2\pi}{\sqrt{\left(1+r^2\right)^2-4r^2}} = \frac{2\pi}{\sqrt{\left(1-r^2\right)^2}} = \frac{2\pi}{1-r^2,}
    \end{align*}
    because $r \in [0, 1)$. Thus,
   $$\int_{0}^{2\pi} \frac{1 - r \cos \tau}{1 + r^2 - 2r \cos \tau} \, d\tau = 2\pi.$$
    Finally,
   $$\int_{0}^{2\pi} \int_{0}^{1} \frac{1 - r \cos \tau}{1 + r^2 - 2r \cos \tau} \, r \, dr \, d\tau = 2\pi\int_{0}^{1} r \, dr = \pi.$$
\end{proof}
\begin{lem}
\label{lem:C1}
Let
$\mathbf m\geq 2$.
Then
\[
    \sum_{j=2}^{\mathbf m}
    \frac{1}{1- e^{i\frac{2\pi(j-1)}{\mathbf m}}}
    =
    \frac{\mathbf m-1}{2}.
\]
\end{lem}

\begin{proof}
Since
\[
    \frac{1}{1-z}
    +
    \frac{1}{1-\overline z}
    =
    1
\]
for every $z\in\mathbb C\setminus\{1\}$ with $|z|=1$, the terms
corresponding to conjugate roots of unity add up to $1$. Pairing the
$\mathbf m-1$ nontrivial roots of unity, and treating the root $-1$
separately when $\mathbf m$ is even, gives
\[
    \sum_{j=2}^{\mathbf m}
    \frac{1}{1-e^{i\frac{2\pi(j-1)}{\mathbf m}}}
    =
    \frac{\mathbf m-1}{2}.
\]
\end{proof}

	\bibliography{references}	
	\bibliographystyle{plain}

\end{document}